\documentclass[12pt,oneside,english]{amsart}
\usepackage[T1]{fontenc}
\usepackage[latin9]{inputenc}
\usepackage[a4paper]{geometry}
\usepackage{babel}
\usepackage{prettyref}
\usepackage{mathrsfs}
\usepackage{mathtools}
\usepackage{amstext}
\usepackage{amsthm}
\usepackage{amssymb}
\usepackage{wasysym}
\usepackage[unicode=true,
 bookmarks=true,bookmarksnumbered=false,bookmarksopen=false,
 breaklinks=false,pdfborder={0 0 1},backref=false,colorlinks=false]
 {hyperref}
\usepackage{breakurl}

\makeatletter
\numberwithin{equation}{section}
\numberwithin{figure}{section}
\theoremstyle{plain}
\newtheorem{thm}{\protect\theoremname}
  \theoremstyle{plain}
  \newtheorem{prop}[thm]{\protect\propositionname}
  \theoremstyle{plain}
  \newtheorem{lem}[thm]{\protect\lemmaname}
  \theoremstyle{definition}
  \newtheorem{example}[thm]{\protect\examplename}
  \theoremstyle{definition}
  \newtheorem{defn}[thm]{\protect\definitionname}
  \theoremstyle{remark}
  \newtheorem{rem}[thm]{\protect\remarkname}

\@ifundefined{rangeHsb}{\usepackage{xcolor}}{}
\usepackage{babel}
\usepackage{inputenc}

\makeatother

  \providecommand{\definitionname}{Definition}
  \providecommand{\examplename}{Example}
  \providecommand{\lemmaname}{Lemma}
  \providecommand{\propositionname}{Proposition}
  \providecommand{\remarkname}{Remark}
\providecommand{\theoremname}{Theorem}

\begin{document}

\title[Spectrum and abnormals]{Spectrum and abnormals in sub-Riemannian geometry: the 4D quasi-contact
case }

\author{Nikhil Savale}

\thanks{The author is partially supported by the DFG funded project CRC/TRR
191.}

\address{Universität zu Köln, Mathematisches Institut, Weyertal 86-90, 50931
Köln, Germany}

\email{nsavale@math.uni-koeln.de}
\begin{abstract}
We prove several relations between spectrum and dynamics including
wave trace expansion, sharp/improved Weyl laws, propagation of singularities
and quantum ergodicity for the sub-Riemannian (sR) Laplacian in the
four dimensional quasi-contact case. A key role in all results is
played by the presence of abnormal geodesics and represents the first
such appearance of these in sub-Riemannian spectral geometry. 
\end{abstract}

\maketitle
\tableofcontents{}

\section{Introduction}

Sub-Riemannian (sR) geometry is the study of metric subbundles $\left(E\subset TX,g^{E}\right)$
inside the tangent bundle of a manifold $X$ that are bracket generating;
we refer to \cite{AgrachevBarilariBoscain2019,Bellaiche-book-1996,Montgomery-book}
for some textbook references on the subject. The geometric/dynamical
significance of the bracket-generating hypothesis is via the theorem
of Chow-Rashevky on connectivity of points by horizontal curves. With
the metric assigning lengths to horizontal curves, the manifold acquires
a natural metric space structure. A geodesic is a horizontal length
minimizing path. A peculiar feature of sub-Riemannian geometry, unlike
Riemannian geometry, is that there are geodesics which do not satisfy
any variational equation or equivalently are not projections of the
corresponding Hamiltonian flow \cite[Ch. 1]{Bismut84}, \cite{Montgomery94};
these geodesics are \textit{abnormal}.

The choice of an auxiliary density $\mu_{X}$ allows for the definition
of a sub-Riemannian Laplacian on the manifold which in general is
not an elliptic operator. The analytic significance of the bracket-generating
hypothesis is then via the classical theorem of Hörmander \cite{Hormander1967}
saying that the sub-Riemannian Laplacian is hypoelliptic and as such
has a discrete spectrum of real eigenvalues. Classical Riemannian
results on spectral asymptotics where geodesic flow plays a role such
as Weyl's law \cite{Avakumovic56,Hormander68,Levitan52,Minakshisundaram-Pleijel49},
wave trace trace formulas \cite{Chazarain74,CdV73,Duistermaat-Guillemin},
propagation of singularities \cite{Duistermaat-Hormander} and quantum
ergodicity \cite{CdV-QE,Schnirelman-QE,Zelditch-QE} remain largely
unexplored in sub-Riemannian geometry. It is in particular an interesting
question whether abnormal geodesics would play a role in sR spectral
geometry. The purpose of this article is to positively answer this
question in one of the simplest cases where abnormals exist, namely
the four dimensional quasi-contact case.

Let us now state our results more precisely. Let $X^{4}$ be a smooth,
compact oriented four dimensional manifold. A nowhere vanishing one
form $a\in\Omega^{1}\left(X\right)$ is called quasi-contact if the
restriction $\textrm{rk}\left.da\right|_{E}=2$ is of maximal rank,
where $E\coloneqq\textrm{ker}a\subset TX$. The three dimensional
distribution $E=\textrm{ker}a\subset TX$ can be shown to be bracket
generating and we equip it with a metric $g^{E}$. The characteristic
line field is defined via $L^{E}=\textrm{ker}\left(a\wedge da\right)\subset E$
and can be seen to only depend on $E=\textrm{ker}a$. It carries a
natural orientation, induced from that of $X$, and hence a positively
oriented unit section $Z\in C^{\infty}\left(L^{E}\right)$. The set
of integral curves of $L^{E}$, also called characteristics, contains
the abnormal geodesics in this case. 

Given an auxiliary volume form $\mu$ on $X$, the sR Laplacian acting
on function is defined via 
\begin{equation}
\Delta_{g^{E},\mu}\coloneqq\left(\nabla^{g^{E}}\right)_{\mu}^{*}\nabla^{g^{E}}:C^{\infty}\left(X\right)\rightarrow C^{\infty}\left(X\right)\label{eq:sR Laplacian}
\end{equation}
 where $\nabla^{g^{E}}:C^{\infty}\left(X\right)\rightarrow C^{\infty}\left(X;E\right),\,\left\langle \nabla^{g^{E}}f,e\right\rangle \coloneqq e\left(f\right),\,\forall e\in E,$
is the sR gradient and the adjoint \prettyref{eq:sR Laplacian} above
is taken with respect to the natural $L^{2}$-inner products coming
from $\mu$. The Laplacian \prettyref{eq:sR Laplacian} is not elliptic
with characteristic variety $\Sigma\subset T^{*}X,\:\Sigma\coloneqq\left\{ \sigma\left(\Delta_{g^{E},\mu}\right)=0\right\} =\mathbb{R}\left[a\right]$
being given by the graph of the one form $a$. However being self-adjoint
of Hörmander type, there is a complete orthonormal basis of $\left\{ \varphi_{j}\right\} _{j=0}^{\infty}$
for $L^{2}\left(X,\mu\right)$ consisting of (real-valued) eigenvectors
for \prettyref{eq:sR Laplacian} $\Delta_{g^{E},\mu}\varphi_{j}=\lambda_{j}\varphi_{j}$,
$0\leq\lambda_{0}\leq\lambda_{1}\leq\ldots$. 

Our first result on spectral asymptotics is then the following sharp
Weyl law for the counting function $N\left(\lambda\right)$ of the
number of eigenvalues of the sR Laplacian $\Delta_{g^{E},\mu}$ of
size at most $\lambda$. Below $\mu_{\textrm{Popp}}$, $\nu_{\textrm{Popp}}=\frac{1}{\left(\int_{X}\mu_{\textrm{Popp}}\right)}\mu_{\textrm{Popp}}$
and $a_{g^{E}}$ denote the unnormalized, normalized Popp volume and
Popp one form respectively (see Section \prettyref{subsec:Quasi-contact-case}).
\begin{thm}
\label{thm:Sharp Weyl law}The Weyl counting function $N\left(\lambda\right)$
for the sR Laplacian $\Delta_{g^{E},\mu}$ in the 4D quasi-contact
case satisfies the sharp asymptotics 
\begin{equation}
N\left(\lambda\right)=\frac{1}{24\pi}\lambda^{5/2}\int_{X}\mu_{\textrm{Popp}}+O\left(\lambda^{2}\right)\label{eq:sharp Weyl law}
\end{equation}
Assuming the union of closed integrals curves of $L^{E}$ to be of
measure zero, one further has 
\begin{equation}
N\left(\lambda\right)=\frac{1}{24\pi}\lambda^{5/2}\int_{X}\mu_{\textrm{Popp}}+o\left(\lambda^{2}\right).\label{eq:improved Weyl law}
\end{equation}
\end{thm}
By a usual Tauberian argument, the sharp Weyl law \prettyref{eq:sharp Weyl law}
above is proved using small time asymptotics of the wave trace. Below
we denote by $T_{\textrm{abnormal}}^{E}$ the length of the shortest
closed integral curve of $L^{E}$. The (signed) lengths of normal
closed geodesics are by definition the periods of closed integral
curves for the Hamilton flow of $\sigma\left(\Delta_{g^{E},\mu}\right)^{1/2}$
away from $\Sigma$. We denote the set of such by $\mathscr{L}_{\textrm{normal}}$.
\begin{thm}
\label{thm:Poisson relation}The singular support of the wave trace
satisfies
\begin{equation}
\textrm{sing spt}\left(\textrm{tr }e^{it\sqrt{\Delta_{g^{E},\mu}}}\right)\subset\left\{ 0\right\} \cup\left(-\infty,-T_{\textrm{abnormal}}^{E}\right]\cup\left[T_{\textrm{abnormal}}^{E},\infty\right)\cup\mathscr{L}_{\textrm{normal}}.\label{eq: sing spt wave trace}
\end{equation}
Furthermore, the singularity at zero is described by the small time
asymptotics
\begin{align}
\textrm{tr }e^{it\sqrt{\Delta_{g^{E},\mu}}} & =\sum_{j=0}^{N}c_{j,0}\left(t+i0\right)^{j-5}+\sum_{j=0}^{N}c_{j,1}\left(t+i0\right)^{j-3}\ln\left(t+i0\right)\nonumber \\
 & \qquad\qquad+\sum_{j=0}^{N}c_{j,2}t^{j}\ln^{2}\left(t+i0\right)+O\left(t^{N-4}\right),\label{eq:small time wave trace}
\end{align}
 $\forall N\in\mathbb{N}$, as $t\rightarrow0$, in the distributional
sense with leading term $c_{0,0}=\frac{1}{12}\int_{X}\mu_{\textrm{Popp}}$. 
\end{thm}
Note the presence of logarithmic terms in the wave trace expansion
\prettyref{eq:small time wave trace} is unlike on a Riemannian manifold.
The singularities of the wave trace \prettyref{eq: sing spt wave trace}
at (isolated) lengths of non-degenerate normal geodesics in the interval
$\left(-T_{\textrm{abnormal}}^{E},T_{\textrm{abnormal}}^{E}\right)$
are described by the usual Duistermaat-Guillemin trace formula. Beyond
this interval there is a possible density of lengths of $\mathscr{L}_{\textrm{normal}}$
inside $\left(-\infty,-T_{\textrm{abnormal}}^{E}\right]\cup\left[T_{\textrm{abnormal}}^{E},\infty\right)$
for albeit non-degenerate characteristics, caused by closed Hamilton
trajectories that approach the characteristic variety (see Prop. \prettyref{prop:(Density-of-periods)}),
making the description of these singularities less tractable.

The large time wave trace formula \prettyref{eq: sing spt wave trace}
is in turn related to the propagation of singularities for the corresponding
wave equation. The classical theorem of \cite{Duistermaat-Hormander}
describes the propagation of singularities for the half wave equation
outside the characteristic variety $\Sigma$. To describe the propagation
of singularities on $\Sigma$ we consider the blowup $\left[T^{*}X;\Sigma\right]$
of the cotangent bundle along the characteristic variety with corresponding
blow-down map $\beta:\left[T^{*}X;\Sigma\right]\rightarrow T^{*}X$.
This is a manifold with boundary $\partial\left[T^{*}X;\Sigma\right]=SN\Sigma$
being identified with the spherical normal bundle of $\Sigma$ which
in turn carries an $\mathbb{R}_{+}$ action extending the one on its
interior. The boundary $SN\Sigma$ is equipped with a natural homogeneous
and $\beta$ fiber preserving circle action, by rotation of its symplectic
directions, and corresponding generator $R_{0}=\left.\frac{d}{d\theta}\left(e^{i\theta}.p\right)\right|_{\theta=0}$.
In Section \prettyref{subsec:Quasi-contact-case} we shall define
a homogeneous of degree zero section $\hat{Z}\in C^{\infty}\left(TSN\Sigma/\mathbb{R}\left[R_{0}\right]\right)$
and a refined circle invariant conic characteristic wave-front set
$WF_{\Sigma}\left(u\right)\subset SN\Sigma$ associated to any distribution
$u\in C^{-\infty}\left(X\right)$. These can be equivalently thought
of as a homogeneous of degree zero vector field on and conic subset
of the quotient $SN\Sigma/S^{1}$ by the circle action. They project
\begin{align*}
\left(\pi\circ\beta\right)_{*}\hat{Z} & \in L^{E}\\
\beta\left(WF_{\Sigma}\left(u\right)\right) & =WF\left(u\right)\cap\Sigma
\end{align*}
onto the characteristic line and intersection of the wavefront set
of $u$ with $\Sigma$ respectively. The interval in \prettyref{eq: sing spt wave trace}
is furthermore related to the set of closed periods of the vector
field $\hat{Z}$ (see Prop. \prettyref{prop:(Density-of-periods)}). 

We now have the following propagation of singularities.
\begin{thm}
\label{thm:propagation on sing.}For any $u\in C^{-\infty}\left(X\right)$,
the characteristic wavefront set satisfies 
\[
WF_{\Sigma}\left(e^{it\sqrt{\Delta_{g^{E},\mu}}}u\right)=e^{t\hat{Z}}\left[WF_{\Sigma}\left(u\right)\right].
\]
\end{thm}
Our final result is quantum ergodicity for the sR Laplacian. The line
field $L^{E}$ is said to be ergodic if any union of closed integral
curves of $L^{E}$ is of zero or full measure. The ergodicity of the
vector field $\hat{Z}$ is a stronger assumption implying the ergodicity
of $L^{E}$. We now have the following.
\begin{thm}
\label{thm:QE theorem}Assume that $\hat{Z}$ is ergodic or $L^{E}$
is ergodic and $L_{Z}\mu_{\textrm{Popp}}=0$. Then one has quantum
ergodicity for $\Delta_{g^{E},\mu}$: there exists a density one subsequence
$\left\{ j_{k}\right\} _{k=0}^{\infty}\subset\mathbb{N}_{0}$ such
that 
\[
\left\langle B\varphi_{j_{k}},\varphi_{j_{k}}\right\rangle \rightarrow\frac{1}{2}\int d\nu_{\textrm{Popp}}\left[b\left(x,a_{g^{E}}\left(x\right)\right)+b\left(x,-a_{g^{E}}\left(x\right)\right)\right].
\]
as $j_{k}\rightarrow\infty$, for each $B\in\Psi_{\textrm{cl}}^{0}\left(X\right)$,
with homogeneous principal symbol $b=\sigma\left(B\right)\in C^{\infty}\left(T^{*}X\right)$.
In particular, the eigenfunctions get uniformly distributed $\left|\varphi_{j_{k}}\right|^{2}\mu\rightharpoonup\nu_{\textrm{Popp}}$
as $j_{k}\rightarrow\infty$.
\end{thm}
We note the role played by characteristics or integral curves of $L^{E}$
in all the results above. Under the natural projection, these correspond
to isotropic directions along $\Sigma$ and thereafter with abnormal
geodesics via their microlocal characterization by Hsu \cite{Hsu92}.
Our results are restricted to dimension four as they rely on a normal
form that is less workable in higher dimensions. Moreover, there is
general lack of understanding of strictly abnormal geodesics in sub-Riemannian
geometry; it is for instance outstanding whether they are necessarily
smooth \cite[Ch. 10]{Montgomery-book}, \cite{Agrachev2014,Hakavuori-LeDonne2016,LeDonne-Leonardi-Monti-Vittone2013}.

The leading term in the Weyl law \prettyref{thm:Sharp Weyl law} has
been long known \cite{Menikoff-SjostrandI,Menikoff-SjostrandII,Metivier-hypspectralfunction},
the improvement here is in the two remainders. The only previous work
treating a sharp Weyl law based on a wave trace expansion of a hypoelliptic
operator is \cite{Melrose-hypoelliptic}. In the sub-Riemannian context
\cite{Melrose-hypoelliptic} however only specializes to the three
dimensional contact case; therein the characteristic variety $\Sigma$
was assumed to be symplectic which is not the case here. There is
one isotropic direction along $\Sigma$ that projects onto $L^{E}$.
A general result for propagation of singularities of hypoelliptic
operators exists in the literature \cite{Lascars-propagation-DUKE}.
Our result \prettyref{thm:propagation on sing.} based on the characteristic
wavefront set is a refinement of the aforementioned in the present
context. Recently, quantum ergodicity for the sub-Riemannian Laplacian
was established in the three dimensional contact case \cite{Colin-de-Verdiere-Hillairet-TrelatI}
and as such was the first result on quantum ergodicity for a hypoelliptic
operator. Our technique here while partly borrowing from \cite{Colin-de-Verdiere-Hillairet-TrelatI}
also overcomes significant additional difficulties. In particular
our proof of \prettyref{thm:QE theorem} requires the use of a more
exotic second microlocal pseudo-differential calculus near the characteristic
variety. Finally unlike here there are no abnormal geodesics in the
three dimensional contact case.

The results here also tie in with the authors previous work \cite{Savale2017-Koszul,Savale-Gutzwiller}
wherein a trace formula was proved for the semiclassical (magnetic)
Dirac operator on a metric contact manifold involving closed Reeb
orbits; semiclassical analogs of quasi-contact characteristics of
$L^{E}$. However there are also significant differences; there firstly
seems to be at present no general analog of the Dirac operator, with
good spectral properties, in sub-Riemannian geometry (see for example
\cite{Hasselmann-thesis,Kath-Ungermann2015}). This forces us to work
with the non-(pseudo)differential square root $\sqrt{\Delta_{g^{E},\mu}}$
and understand it in a more exotic pseudo-differential calculus. Secondly,
the trace considered in \cite{Savale-Gutzwiller,Savale2017-Koszul}
was microlocalized on an $\sqrt{h}$ scale near the characteristic
variety, using the intrinsic semiclassical parameter, cutting off
the Hamilton trajectories away from it. This microlocalized trace
formula subsequently does not see the dense accumulation of the Hamilton
periods \prettyref{eq: sing spt wave trace}, involves contributions
only from the Reeb orbits and works in higher dimension. 

The paper is organized as follows. In \prettyref{sec:sub-Riemannian-geometry}
we begin with some preliminaries on sub-Riemannian geometry including
certain specific features of the four dimensional quasi-contact case
in Section \prettyref{subsec:Quasi-contact-case}. In \prettyref{sec:Hermite-Calculus}
we develop the relevant second microlocal Hermite-Landau calculus
on Euclidean space necessary for the proofs. In Section \prettyref{sec:Birkhoff-normal-forms}
we derive normal forms for the sR Laplacian. The normal form of \prettyref{subsec:Normal-form-near-Sigma}
is then used in Section \prettyref{sec:Global-calculus} to develop
a global Hermite-Landau calculus on a quasi-contact manifold. The
calculus is then used to prove the propagation theorem \prettyref{thm:propagation on sing.}
in \prettyref{subsec:Egorov-and-propagation} and construct a parametrix
for the wave operator in \prettyref{subsec:Parametrices}. The parametrix
gives a proof of the Weyl laws in \ref{thm:Sharp Weyl law} via the
wave trace expansion \prettyref{thm:Poisson relation} in Section
\prettyref{sec:Poisson-relations}. In the final \prettyref{sec:Quantum-ergodicity}
the calculus is used to prove the quantum ergodicity \prettyref{thm:QE theorem}.

\section{sub-Riemannian geometry\label{sec:sub-Riemannian-geometry}}

Sub-Riemannian (sR) geometry is the study of (metric-)distributions
in smooth manifolds. More precisely, a sub-Riemannian manifold is
a triple $\left(X^{n},E^{k}\subset TX,g^{E}\right)$ consisting of
an $n$-dimensional manifold $X$ with and a metric subbundle $\left(E,g^{E}\right)$
of rank $k$ inside its tangent space. This sub-bundle is assumed
to be\textit{ bracket generating:} sections of $E$ generate all sections
of $TX$ under the Lie bracket. The metric $g^{E}$ allows for the
definition of a length function $l\left(\gamma\right)\coloneqq\int_{0}^{1}\left|\dot{\gamma}\right|dt$
on the space of horizontal paths of Sobolev regularity one 
\[
\Omega_{E}\left(x_{0},x_{1}\right)\coloneqq\left\{ \gamma\in H^{1}\left(\left[0,1\right];X\right)|\gamma\left(0\right)=x_{0},\,\gamma\left(1\right)=x_{1},\,\dot{\gamma}\left(t\right)\in E_{\gamma\left(t\right)}\textrm{ a.e.}\right\} 
\]
connecting any two points $x_{0},x_{1}\in X$. This in turn defines
the distance function between these points via
\begin{equation}
d^{E}\left(x_{0},x_{1}\right)\coloneqq\inf_{\gamma\in\Omega_{E}\left(x_{0},x_{1}\right)}l\left(\gamma\right).\label{eq:subRiemannian distance}
\end{equation}
The theorem of Chow-Rashevsky \cite[Thm 1.6.2]{Montgomery-book} gives
the existence of a horizontal path connecting $x_{0},x_{1}$. This
shows that the distance function above is finite and defines a metric
space $\left(X,d^{E}\right)$.

Using the bracket generating condition for $E$, the canonical flag
may be defined 
\begin{equation}
\underbrace{E_{0}\left(x\right)}_{=\left\{ 0\right\} }\subset\underbrace{E_{1}\left(x\right)}_{=E}\subset\ldots\subset\subsetneq E_{r\left(x\right)}\left(x\right)=TX\label{eq: canonical flag}
\end{equation}
inductively via $E_{j}=E+\left[E,E_{j-1}\right]$, $j\geq2$, as a
flag of vector subspaces of $TX$ at any point $x\in X$ . Here $r\left(x\right)$
is the smallest number such that $E_{r\left(x\right)}=TX$ and called
the degree of nonholonomy or step of the distribution $E$ at $x$.
The dual canonical flag is then 

\begin{equation}
T^{*}X=\Sigma_{0}\left(x\right)\supset\underbrace{\Sigma_{1}\left(x\right)}_{\eqqcolon\Sigma}\supset\ldots\supset\underbrace{\Sigma_{r\left(x\right)}\left(x\right)}_{=\left\{ 0\right\} },\label{eq:dual canonical flag}
\end{equation}
$\Sigma_{j}\left(x\right)=E_{j}^{\perp}\coloneqq\textrm{ker }\left[T^{*}X\rightarrow E_{j}^{*}\right]$,
$1\leq j\leq r\left(x\right)$. We further define the growth and weight
vectors at the point $x\in X$ via 
\begin{align}
k^{E}\left(x\right)\coloneqq & \left(\underbrace{k_{0}^{E}}_{\coloneqq0},\underbrace{k_{1}^{E}}_{=\textrm{dim }E_{1}},\underbrace{k_{2}^{E}}_{=\textrm{dim }E_{2}},\ldots,\underbrace{k_{r}^{E}}_{=n}\right)\label{eq:growth vector}\\
w^{E}\left(x\right)\coloneqq & \left(\underbrace{1,\ldots,1}_{k_{1}^{E}\textrm{ times}},\underbrace{2,\ldots2}_{k_{2}^{E}-k_{1}^{E}\textrm{ times}},\ldots\underbrace{j,\ldots j}_{k_{j}^{E}-k_{j-1}^{E}\textrm{ times}},\ldots,\underbrace{r,\ldots,r}_{k_{r}^{E}-k_{r-1}^{E}\textrm{ times}}\right)\label{eq:weight vector}
\end{align}
respectively. The distribution $E$ is called regular at the point
$x\in X$ if each $k_{j}^{E}$ is a locally constant function near
$x$. The distribution $E$ is said to be equiregular if it is regular
at all points of $X$, in which case each element $E_{j}$ of the
canonical flag \prettyref{eq: canonical flag} is a vector bundle.
Finally we set 
\begin{align*}
Q\left(x\right)\coloneqq & \sum_{j=1}^{r\left(x\right)}j\left(k_{j}^{E}\left(x\right)-k_{j-1}^{E}\left(x\right)\right)\\
= & \sum_{j=1}^{n}w_{j}^{E}\left(x\right)
\end{align*}
whose significance is given by the Mitchell measure theorem \cite[Theorem 2.8.3]{Montgomery-book}:
$Q\left(x\right)$ is the Hausdorff dimension of $\left(X,d^{E}\right)$
at a regular point $x\in X$.

A canonical volume form on $X$ (analogous to the Riemannian volume)
can be defined in the equiregular case. To define this, first note
that any surjection $\pi:V\rightarrow W$ between two vector spaces
allows one to pushforward a metric $g^{V}$ on $V$ to another $\pi_{*}g^{V}$
on $W$. This is simply the metric on $W$ induced via the identification
$W\cong\left(\textrm{ker }\pi\right)^{\perp}\subset V$, with the
metric on $\left(\textrm{ker }\pi\right)^{\perp}$ being the restriction
of $g^{V}$ . Now for each $j$ we define the linear surjection 
\begin{align*}
B_{j}:E^{\otimes j} & \rightarrow E_{j}/E_{j-1}\\
B_{j}\left(e_{1},\ldots e_{j}\right) & \coloneqq\textrm{ad}_{\tilde{e}_{1}}\textrm{ad}_{\tilde{e}_{2}}\ldots\textrm{ad}_{\tilde{e}_{j-1}}\tilde{e}_{j}
\end{align*}
with $\tilde{e}_{j}\in C^{\infty}\left(E\right)$ denoting local sections
extending $e_{j}\in E$. The pushforward metrics are then well defined
on $E_{j}/E_{j-1}$ and hence define canonical volume elements 
\begin{equation}
\det g_{j}^{E}\in\Lambda^{*}\left(E_{j}/E_{j-1}\right)^{*}.\label{eq: flag volume elements}
\end{equation}
The canonical isomorphism of determinant lines
\begin{align}
\bigotimes_{j=1}^{r}\Lambda^{*}\left(E_{j}/E_{j-1}\right)=\Lambda^{*}\left(\bigoplus_{j=1}^{r}E_{j}/E_{j-1}\right)\cong & \Lambda^{*}TX\label{eq: canonical isomorphism det lines}
\end{align}
along with its dual isomorphism to now gives a canonical smooth volume
form 
\begin{equation}
\mu_{\textrm{Popp}}\coloneqq\bigotimes_{j=1}^{r}\det g_{j}^{E}\in\Lambda^{*}\left(T^{*}X\right)\label{eq:Popp measure}
\end{equation}
known as the \textit{Popp volume }form. We remark that although the
definition makes sense in general it only leads to a smooth form in
the equiregular case.

In \prettyref{subsec:Microlocal-Weyl-laws} we shall need the important
notion of a privileged coordinate system. To define this let $U_{1},U_{2},\ldots U_{k}$
be a locally defined set of orthonormal, generating vector fields
near $x\in X$. The $E$-order of a function at the point $x$ is
defined via 
\[
\textrm{ord}_{E,x}\left(f\right)\coloneqq\max\left\{ s|\left(U_{1}^{s_{1}}\ldots U_{k}^{s_{k}}f\right)\left(x\right)=0,\:\forall\left(s_{1},\ldots,s_{k}\right)\in\mathbb{N}_{0}^{k},\,\sum_{j=1}^{k}s_{j}=s\right\} .
\]
Similarly the $E-$order of a differential operator $P$ at the point
$x$ is defined via 
\[
\textrm{ord}_{E,x}\left(P\right)\coloneqq\max\left\{ s|\textrm{ord}_{E,x}\left(f\right)\geq s'\Longrightarrow\textrm{ord}_{E,x}\left(Pf\right)\geq s+s'\right\} .
\]
It is clear from this definition that the defining vector fields $U_{j}$
each have $E$-order at least $-1$. A coordinate system centered
at $x$ is said to be \textit{privileged} if: the set $\frac{\partial}{\partial x_{1}},\frac{\partial}{\partial x_{2}},\ldots,\frac{\partial}{\partial x_{k_{j}^{E}}}$
forms a basis of $E_{j}\left(x\right)$ for each $j$ and furthermore
each $x_{j}$ has $E$-order $w_{j}^{E}\left(x\right)$ at $x$. The
order of the coordinate vector field $\frac{\partial}{\partial x_{j}}$
is then easily computed to be $-w_{j}^{E}\left(x\right)$. There exists
a privileged coordinate system at centered at each point of $X$ (see
\cite{Bellaiche-book-1996} pg. 30). Next define the privileged coordinate
dilation $\delta_{\varepsilon}:\mathbb{R}^{n}\rightarrow\mathbb{R}^{n}$,
$\forall\varepsilon>0$, using the weight vector \prettyref{eq:weight vector}
via $\delta_{\varepsilon}\left(x_{1},\ldots,x_{n}\right)=\left(\varepsilon^{w_{1}}x_{1},\ldots,\varepsilon^{w_{n}}x_{n}\right)$.
A differential operator $P$ is said to be homogeneous of $E$-order
$s$ if $\left(\delta_{\varepsilon}\right)_{*}P=\varepsilon^{s}P$.
We may now Taylor expand each defining vector field in terms of homogeneous
degrees 
\begin{equation}
\left(\delta_{\varepsilon}\right)_{*}U_{j}=\varepsilon^{-1}\hat{U}_{j}^{\left(-1\right)}+\hat{U}_{j}^{\left(0\right)}+\varepsilon\hat{U}_{j}^{\left(1\right)}+\ldots,\label{eq:priv cord. exp v. field}
\end{equation}
, where each $\hat{U}_{j}^{\left(s\right)}$ is an $\varepsilon$-independent
vector field with polynomial coefficients. The \textit{nilpotentization}
$\left(\hat{X},\hat{E},\hat{g}^{E}\right)$of the sR structure at
$x\in X$ is now defined via $\hat{X}=\mathbb{R}^{n}$, $\hat{E}\coloneqq\mathbb{R}\left[\hat{U}_{1}^{\left(-1\right)},\ldots,\hat{U}_{k}^{\left(-1\right)}\right]$
and where the metric $\hat{g}^{E}$ makes $\left\{ \hat{U}_{j}^{\left(-1\right)}\right\} _{j=1}^{k}$
orthonormal. For any smooth volume form $\mu$ on $X$, one may similarly
define its nilpotentization $\hat{\mu}=\mu_{0}$ at $x$ as the leading
order part in its expansion under the privileged coordinate dilation
\begin{equation}
\delta_{\varepsilon}^{*}\mu=\varepsilon^{Q\left(x\right)}\left[\mu_{0}+\varepsilon\mu_{1}+\varepsilon^{2}\mu_{2}+\ldots\right].\label{eq: prov cord exp. measure}
\end{equation}
 The nilpotentizations of the sR structure and the volume can be shown
to be independent of the choice of privileged coordinates upto sR
isometry (\cite{Bellaiche-book-1996} Ch. 5).

At a regular point, an invariant definition of the nilpotentizations
maybe given. First the sR structure defines a nilpotent Lie algebra
at $x$ via
\begin{equation}
\mathfrak{g}_{x}\coloneqq\left(E_{1}\right)_{x}\oplus\left(E_{2}/E_{1}\right)_{x}\oplus\ldots\oplus\left(E_{r}/E_{r-1}\right)_{x}\label{eq:nilpotent tangent space}
\end{equation}
with the Lie bracket of vector fields inducing an anti-linear map
$\left[.,.\right]:\mathfrak{g}_{x}\otimes\mathfrak{g}_{x}\rightarrow\mathfrak{g}_{x}$.
The algebra is clearly graded with its $j$th graded component $\left(\mathfrak{g}_{x}\right)_{j}\coloneqq\left(E_{j}/E_{j-1}\right)_{x}$
and the bracket preserving the grading $\left[\left(\mathfrak{g}_{x}\right)_{i},\left(\mathfrak{g}_{x}\right)_{j}\right]\subset\left(\mathfrak{g}_{x}\right)_{i+j}$.
Associated to the nilpotent Lie algebra $\mathfrak{g}$ is a unique
simply connected Lie group $G$ with the exponential map giving a
diffeomorphism $\exp:\mathfrak{g}\rightarrow G$. We define the \textit{nilpotentization}
of the sR structure $\left(\hat{X},\hat{E},\hat{g}^{E}\right)$ at
$x$ to be $\hat{X}\coloneqq G$ with the metric distribution $\hat{E},\hat{g}^{E}$
obtained via left translation. Given any volume form $\mu$ on $X$,
the canonical identification $\varLambda^{n}\mathfrak{g}_{x}=\varLambda^{n}\left[\left(E_{1}\right)_{x}\oplus\left(E_{2}/E_{1}\right)_{x}\oplus\ldots\oplus\left(E_{r}/E_{r-1}\right)_{x}\right]\cong\varLambda^{n}T_{x}X$
allows for a definition of the nilpotentization $\hat{\mu}$ of the
volume form $\mu$ on $\hat{X}$.

Sub-Riemannian geometry may be viewed as a limit of Riemannian geometry.
Namely, choose a metric complement $\left(F,g^{F}\right)$ for the
sR distribution satisfying $E\oplus F=TX$ . This gives a one parameter
family of Riemannian metrics 
\begin{equation}
g_{\epsilon}^{TX}=g^{E}\oplus\frac{1}{\epsilon}g^{F}\label{eq: extending metrics}
\end{equation}
which converge $g_{\epsilon}^{TX}\rightarrow g^{E}$ as $\epsilon\rightarrow0$.
We call the above a family of Riemannian metrics \textit{extending}/\textit{taming}
$g^{E}$. The corresponding Riemannian distance then converges $d^{\epsilon}\left(x_{0},x_{1}\right)\rightarrow d^{E}\left(x_{0},x_{1}\right)$
to the sR distance \prettyref{eq:subRiemannian distance} for any
$x_{0},x_{1}\in X$ as $\epsilon\rightarrow0$ (see for eg. \cite[Prop. 4]{Marinescu-Savale18}).

\subsubsection{sR Laplacian}

We now define the sub-Riemannian Laplacian and state some of its first
properties. First given any function $f\in C^{\infty}\left(X\right)$,
define its sR gradient $\nabla^{g^{E}}f\in C^{\infty}\left(E\right)$
by the equation 
\begin{equation}
g^{E}\left\langle \nabla^{g^{E}}f,e\right\rangle \coloneqq e\left(f\right),\quad\forall e\in C^{\infty}\left(E\right).\label{eq:gradient}
\end{equation}
Fixing an arbitrary volume form $\mu$ defines the natural $L^{2}$-
inner products on $C^{\infty}\left(X\right)$ and $C^{\infty}\left(X;E\right)$
giving the adjoint $\left(\nabla^{g}\right)_{\mu}^{*}$ to the gradient
depending on $\mu$. The sR Laplacian is now given by
\begin{equation}
\Delta_{g^{E},\mu}\coloneqq\left(\nabla^{g^{E}}\right)_{\mu}^{*}\circ\nabla^{g^{E}}:C^{\infty}\left(X\right)\rightarrow C^{\infty}\left(X\right).\label{eq:sR Laplacian def.}
\end{equation}
In terms of a local frame $U_{1},\ldots,U_{k}$ for $E$, the above
maybe written 
\begin{equation}
\Delta_{g^{E},\mu}f=-U_{i}\left[g^{E,ij}U_{j}\left(f\right)\right]+g^{E,ij}U_{j}\left(f\right)\left(\left(\nabla^{g^{E}}\right)_{\mu}^{*}U_{i}\right)\label{eq:sR Laplacian in general frame}
\end{equation}
where $g_{ij}^{E}=g^{E}\left(U_{i},U_{j}\right)$ and $g^{E,ij}$
is the inverse metric. If the frame is orthonormal the formula simplifies
to 
\begin{equation}
\Delta_{g^{E},\mu}f=\sum_{j=1}^{k}\left[-U_{j}^{2}\left(f\right)+U_{j}\left(f\right)\left(\left(\nabla^{g^{E}}\right)_{\mu}^{*}U_{j}\right)\right].\label{eq: sR laplacian in orthonormal frame}
\end{equation}
To remark on how the choice of the auxiliary form $\mu$ affects the
Laplacian, let $\mu'=h\mu$ denote another non-vanishing volume form
where $h$ is a positive smooth function on $X$. From the definition
\prettyref{eq:sR Laplacian def.} it now follows easily that one has
the relation 
\[
\Delta_{g^{E},\mu'}=h^{-1}\Delta_{g^{E},\mu}h+h^{-1}\left(\Delta_{g^{E},\mu}h\right).
\]
Thus the two corresponding Laplacians are conjugate modulo a zeroth-order
term. The sR Laplacian $\Delta_{g^{E},\mu}$ is self adjoint with
respect to the obvious inner product $\left\langle f,g\right\rangle =\int_{X}fg\mu$.
The principal symbol of $\Delta_{g^{E},\mu}$ is easily computed to
be the Hamiltonian 
\begin{equation}
\sigma\left(\Delta_{g^{E},\mu}\right)\left(x,\xi\right)=H^{E}\left(x,\xi\right)\coloneqq\left|\left.\xi\right|_{E}\right|^{2}\label{eq:symbol sr Laplace}
\end{equation}
using the dual metric while its sub-principal symbol is zero. The
characteristic variety 
\begin{align}
\Sigma & =\left\{ \sigma\left(\Delta_{g^{E},\mu}\right)=0\right\} =E^{\perp}\coloneqq\left\{ \xi\in T^{*}X|\xi\left(v\right)=0,\forall v\in E\right\} \label{eq:characteristic variety}
\end{align}
is the annihilator.

From the local expression \prettyref{eq: sR laplacian in orthonormal frame}
the sR Laplacian is seen to be a sum of squares operator of Hörmander
type \cite{Hormander1967} and is thus hypoelliptic. Further it satisfies
the following optimal sub-elliptic estimate \cite{Rothschild-Stein76}
\begin{equation}
\left\Vert f\right\Vert _{H^{s+2/r}}\leq C\left[\left\Vert \Delta_{g^{E},\mu}f\right\Vert _{H^{s}}+\left\Vert f\right\Vert _{H^{s}}\right],\quad\forall f\in C^{\infty}\left(X\right),\,\forall s\in\mathbb{R},\label{eq:subelliptic estimate}
\end{equation}
where $r\coloneqq\sup_{x\in X}\,r\left(x\right)$ is the maximal degree
of non-holonomy. It now follows that $\Delta_{g^{E},\mu}$ has a compact
resolvent and thus there is a complete orthonormal basis of $\left\{ \varphi_{j}\right\} _{j=0}^{\infty}$
for $L^{2}\left(X,\mu\right)$ consisting of (real-valued) eigenvectors
$\Delta_{g^{E},\mu}\varphi_{j}=\lambda_{j}\varphi_{j}$, $0\leq\lambda_{0}\leq\lambda_{1}\leq\ldots$. 

For each $p\in\Sigma$ on the characteristic variety, the fundamental
matrix $F_{p}\in\textrm{End}\left(T_{p}M\right)$, $M\coloneqq T^{*}X$,
is defined via 
\[
\omega\left(.,F_{p}.\right)=\nabla^{2}\sigma\left(.,.\right),
\]
where $\nabla^{2}\sigma$ denotes the Hessian of the symbol \prettyref{eq:symbol sr Laplace}
and $\omega$ the symplectic form on $T^{*}X$. The fundamental matrix
clearly satisfies $\omega\left(.,F_{p}.\right)=-\omega\left(F_{p}.,.\right)$
and we denote by $\textrm{Spec}^{+}\left(iF_{p}\right)$ the set of
real and positive eigenvalues of $iF_{p}$. Under the condition that
\begin{equation}
\textrm{tr}^{+}F_{p}\coloneqq\sum_{\mu\in\textrm{Spec}^{+}\left(iF_{p}\right)}\mu>0\label{eq:trace fund. matrix}
\end{equation}
the sR Laplacian $\Delta_{g^{E},\mu}$ is known to satisfy the better
sub-elliptic estimate with loss of one derivative \cite{HormanderIII}
\begin{equation}
\left\Vert f\right\Vert _{H^{s+1}}\leq C\left[\left\Vert \Delta_{g^{E},\mu}f\right\Vert _{H^{s}}+\left\Vert f\right\Vert _{H^{s}}\right],\quad\forall f\in C^{\infty}\left(X\right),\,\forall s\in\mathbb{R}.\label{eq:hypoelliptic with loss of one derivative}
\end{equation}
In \prettyref{eq:refined subelliptic estimate} we shall prove a further
refined subelliptic estimate for $\Delta_{g^{E},\mu}$ in the particular
4D quasi-contact case of our interest.

As a first property for the sR Laplacian we prove the finite propagation
speed for its half-wave equation.
\begin{thm}
\label{thm:(Finite-propagation-speed)}(Finite propagation speed)
Let $u\left(x;t\right)$ be the unique solution to the initial value
problem 
\begin{align}
\left(i\partial_{t}+\sqrt{\Delta_{g^{E},\mu}}\right)u & =0\nonumber \\
u\left(x,0\right) & =u_{0}\in C^{-\infty}\left(X\right).\label{eq:half wave eq}
\end{align}
Then the solution satisfies 
\[
\textrm{spt }u\left(x;t\right)\subset\left\{ y|\exists x\in\textrm{spt}u_{0};\,d^{E}\left(x,y\right)\leq\left|t\right|\right\} .
\]
\end{thm}
\begin{proof}
The result maybe restated in terms of the Schwartz kernel $K_{t}=\left[e^{it\sqrt{\Delta_{g^{E},\mu}}}\right]_{\mu}$
of the half-wave operator 
\[
\textrm{spt }K_{t}\subset\left\{ \left(x,y\right)|d^{E}\left(x,y\right)\leq\left|t\right|\right\} .
\]
We choose a family of metrics $g_{\epsilon}^{TX}$ \prettyref{eq: extending metrics}
extending $g^{E}$. The Riemannian Laplacian $\Delta_{g_{\varepsilon}^{TX},\mu}$
(still coupled to the form $\mu$) is written 
\[
\Delta_{g_{\epsilon}^{TX},\mu}=\Delta_{g^{E},\mu}+\epsilon\Delta_{g^{F},\mu}
\]
where $\Delta_{g^{F},\mu}$ is the sR Laplacian on the complementary
distribution $F$. The min-max principle for eigenvalues implies the
$L^{2}$ convergence $\Pi_{\left[0,L\right]}^{\Delta_{g_{\epsilon}^{TX},\mu}}\rightarrow\Pi_{\left[0,L\right]}^{\Delta_{g^{E},\mu}}$of
the corresponding spectral projectors onto the interval $\left[0,L\right]$,
$\forall L>0$. It now follows that $K_{t}^{\varepsilon}\rightharpoonup K_{t}$
weakly as $\varepsilon\rightarrow0$ with $K_{t}^{\varepsilon}\coloneqq\left[e^{it\sqrt{\Delta_{g_{\epsilon}^{TX},\mu}}}\right]_{\mu}$.
Knowing that $d^{E}$ is the limit of the Riemannian distance function
for $g_{\epsilon}^{TX}$, the theorem now follows from the finite
propagation speed of $\Delta_{g_{\varepsilon}^{TX},\mu}$.
\end{proof}

\subsection{\label{subsec:Quasi-contact-case}Quasi-contact case}

We now describe some sR geometric features in the particular four
dimensional quasi-contact case of our interest. We now let $X$ be
a smooth, compact oriented four dimensional manifold. A nowhere vanishing
one form $a\in\Omega^{1}\left(X\right)$ is called quasi-contact,
sometimes referred to as even-contact, if the restriction $\textrm{rk}\left.da\right|_{E}=2$,
$E\coloneqq\textrm{ker }a\subset TX$, is of maximal rank. The kernel
$L^{E}\coloneqq\textrm{ker}\left(a\wedge da\right)\subset E$ is then
seen to be one dimensional defining the characteristic line field
which furthermore only depends on $E=\textrm{ker}a$. Let $\left(L^{E}\right)^{\perp}\subset E$
denote the two dimensional orthogonal complement of the characteristic
line on which the restriction $\left.da\right|_{\left(L^{E}\right)^{\perp}}$
is non-degenerate by definition. In particular the bundle $\left(L^{E}\right)^{\perp}$
is orientable. A canonical Popp one form $a_{g^{E}}$ (well-defined
up to a sign) defining $E=\textrm{ker}\left(a_{g^{E}}\right)$ may
now be given by requiring that 
\begin{equation}
\left.da_{g^{E}}\right|_{\left(L^{E}\right)^{\perp}}=\textrm{vol}\left(\left.g^{E}\right|_{\left(L^{E}\right)^{\perp}}\right)\label{eq:Popp one form}
\end{equation}
 agree with the metric volume form of $\left.g^{E}\right|_{\left(L^{E}\right)^{\perp}}$
corresponding to some choice of orientation for $\left(L^{E}\right)^{\perp}$.
It is now easy to check that the distributions 
\begin{align*}
\left(L^{E}\right)_{2}^{\perp} & \coloneqq\left(L^{E}\right)^{\perp}+\left[\left(L^{E}\right)^{\perp},\left(L^{E}\right)^{\perp}\right]\\
\left(L^{E}\right)^{\perp,da_{g^{E}}} & \coloneqq\left\{ v\in TX|da_{g^{E}}\left(v,e\right)=0,\,\forall e\in\left(L^{E}\right)^{\perp}\right\} 
\end{align*}
are three and two dimensional respectively and both transverse to
$E$. Thus their intersection is one-dimensional and transverse to
$E$. We now define the quasi-contact Reeb vector field $R\in C^{\infty}\left(TX\right)$
to be the unique vector field satisfying $R\in\left(L^{E}\right)_{2}^{\perp}\cap\left(L^{E}\right)^{\perp,da_{g^{E}}}$,
$i_{R}a_{g^{E}}$=1 (cf. \cite{Charlot2002}, \cite[Sec. 10.1]{Boscain-Neel-Rizzi2017}).
Note again that the orientation of $R$ depends on the choice of sign
for $a_{g^{E}}$. However the orientation of $\left(L^{E}\right)^{\perp}\oplus\mathbb{R}\left[R\right]$
defined by $a_{g^{E}}\wedge da_{g^{E}}$ is clearly independent of
the choice of sign. Furthermore, given that $L^{E}$ is transverse
to $\left(L^{E}\right)^{\perp}\oplus\mathbb{R}\left[R\right]$, the
orientation defined by $a_{g^{E}}\wedge da_{g^{E}}$ combines with
the $\mu$-orientation of manifold to define an orientation of $L^{E}$.
This defines the unique positively oriented vector field $Z\in C^{\infty}\left(L^{E}\right)$
such that $\left|Z\right|=1$. We note that ergodicity of $L^{E}$
is equivalent to the ergodicity of the vector field $Z$. Let $Z^{*}\in\Omega^{1}\left(X\right)$
denote the one form which satisfies $Z^{*}\left(Z\right)=1$ and annihilates
$\left(L^{E}\right)^{\perp}\oplus\mathbb{R}\left[R\right]$. The Popp
volume form \prettyref{eq:Popp measure} in the quasi-contact case
is now seen to be
\begin{equation}
\mu_{\textrm{Popp}}\coloneqq Z^{*}\wedge a_{g^{E}}\wedge da_{g^{E}}\label{eq: QC Popp volume}
\end{equation}
and we may also define the normalized Popp volume $\nu_{\textrm{Popp}}\coloneqq\frac{1}{P\left(X\right)}\mu_{\textrm{Popp}}$,
$P\left(X\right)\coloneqq\int\mu_{\textrm{Popp}}$. One now has the
relations 
\begin{align}
L_{Z}a_{g^{E}} & =-da_{g^{E}}\left(R,Z\right)a_{g^{E}}\nonumber \\
L_{Z}\mu_{\textrm{Popp}} & =-da_{g^{E}}\left(R,Z\right)\mu_{\textrm{Popp}}.\label{eq:Z flow reln}
\end{align}
In particular the $Z$-flow preserves $E=\textrm{ker}\left(a_{g^{E}}\right)$.

The characteristic line $L^{E}$ is said to be volume preserving if
there exists a smooth volume on $X$ that is invariant under some
non-vanishing section of $L^{E}$; the existence does not depend on
the choice of the section. In particular there exists a $Z$-invariant
volume $L_{Z}\left(\underbrace{\hat{\rho}_{Z}\mu_{\textrm{Popp}}}_{\eqqcolon\mu_{Z}}\right)=0$
for some positive function $\hat{\rho}_{Z}$ which would in turn satisfy
a similar equation $L_{Z}\hat{\rho}_{Z}=da_{g^{E}}\left(R,Z\right)\hat{\rho}_{Z}$;
thus further giving $L_{Z}\left(\underbrace{\hat{\rho}_{Z}a_{g^{E}}}_{=\hat{a}_{g^{E}}}\right)=0$.
It now follows that the volume preserving condition is equivalent
to the existence of a defining one form $a=\hat{a}_{g^{E}}$ for $E$
with $a\wedge da$ closed. Furthermore it is also known to be equivalent
to the existence of a defining one form $a$ for $E$ with $\textrm{rk }da=2$
being constant \cite[Lemma 2.3]{Kotschick&Vogel2018} or the existence
of a vector field transverse to and preserving $E$ \cite[Prop. 2.1]{Pia2019}.
We note however that the volume preserving condition on $L^{E}$ is
quite restrictive and often violated (see Example \prettyref{exa:non-measure preserving example}
below).

Next, let $Y^{3}\subset X$, $TY\pitchfork L^{E}$ be a locally defined
transverse hypersurface near a point $x\in X$. The restriction of
the one form $a_{g^{E}}$ to $Y$ is then a contact form and one has
Darboux coordinates $\left(x_{1},x_{2},x_{3}\right)$ on $Y$ such
that $\left.a_{g^{E}}\right|_{Y}=\frac{1}{2}\left[dx_{3}+x_{1}dx_{2}-x_{2}dx_{1}\right]$.
One now translates these coordinates by the flow of $Z$ to obtain
local coordinates $\left(x_{0},x_{1},x_{2},x_{3}\right)$ near the
point $x$. Defining the positive function $\hat{\rho}\coloneqq\exp\left\{ \int_{0}^{x_{0}}da_{g^{E}}\left(R,Z\right)\right\} $,
satisfying $Z\hat{\rho}=\partial_{x_{0}}\hat{\rho}=da_{g^{E}}\left(R,Z\right)\hat{\rho}$,
one now computes $\mathcal{L}_{Z}\left(\hat{\rho}a_{g^{E}}\right)=0$
giving 
\begin{align}
\hat{\rho}a_{g^{E}} & =\frac{1}{2}\left[dx_{3}+x_{1}dx_{2}-x_{2}dx_{1}\right]\label{eq: quasi-contact normal form-1}\\
Z & =\partial_{x_{0}}\label{eq:characterictic normal form-1}
\end{align}
$\mu_{\textrm{Popp}}=\frac{1}{2}\hat{\rho}^{-2}dx$ and $Y=\left\{ x_{0}=0\right\} $
locally.

The characteristic variety $\Sigma\subset T^{*}X$ of the Laplacian
$\Sigma=E^{\perp}=\mathbb{R}\left[a\right]$ is clearly the graph
of a defining one form by \prettyref{eq:characteristic variety} in
this case. A homogeneous function of degree one on the characteristic
variety is then defined via 
\begin{align}
\rho:\Sigma & \rightarrow\mathbb{R}\nonumber \\
\rho\left(x,sa_{g^{E}}\left(x\right)\right) & =s,\quad\forall s\in\mathbb{R},\label{eq: homogenous dil func.}
\end{align}
and equals the restriction of the symbol of the Reeb vector field
$\rho=\left.\sigma\left(R\right)\right|_{\Sigma}$. With $X,Y\in\left(L^{E}\right)^{\perp}$,
$da_{g^{E}}\left(X,Y\right)=1$, being a positively oriented orthonormal
basis, the relations
\begin{align}
\left.\left\{ \sigma\left(X\right),\sigma\left(Y\right)\right\} \right|_{\Sigma}=\left.\sigma\left(\left[X,Y\right]\right)\right|_{\Sigma} & =\rho a_{g^{E}}\left(\left[X,Y\right]\right)=\rho\nonumber \\
\left.\left\{ \sigma\left(X\right),\sigma\left(Z\right)\right\} \right|_{\Sigma}=\left.\sigma\left(\left[X,Z\right]\right)\right|_{\Sigma} & =\rho a_{g^{E}}\left(\left[X,Z\right]\right)=0\nonumber \\
\left.\left\{ \sigma\left(Y\right),\sigma\left(Z\right)\right\} \right|_{\Sigma}=\left.\sigma\left(\left[Y,Z\right]\right)\right|_{\Sigma} & =\rho a_{g^{E}}\left(\left[Y,Z\right]\right)=0\label{eq: computations for tr+}
\end{align}
as well as \prettyref{eq: sR laplacian in orthonormal frame} show
that $\rho=\textrm{tr}^{+}F_{p}$ is identifiable with \prettyref{eq:trace fund. matrix}
via the fundamental matrix in this case. This is seen to satisfy the
equation
\begin{equation}
\left.H_{\sigma\left(Z\right)}\rho\right|_{\Sigma}=\left.\left\{ \sigma\left(Z\right),\sigma\left(R\right)\right\} \right|_{\Sigma}=\left.\sigma\left(\left[Z,R\right]\right)\right|_{\Sigma}=\rho a_{g^{E}}\left(\left[Z,R\right]\right)=\rho da_{g^{E}}\left(R,Z\right)\label{eq:rho constant along isotropic directions}
\end{equation}
 along the isotropic directions of $\Sigma.$ From the above computations
the following conditions are seen to be equivalent
\begin{equation}
da_{g^{E}}\left(R,Z\right)=0,\;L_{Z}a_{g^{E}}=0,\;L_{Z}\mu_{\textrm{Popp}}=0,\;\left.H_{\sigma\left(Z\right)}\rho\right|_{\Sigma}=0.\label{eq:equivalent conditions}
\end{equation}

The Popp volume form pulls back under the natural projection to a
four form on $\Sigma$ which we denote by the same notation $\mu_{\textrm{Popp}}$.
It further defines a volume form on $\Sigma$ via $\mu_{\textrm{Popp}}^{\Sigma}\coloneqq d\rho\wedge\mu_{\textrm{Popp}}$.
The Hessian of the symbol $\nabla^{2}\sigma$ gives a non-degenerate,
positive-definite quadratic form on the normal bundle $N\Sigma\coloneqq TM/T\Sigma$,
$M\coloneqq T^{*}X$, over the characteristic variety. Under the canonical
isomorphism of determinant lines $\Lambda^{*}T^{*}M=\left(\Lambda^{*}T^{*}\Sigma\right)\otimes\left(\Lambda^{*}N^{*}\Sigma\right)$,
the lift of the Popp volume is the unique volume satisfying $\frac{1}{4!}\omega^{4}=\mu_{\textrm{Popp}}^{\Sigma}\wedge\det\left(\nabla^{2}\sigma\right)$
(cf. \cite{Menikoff-SjostrandI,Menikoff-SjostrandII}).

Next we define the spherical normal bundle $SN\Sigma\xrightarrow{\pi_{S}}\Sigma$,
$SN\Sigma\coloneqq\left\{ v\in N\Sigma|\nabla^{2}\sigma\left(v,v\right)=1\right\} $.
Let $T\Sigma^{\omega}\subset TM$ be the symplectic complement of
$T\Sigma$. The image $N_{1}\Sigma$ of $T\Sigma^{\omega}\hookrightarrow TM\rightarrow TM/T\Sigma\eqqcolon N\Sigma$
is two dimensional and equipped with an induced symplectic form $\omega_{0}$.
The bundle $N_{1}\Sigma$ has a one dimensional $\nabla^{2}\sigma$-
orthocomplement $N_{0}\Sigma\subset N\Sigma$ . This defines (the
absolute value of) a homogeneous of degree zero function $\Xi_{0}\in C^{\infty}\left(SN\Sigma\right)$
satisfying 
\begin{equation}
\left|\Xi_{0}\right|\coloneqq\left\Vert \pi_{N_{0}\Sigma}\left(v\right)\right\Vert ,\quad\forall v\in SN\Sigma,\label{eq:invariant function blowup}
\end{equation}
with respect to the orthogonal projection/decomposition $N\Sigma=N_{0}\Sigma\oplus N_{1}\Sigma$.
A sign for this function will be defined shortly. An endomorphism
$\mathfrak{J}$ of $N_{1}\Sigma$ is defined via $\nabla^{2}\sigma\left(.,\mathfrak{J}.\right)=\omega_{0}\left(.,.\right)$.
This defines a circle action on $N_{1}\Sigma$ via $e^{i\theta}.v_{0}=\left(\cos\theta\right)v_{0}+\left(\sin\theta\right)\frac{\mathfrak{J}}{\left|\mathfrak{J}\right|}v_{0}$
and subsequently one on $SN\Sigma$ which fixes $N_{0}\Sigma$. We
denote by $R_{0}=\partial_{\theta}\in C^{\infty}\left(TSN\Sigma\right)$
the generating vector field satisfying $\left(\pi\circ\beta\right)_{*}R_{0}=0\in TX$.
The quotient $\left(SN\Sigma\right)/S^{1}$ is an interval $\left[-1,1\right]_{\Xi_{0}}$
bundle over $\Sigma$. The vertical fiber measure $\mu_{V}=\left(1-\Xi_{0}^{2}\right)d\Xi_{0}$
again allows to lift the Popp volume via 
\begin{equation}
\mu_{\textrm{Popp}}^{SN\Sigma}\coloneqq\mu_{V}\wedge\pi_{S}^{*}\mu_{\textrm{Popp}}^{\Sigma}.\label{eq:homogeneous lift Popp measure}
\end{equation}
which may equivalently be thought of as a rotationally invariant volume
on the spherical normal bundle $SN\Sigma$ satisfying
\begin{align}
L_{R_{0}}\Xi_{0}=0,\; & L_{R_{0}}\mu_{\textrm{Popp}}^{SN\Sigma}=0.\label{eq:rotational invariance}
\end{align}

The blow-up of the cotangent space along the characteristic variety
\begin{equation}
\left[M;\Sigma\right]\coloneqq\left(M\setminus\Sigma\right)\amalg SN\Sigma\label{eq: blowup def.}
\end{equation}
and the corresponding blow-down map 
\begin{align}
\beta:\left[M;\Sigma\right] & \rightarrow M\nonumber \\
\beta\left(p\right) & \coloneqq\begin{cases}
p; & p\in\left(M\setminus\Sigma\right)\\
\pi_{S}\left(p\right); & p\in SN\Sigma.
\end{cases}\label{eq: blowdow map def.}
\end{align}
may now be defined. The blowup has the structure of a smooth manifold
with boundary; its interior is $\left[M;\Sigma\right]^{o}=\left(M\setminus\Sigma\right)$
while the boundary $\partial\left[M;\Sigma\right]=SN\Sigma$ is identified
with the spherical normal bundle. The boundary defining function is
the square root of the symbol $\sigma^{1/2}$ (or its pullback to
the blowup). There is a natural action of $\mathbb{R}_{+}$ on the
blowup with the quotient $\left[M;\Sigma\right]/\mathbb{R}_{+}=\left[S^{*}X;S^{*}\Sigma\right]$
canonically identified with the corresponding blowup of the cospheres
$S^{*}X=T^{*}X/\mathbb{R}_{+}$, $S^{*}\Sigma\coloneqq\Sigma/\mathbb{R}_{+}$.
The cosphere of the characteristic variety 
\[
S^{*}\Sigma\coloneqq\Sigma/\mathbb{R}_{+}=\left\{ \rho=\pm1\right\} =\underbrace{X_{+}}_{\eqqcolon\left\{ \left(x,a_{g^{E}}\left(x\right)\right)\right\} }\cup\underbrace{X_{-}}_{\eqqcolon\left\{ \left(x,-a_{g^{E}}\left(x\right)\right)\right\} }\subset\Sigma
\]
is identifiable with two copies of the manifold given a choice of
sign for the Popp form $a_{g^{E}}$ and thus carries the lift of the
Popp volume $\mu_{\textrm{Popp}}$. The spherical normal bundle carries
a similar $\mathbb{R}_{+}$-action and we denote the quotient by $SNS^{*}\Sigma\coloneqq SN\Sigma/\mathbb{R}_{+}$.
The $S^{1}$ action on $SN\Sigma$ is homogeneous of degree zero and
one may form the double quotient $SNS^{*}\Sigma/S^{1}$ as an $\left[-1,1\right]_{\Xi_{0}}$
bundle over $X$. In similar vein as \prettyref{eq:homogeneous lift Popp measure}
this now carries a lift of the Popp measure
\begin{equation}
\mu_{\textrm{Popp}}^{SNS^{*}\Sigma}\coloneqq\mu_{V}\wedge\pi_{S}^{*}\mu_{\textrm{Popp}}\label{eq:lift Popp measure}
\end{equation}
which is again equivalently thought of as a rotationally invariant
volume on the spherical normal bundle $SNS^{*}\Sigma$. We also define
the normalized versions $\nu_{\textrm{Popp}}$, $\nu_{\textrm{Popp}}^{SNS^{*}\Sigma}$
of $\mu_{\textrm{Popp}}$, $\mu_{\textrm{Popp}}^{SNS^{*}\Sigma}$
with total volume one.

In \prettyref{subsec:Normal-form-near-Sigma}, \prettyref{sec:Global-calculus}
we shall show the existence of smooth function $\Omega$, invariantly
defined using the sR structure on a neighborhood of the characteristic
variety $\Sigma$, whose Hamilton vector field restricts
\begin{equation}
\left.H_{\Omega}\right|_{SN\Sigma}=R_{0}\label{eq:singular part Ham. vfield}
\end{equation}
to the rotational derivative $R_{0}$. The Hamilton vector field $H_{\sigma^{1/2}}$
of the square root symbol $\sigma^{1/2}\coloneqq\sigma\left(\Delta_{g^{E},\mu}\right)^{1/2}$
is well-defined on the complement $\left[M;\Sigma\right]^{o}=\left(M\setminus\Sigma\right)$
of the characteristic variety and hence on the interior of the blowup.
Its singularity near the boundary is then captured by the rotational
vector field $R_{0}$. In particular, the following will be proved
in \prettyref{subsec:Normal-form-near-Sigma}.
\begin{prop}
\label{prop: Hamilton v field extension} The Hamilton vector field
$H_{\sigma^{1/2}}$ has a singular expansion 
\begin{equation}
H_{\sigma^{1/2}}=\frac{\sigma\left(R\right)}{\sigma^{1/2}}H_{\Omega}+\hat{Z}+o\left(1\right)\label{eq:singular expansion Hamiltonian}
\end{equation}
near the boundary of the blowup $\left[M;\Sigma\right]$. Here $\hat{Z}\in C^{\infty}\left(TS^{*}N\Sigma\right)$
projects 
\begin{equation}
\left(\pi\circ\beta\right)_{*}\hat{Z}\in L^{E}\subset TX,\label{eq: characteristic pushforward}
\end{equation}
with $\left|\left(\pi\circ\beta\right)_{*}\hat{Z}\right|=\left|\Xi_{0}\right|$,
onto the characteristic line with the lift of the Popp volume \prettyref{eq:lift Popp measure}
preserved under the flow of $\left[\hat{Z}\right]\in C^{\infty}\left(T\left(S^{*}N\Sigma/S^{1}\right)\right)=C^{\infty}\left(TS^{*}N\Sigma/\mathbb{R}\left[R_{0}\right]\right)$
\begin{align}
L_{\left[\hat{Z}\right]}\mu_{\textrm{Popp}}^{SNS^{*}\Sigma} & =0.\label{eq:invariance properties of flow}
\end{align}
\end{prop}
We note that the above \prettyref{eq: characteristic pushforward}
also defines a signed version of \prettyref{eq:invariant function blowup}
via $\Xi_{0}=\left\langle Z,\left(\pi\circ\beta\right)_{*}\hat{Z}\right\rangle $.

\section{\label{sec:Hermite-Calculus}Hermite Calculus}

In this section we define the requisite Hermite-Landau calculus. We
begin with the definition of the Hermite transform.

\subsection{Hermite transform}

Below we denote by $\left(x_{0},x_{1},x_{2},x_{3}\right)$ the coordinates
on $\mathbb{R}^{4}$ and abbreviate $\underline{x}=\left(x_{0},x_{2},x_{3}\right)$.
Let $\left(T^{*}\mathbb{R}^{4}\right)_{+}=\left\{ \left(x,\xi\right)\in T^{*}\mathbb{R}^{4}|\xi_{3}>0\right\} $
and let $\left(\hat{\xi}_{0},\hat{\xi}_{1},\hat{\xi}_{2}\right)=\left(\xi_{3}^{-1}\xi_{0},\xi_{3}^{-1}\xi_{1},\xi_{3}^{-1}\xi_{2}\right)$
to be the homogeneous variables on this cone. It shall also be useful
to define the homogeneous variables 
\begin{align}
\hat{x}_{3} & \coloneqq x_{3}+\frac{1}{2}x_{1}\hat{\xi}_{1}\nonumber \\
\Omega & \coloneqq\xi_{3}\left(x_{1}^{2}+\hat{\xi}_{1}^{2}\right)\quad\textrm{ satisfying }\nonumber \\
\left\{ \hat{x}_{3},\Omega\right\}  & =0.\label{eq:symplectic commutation}
\end{align}

Set $h_{k}\left(u\right)\coloneqq\frac{\pi^{1/4}}{\left(2^{k}k!\right)}\left[-\partial_{u}+u\right]^{k}e^{-\frac{1}{2}u^{2}}$
to be the $k$th Hermite function and set $h_{k}\left(x_{1},\xi_{3}\right)\coloneqq\left|\xi_{3}\right|^{1/4}h_{k}\left(\left|\xi_{3}\right|^{1/2}x_{1}\right)$;
$k\in\mathbb{N}_{0}$. The Hermite operators $H_{k}:\mathcal{S}_{c}'\left(\mathbb{R}_{\underline{x}}^{3}\right)\rightarrow\mathcal{S}'\left(\mathbb{R}_{x}^{4}\right)$,
$H_{k}^{*}:\mathcal{S}_{c}'\left(\mathbb{R}_{x}^{4}\right)\rightarrow\mathcal{S}'\left(\mathbb{R}_{\underline{x}}^{3}\right)$
are then defined
\begin{align}
\left(H_{k}u\right)\left(x\right) & \coloneqq\left(2\pi\right)^{-1}\int e^{ix_{3}\xi_{3}}h_{k}\left(x_{1},\xi_{3}\right)\left(\mathcal{F}_{x_{3}}u\right)\left(\xi_{3}\right)d\xi_{3}\label{eq:Hermite operator}\\
\left(H_{k}^{*}u\right)\left(\underline{x}\right) & \coloneqq\left(2\pi\right)^{-1}\int e^{ix_{3}\xi_{3}}h_{k}\left(x_{1},\xi_{3}\right)\left(\mathcal{F}_{x_{3}}u\right)\left(\xi_{3}\right)dx_{1}d\xi_{3}\label{eq:Hermite operator adjoint}
\end{align}
where $\mathcal{F}_{x_{3}}\left(\xi_{3}\right)\coloneqq\int e^{-ix_{3}\xi_{3}}u\left(x_{3}\right)dx_{3}$
denotes the partial Fourier transform in the $x_{3}$ variable. The
above clearly maps $L^{2}\left(\mathbb{R}_{\underline{x}}^{3}\right)$,
$L^{2}\left(\mathbb{R}_{x}^{4}\right)$ into each other and as such
are adjoints satisfying 
\begin{equation}
H_{k}^{*}H_{l}=\delta_{kl}.\label{eq:Hermite op. orth.}
\end{equation}
 It is then an easy exercise to show 
\begin{align*}
WF\left(H_{k}u\right) & =\left\{ \left(0,\underline{x};0,\underline{\xi}\right)|\left(\underline{x};\underline{\xi}\right)\in WF\left(u\right)\right\} \\
WF\left(H_{k}^{*}v\right) & =\left\{ \left(\underline{x};\underline{\xi}\right)|\left(0,\underline{x};0,\underline{\xi}\right)\in WF\left(v\right)\right\} 
\end{align*}
$\forall u\in\mathcal{S}_{c}'\left(\mathbb{R}_{\underline{x}}^{3}\right),\,v\in\mathcal{S}_{c}'\left(\mathbb{R}_{x}^{4}\right)$
and $k\in\mathbb{N}_{0}$. In particular distributions in $\mathcal{S}_{c}'\left(\mathbb{R}_{\underline{x}}^{3}\right)$
micro-supported in $\left\{ \xi_{3}>c\left|\underline{\xi}\right|\right\} \subset T^{*}\mathbb{R}^{3}$
are mapped into those micro-supported in $\left\{ \xi_{3}>c\left|\xi\right|\right\} \subset T^{*}\mathbb{R}^{4}$
under $H_{k}$ for each $c>0$ and vice versa under $H_{k}^{*}$.
As acting on such one now has the identities

\begin{align}
\left[-\xi_{1}+x_{1}\xi_{3}\right]H_{k} & =H_{k+1}\left[2\left(k+1\right)\xi_{3}\right]^{1/2}\nonumber \\
\left[\xi_{1}+x_{1}\xi_{3}\right]H_{k} & =H_{k-1}\left[2k\xi_{3}\right]^{1/2},\nonumber \\
\Omega H_{k} & =H_{k}\left(2k+1\right)\nonumber \\
\hat{x}_{3}H_{k} & =H_{k}x_{3}\label{eq:raising lowering Op}
\end{align}
(cf. \cite{Boutet-Treves-74} Sec. 6). In particular 
\begin{equation}
a\left(x_{0},x_{2},\hat{x}_{3};\xi_{0},\xi_{2},\xi_{3};x_{1}^{2}+\hat{\xi}_{1}^{2}\right)^{W}H_{k}=H_{k}a\left(\underline{x},\underline{\xi};\xi_{3}^{-1}\left(2k+1\right)\right)^{W}\label{eq: Hermite transf symbols}
\end{equation}
for any $a\in S^{m}\left(T^{*}\mathbb{R}_{x}^{4}\right)$ of the given
form. The image of each $H_{k}$ thus corresponds to an eigenspace
of $\Omega$ by \prettyref{eq:raising lowering Op} and is referred
to as a Landau level.

The \textit{Hermite transform }is now defined\textit{
\begin{align}
H^{*}:\mathcal{S}_{c}'\left(\mathbb{R}_{x}^{4}\right) & \rightarrow\mathcal{S}'\left(\mathbb{R}_{\underline{x}}^{3};\mathbb{C}^{\mathbb{N}_{0}}\right);\nonumber \\
\left(H^{*}u\right)_{k} & \coloneqq H_{k}^{*}u\label{eq:Hermite transform}
\end{align}
} as the map from $\mathcal{S}_{c}'\left(\mathbb{R}_{4}^{4}\right)$
into $\mathcal{S}'\left(\mathbb{R}_{\underline{x}}^{3}\right)$-valued
$\mathbb{N}_{0}$-sequences. 

Next, set $h^{s}\coloneqq\left\{ u:\mathbb{N}_{0}\rightarrow\mathbb{C}|\left\Vert u\right\Vert _{s}\coloneqq\sum\left\langle k\right\rangle ^{2s}\left|u\left(k\right)\right|^{2}<\infty\right\} \subset\mathbb{C}^{\mathbb{N}_{0}}$
with the special notation $l^{2}=h^{0}$. As just noted $H$ maps
$L_{c}^{2}\left(\mathbb{R}_{x}^{4}\right)$ into $L^{2}\left(\mathbb{R}_{\underline{x}}^{3};l^{2}\right)$.
More generally, we define $\forall s_{1},s_{2}\in\mathbb{R}$ the
anisotropic Sobolev space 
\[
h^{s_{1},s_{2}}=H^{s_{2}}\left(\mathbb{R}_{\underline{x}}^{3};h^{s_{1}}\right)=\left\{ u:\mathbb{N}_{0}\rightarrow H^{s_{2}}\left(\mathbb{R}_{\underline{x}}^{3}\right)|\left\Vert u\right\Vert _{s_{1},s_{2}}\coloneqq\left(\sum_{k\in\mathbb{N}_{0}}\left\langle k\right\rangle ^{2s_{1}}\left\Vert u\left(k\right)\right\Vert _{H^{s_{2}}}^{2}\right)^{1/2}<\infty\right\} .
\]
For $s_{1},s_{2}\in\mathbb{N}_{0}$, it follows from \prettyref{eq:raising lowering Op}
that the Hermite transform $H$ is a isomorphism between $h^{s_{1},s_{2}}$
and the space
\begin{equation}
H^{s_{1},s_{2}}\coloneqq\left\{ u\in\mathcal{S}'\left(\mathbb{R}_{x}^{4}\right)|\left(x_{1}\xi_{3}^{1/2}\right)^{\alpha}\left(\xi_{3}^{-1/2}\xi_{1}\right)^{\beta}\left\langle \underline{\xi}\right\rangle ^{s_{2}}\hat{u}\in L^{2}\left(\mathbb{R}_{\xi}^{4}\right),\,\forall\left|\alpha\right|+\left|\beta\right|\leq2s_{1}\right\} \subset\mathcal{S}'\left(\mathbb{R}_{x}^{4}\right).\label{eq:anisotropic Sobolev}
\end{equation}

\subsection{Symbol classes}

In this subsection we define classes of pseudo-differential operators
on $\mathbb{R}^{4}$ using the Hermite transform \prettyref{eq:Hermite transform}. 

First for each $\delta_{1},\delta_{2}>0$ define the conic subsets
\begin{align}
K_{\delta_{1},\delta_{2}}\coloneqq & \left\{ \xi_{3}>0;\left|\left(\hat{\xi}_{2},x_{0},x_{2},\hat{x}_{3}\right)\right|<\delta_{1},\left|\left(\hat{\xi}_{0},\hat{\xi}_{1},x_{1}\right)\right|<\delta_{2}\right\} \subset T^{*}\mathbb{R}^{4}\nonumber \\
\Sigma_{0}\coloneqq & \left\{ \left(x,\xi\right)\in K_{\delta_{1},\delta_{2}}|\xi_{0}=x_{1}=\xi_{1}=0\right\} \subset K_{\delta_{1},\delta_{2}}\subset T^{*}\mathbb{R}^{4}\label{eq:model cones}
\end{align}
containing the point $\left(0,0,0,0;0,0,0,1\right)\in T^{*}\mathbb{R}^{4}$.
The corresponding spherical bundles for the cones above are $S^{*}K_{\delta_{1},\delta_{2}}=\left\{ \left(x,\xi\right)\in K_{\delta_{1},\delta_{2}}|\left|\xi\right|=1\right\} $,
$S^{*}\Sigma_{0}=\left\{ \left(x,\xi\right)\in\Sigma_{0}|\left|\xi\right|=1\right\} $.
Letting $\rho\left(x_{0},x_{2}\hat{x}_{3};\hat{\xi}_{2},\xi_{3}\right)=\xi_{3}\hat{\rho}\left(x_{0},x_{2}\hat{x}_{3};\hat{\xi}_{2}\right)\in C^{\infty}\left(\Sigma_{0}\right)$,
be a positive homogeneous function of degree one on the sub-cone $\Sigma_{0}$
set 
\begin{equation}
d_{\rho}\left(x,\xi\right)=\xi_{3}\hat{d}_{\rho}\coloneqq\sqrt{\xi_{0}^{2}+\rho\Omega}\label{eq:cone defining function}
\end{equation}
as a defining function for the respective sub-cone $\left\{ d_{\rho}=0\right\} =\Sigma_{0}\subset K_{\delta_{1},\delta_{2}}$
and subset $\left\{ d_{\rho}=0\right\} =S^{*}\Sigma_{0}\subset S^{*}K_{\delta_{1},\delta_{2}}$
above. It is further mapped to $d_{k}\coloneqq\sqrt{\xi_{0}^{2}+\rho\left(2k+1\right)}$
under the Hermite transform $H_{k}^{*}$ for each $k\in\mathbb{N}_{0}$.
The blowup along these sub-cones and corresponding blowdown map are
defined via
\begin{align}
\left[K_{\delta_{1},\delta_{2}};\Sigma_{0}\right] & \coloneqq\left\{ \left(x,\xi\right)\in K_{\delta_{1},\delta_{2}}|\hat{d}_{\rho}\geq1\right\} \nonumber \\
\beta: & \left[K_{\delta_{1},\delta_{2}};\Sigma_{0}\right]\rightarrow K_{\delta_{1},\delta_{2}};\nonumber \\
\beta\left(x_{1},\hat{\xi}_{0},\hat{\xi}_{1};x_{0},x_{2},x_{3},\hat{\xi}_{2},\xi_{3}\right) & \coloneqq\left(\frac{\hat{d}_{\rho}-1}{\hat{d}_{\rho}}\left(x_{1},\hat{\xi}_{0},\hat{\xi}_{1}\right);x_{0},x_{2},x_{3},\hat{\xi}_{2},\xi_{3}\right).\label{eq:blowdown}
\end{align}
The blowup $\left[K_{\delta_{1},\delta_{2}};\Sigma_{0}\right]$ is
a manifold with boundary 
\[
\partial\left[K_{\delta_{1},\delta_{2}};\Sigma_{0}\right]=\left\{ \left(x,\xi\right)\in K_{\delta_{1},\delta_{2}},\,\hat{d}_{\rho}=1\right\} 
\]
 and interior $\left[K_{\delta_{1},\delta_{2}};\Sigma_{0}\right]^{o}=\left\{ \left(x,\xi\right)\in K_{\delta_{1},\delta_{2}},\,\hat{d}_{\rho}>1\right\} $.
The boundary defining function is the pullback 
\begin{equation}
\beta^{*}\hat{d}_{\rho}=\hat{d}_{\rho}-1\label{eq:boundary defining function}
\end{equation}
of \prettyref{eq:cone defining function} under the blowdown. A similar
blowup $\left[S^{*}K_{\delta_{1},\delta_{2}};S^{*}\Sigma_{0}\right]$
with interior $\left[S^{*}K_{\delta_{1},\delta_{2}};S^{*}\Sigma_{0}\right]^{o}$
and corresponding blowdown map to $S^{*}K_{\delta_{1},\delta_{2}}$
may also be defined. Let $C^{\infty}\left(\left[K_{\delta_{1},\delta_{2}};\Sigma_{0}\right]^{o}\right),\,C^{\infty}\left(\left[S^{*}K_{\delta_{1},\delta_{2}};S^{*}\Sigma_{0}\right]^{o}\right)$,
$C^{\infty}\left(\left[K_{\delta_{1},\delta_{2}};\Sigma_{0}\right]\right),\,C^{\infty}\left(\left[S^{*}K_{\delta_{1},\delta_{2}};S^{*}\Sigma_{0}\right]\right)$
denote smooth functions on the interior and those extending to the
boundary respectively. Similarly, 
\begin{align}
C_{\textrm{inv}}^{\infty}\left(\left[K_{\delta_{1},\delta_{2}};\Sigma_{0}\right]^{o}\right) & \subset C^{\infty}\left(\left[K_{\delta_{1},\delta_{2}};\Sigma_{0}\right]^{o}\right)\nonumber \\
C_{\textrm{inv}}^{\infty}\left(\left[S^{*}K_{\delta_{1},\delta_{2}};S^{*}\Sigma_{0}\right]^{o}\right) & \subset C^{\infty}\left(\left[S^{*}K_{\delta_{1},\delta_{2}};S^{*}\Sigma_{0}\right]^{o}\right)\nonumber \\
\,C_{\textrm{inv}}^{\infty}\left(\left[K_{\delta_{1},\delta_{2}};\Sigma_{0}\right]\right) & \subset C^{\infty}\left(\left[K_{\delta_{1},\delta_{2}};\Sigma_{0}\right]\right)\nonumber \\
C_{\textrm{inv}}^{\infty}\left(\left[S^{*}K_{\delta_{1},\delta_{2}};S^{*}\Sigma_{0}\right]\right) & \subset C^{\infty}\left(\left[S^{*}K_{\delta_{1},\delta_{2}};S^{*}\Sigma_{0}\right]\right)\label{eq:Classes 1}
\end{align}
 are subsets of those functions $f$ which have the rotational symmetry
\[
\left\{ \Omega,f\right\} =\left(x_{1}\partial_{\hat{\xi}_{1}}-\hat{\xi}_{1}\partial_{x_{1}}\right)f=0.
\]
These are functions of the arguments 
\begin{equation}
\left(x_{1}^{2}+\hat{\xi}_{1}^{2},\hat{\xi}_{0},\hat{\xi}_{2},\xi_{3};x_{0},x_{2},\hat{x}_{3}\right).\label{eq:arguments}
\end{equation}
Further, let 
\begin{align}
\left(\beta^{*}d_{\rho}\right)^{-m_{2}}C^{\infty}\left(\left[K_{\delta_{1},\delta_{2}};\Sigma_{0}\right]\right) & \subset C^{\infty}\left(\left[K_{\delta_{1},\delta_{2}};\Sigma_{0}\right]^{o}\right)\nonumber \\
\left(\beta^{*}d_{\rho}\right)^{-m_{2}}C_{\textrm{inv}}^{\infty}\left(\left[K_{\delta_{1},\delta_{2}};\Sigma_{0}\right]\right) & \subset C^{\infty}\left(\left[K_{\delta_{1},\delta_{2}};\Sigma_{0}\right]^{o}\right)\nonumber \\
\left(\beta^{*}\hat{d}_{\rho}\right)^{-m_{2}}C^{\infty}\left(\left[S^{*}K_{\delta_{1},\delta_{2}};S^{*}\Sigma_{0}\right]\right) & \subset C^{\infty}\left(\left[S^{*}K_{\delta_{1},\delta_{2}};S^{*}\Sigma_{0}\right]^{o}\right)\nonumber \\
\left(\beta^{*}\hat{d}_{\rho}\right)^{-m_{2}}C_{\textrm{inv}}^{\infty}\left(\left[S^{*}K_{\delta_{1},\delta_{2}};S^{*}\Sigma_{0}\right]\right) & \subset C^{\infty}\left(\left[S^{*}K_{\delta_{1},\delta_{2}};S^{*}\Sigma_{0}\right]^{o}\right)\label{eq:Classes 2}
\end{align}
denote the set of functions $f$ in the interiors such that $\left(\beta^{*}d_{\rho}\right)^{m_{2}}f\in C^{\infty}\left(\left[K_{\delta_{1},\delta_{2}};\Sigma_{0}\right]\right)$,
$\left(\beta^{*}d_{\rho}\right)^{m_{2}}f\in C_{\textrm{inv}}^{\infty}\left(\left[K_{\delta_{1},\delta_{2}};\Sigma_{0}\right]\right)$,$\left(\beta^{*}\hat{d}_{\rho}\right)^{m_{2}}f\in C^{\infty}\left(\left[S^{*}K_{\delta_{1},\delta_{2}};S^{*}\Sigma_{0}\right]\right)$,
$\left(\beta^{*}\hat{d}_{\rho}\right)^{m_{2}}f\in C_{\textrm{inv}}^{\infty}\left(\left[S^{*}K_{\delta_{1},\delta_{2}};S^{*}\Sigma_{0}\right]\right)$
respectively. Finally denote by $\left(\beta^{*}d_{\rho}\right)^{-m_{2}}C_{c,\textrm{inv}}^{\infty}\left(\left[K_{\delta_{1},\delta_{2}};\Sigma_{0}\right]\right)$,
$\left(\beta^{*}\hat{d}_{\rho}\right)^{-m_{2}}C_{c,\textrm{inv}}^{\infty}\left(\left[S^{*}K_{\delta_{1},\delta_{2}};S^{*}\Sigma_{0}\right]\right)$
the subset of those functions supported in $\left[K_{\delta_{1}-\varepsilon,\delta_{2}-\varepsilon};\Sigma_{0}\right]$,
$\left[SK_{\delta_{1}-\varepsilon,\delta_{2}-\varepsilon};S^{*}\Sigma_{0}\right]$
respectively for some $\varepsilon>0$. 

Next set $\partial_{\hat{d}_{\rho}}\coloneqq\frac{1}{\hat{d}_{\rho}}\left[\hat{\xi}_{0}\partial_{\hat{\xi}_{0}}+\hat{\rho}\left(x_{1}\partial_{x_{1}}+\hat{\xi}_{1}\partial_{\hat{\xi}_{1}}\right)\right]$
to be the homogeneous radial vector field on the blowups $\left[K_{\delta_{1},\delta_{2}};\Sigma_{0}\right],\left[S^{*}K_{\delta_{1},\delta_{2}};S^{*}\Sigma_{0}\right].$
We now define the class of symbols $S^{m_{1},m_{2}}\left(\mathbb{R}^{4},\Sigma_{0}\right)$,
$m_{1},m_{2}\in\mathbb{R}$, as the set of functions $a\in\left(\beta^{*}d_{\rho}\right)^{-m_{2}}C_{c,\textrm{inv}}^{\infty}\left(\left[K_{1,1};\Sigma_{0}\right]\right)$
satisfying
\begin{equation}
\left\Vert a\right\Vert _{\alpha,\beta,\gamma}\coloneqq\sup_{\left[K_{1,1};\Sigma_{0}\right]}\left|\left(\beta^{*}\hat{d}_{\rho}\right)^{m_{2}+\beta}\xi_{3}^{-m_{1}+\alpha}\partial_{\xi_{3}}^{\alpha}\partial_{\hat{d}_{\rho}}^{\beta}\left(T_{1}\ldots T_{N}\right)a\right|<\infty,\label{eq:symbolic estimates}
\end{equation}
$\forall\left(\alpha,\beta\right)\in\mathbb{N}_{0}\times\mathbb{N}_{0}$
and any set of smooth, homogeneous of degree zero, vector fields $\left(T_{1},\ldots,T_{N}\right)$
on the blowup that are tangent to the boundary. For any $a\in S^{m_{1},m_{2}}\left(\mathbb{R}^{4},\Sigma_{0}\right)$,
we shall also define the associated sequence of functions $a_{k}\in C^{\infty}\left(\mathbb{R}_{\underline{x},\underline{\xi}}^{6}\right)$
, $k\in\mathbb{N}_{0}$, via 
\begin{equation}
a_{k}\coloneqq\left(\beta^{-1}\right)^{*}a\left(\left(2k+1\right)\xi_{3}^{-1};\hat{\xi}_{0},\hat{\xi}_{2},\xi_{3};x_{0},x_{2},x_{3}\right),\label{eq: Hermite transform of symbol}
\end{equation}
where $\left(2k+1\right)\xi_{3}^{-1}$, $x_{3}$ replace the $x_{1}^{2}+\hat{\xi}_{1}^{2}$,
$\hat{x}_{3}$ arguments \prettyref{eq:arguments} respectively.

\subsection{Quantization and calculus}

The quantization $a^{H}:\mathcal{S}\left(\mathbb{R}_{x}^{4}\right)\rightarrow\mathcal{S}\left(\mathbb{R}_{x}^{4}\right)$
of a symbol $a\in S^{m_{1},m_{2}}\left(\mathbb{R}^{4},\Sigma_{0}\right)$
is defined by the rule 
\begin{equation}
H_{k}a^{H}H_{k'}^{*}=\delta_{kk'}a_{k}^{W}.\label{eq:Hermite Quantization 1}
\end{equation}
or alternately written
\begin{equation}
a^{H}=\sum_{k=0}^{\infty}H_{k}^{*}a_{k}^{W}H_{k}.\label{eq:def. Hermite Quantization 2}
\end{equation}
We denote by $\Psi^{m_{1},m_{2}}\left(\mathbb{R}^{4};\Sigma_{0}\right)$
the set of such quantizations. We remark that this class depends on
the decorated cone $\left(\Sigma_{0},\rho\right)$, i.e. additionally
on the homogeneous function $\rho\in C^{\infty}\left(\Sigma_{0}\right)$;
we shall sometimes precise this with the notation $\Psi^{m_{1},m_{2}}\left(\mathbb{R}^{4};\Sigma_{0},\rho\right)$
instead to avoid confusion. We note that the quantization above depends
only on the value of the symbol at points of $\left[K_{1,1};\Sigma_{0}\right]$
where $\xi_{3}\beta^{*}\left(x_{1}^{2}+\hat{\xi}_{1}^{2}\right)\in2\mathbb{N}_{0}+1$.
In particular the quantization $a^{H}$ only depends on the restriction
of $a$ to the parabolic region 
\begin{equation}
P=\left\{ \beta^{*}\hat{d}_{\rho}\geq\xi_{3}^{-1}\right\} .\label{eq: parabolic region}
\end{equation}
 This gives the inclusions
\begin{align}
\Psi^{m_{1},m_{2}}\left(\mathbb{R}^{4};\Sigma_{0}\right) & \subset\Psi^{m_{1}+\frac{1}{2},m_{2}-1}\left(\mathbb{R}^{4};\Sigma_{0}\right)\label{eq: inclusion of pseudo 1}\\
\Psi^{-\infty,m_{2}}\left(\mathbb{R}^{4};\Sigma_{0}\right) & \subset\Psi^{-\infty}\left(\mathbb{R}^{4}\right).\label{eq:inclusion of psedos}
\end{align}
In the case where the symbol $a=\beta^{*}a_{0}$ happens to be the
pullback of $a_{0}\in C^{\infty}\left(\left(T^{*}\mathbb{R}_{x}^{4}\right)^{+}\right)$
under the blowdown one has 
\begin{equation}
a^{H}=a_{0}^{W}\label{eq:twist pullback}
\end{equation}
by \prettyref{eq:raising lowering Op}, \prettyref{eq: Hermite transf symbols}
and this partly motivates our definition \prettyref{eq:Hermite Quantization 1},
\prettyref{eq:def. Hermite Quantization 2}. This also gives the inclusion
\begin{align}
\Psi_{\textrm{inv}}^{m}\left(\mathbb{R}^{4}\right) & \subset\Psi^{m,0}\left(\mathbb{R}^{4};\Sigma_{0}\right)\quad\textrm{where }\nonumber \\
\Psi_{\textrm{inv}}^{m}\left(\mathbb{R}^{4}\right) & \coloneqq\left\{ A=a^{W}\in\Psi^{m}\left(\mathbb{R}^{4}\right)|\textrm{spt}\left(a\right)\subset K_{1,1},\,\left\{ a,\Omega\right\} =0\right\} .\label{eq:inclusion of inv. classical symb.}
\end{align}

Next define a subclass $S_{\textrm{cl}}^{m_{1},m_{2}}\left(\mathbb{R}^{4},\Sigma_{0}\right)\subset S^{m_{1},m_{2}}\left(\mathbb{R}^{4},\Sigma_{0}\right)$
of classical symbols. This is the subset of those symbols $a$ for
which there exist $a_{j}\in\left(\beta^{*}\hat{d}_{\rho}\right)^{-m_{2}}C_{c,\textrm{inv}}^{\infty}\left(\left[S^{*}K_{1,1};S^{*}\Sigma_{0}\right]\right)$,
$j=0,1,\ldots$, $\chi\in C_{c}^{\text{\ensuremath{\infty}}}\left(\mathbb{R}\right)$,
such that 
\begin{equation}
a=\left[1-\chi\left(\left|\xi\right|\right)\right]\xi_{3}^{m_{1}}\left[a_{0}+\left(\beta^{*}d_{\rho}\right)^{-1}a_{1}+\ldots+\left(\beta^{*}d_{\rho}\right)^{-N}a_{N}\right]+S^{m_{1}-\frac{N+1}{2},m_{2}}\left(\mathbb{R}^{4},\Sigma_{0}\right),\label{eq:classical symbolic expansion}
\end{equation}
$\forall N\in\mathbb{N}_{0}$. Here the remainder estimate is understood
on the parabolic region $P$ \prettyref{eq: parabolic region}. Further
time dependent symbol classes $S_{t}^{m_{1},m_{2}}\left(\mathbb{R}^{4},\Sigma_{0}\right)$,
$S_{\textrm{cl},t}^{m_{1},m_{2}}\left(\mathbb{R}^{4},\Sigma_{0}\right)$
are defined as follows: $S_{t}^{m_{1},m_{2}}\left(\mathbb{R}^{4},\Sigma_{0}\right)$
is the set of time-dependent functions $a\in C^{\infty}\left(\mathbb{R}_{t}\times\left[K_{1,1};\Sigma_{0}\right]^{o}\right)$
such that each $a\left(t;.\right)\in\left(\beta^{*}d_{\rho}\right)^{-m_{2}}C_{c,\textrm{inv}}^{\infty}\left(\left[K_{1,1};\Sigma_{0}\right]\right)$,
$t\in\mathbb{R}$, with each estimate \prettyref{eq:symbolic estimates}
being uniform on compact intervals of time. Finally $S_{\textrm{cl},t}^{m_{1},m_{2}}\left(\mathbb{R}^{4},\Sigma_{0}\right)$
is the subset of those symbols $a\in S_{t}^{m_{1},m_{2}}\left(\mathbb{R}^{4},\Sigma_{0}\right)$
for which there exist time independent $a_{j}\in S_{\textrm{cl}}^{m_{1}-j,m_{2}+j}\left(\mathbb{R}^{4},\Sigma_{0}\right)$,
$j=0,1,\ldots$, such that 
\begin{equation}
a=\sum_{j=0}^{N}t^{j}a_{j}+t^{N+1}S^{m_{1}-\frac{N+1}{2},m_{2}}\left(\mathbb{R}^{4},\Sigma_{0}\right),\label{eq:time dep symbolic exp.}
\end{equation}
$\forall N\in\mathbb{N}_{0}$. We denote by $\Psi_{\textrm{cl}}^{m_{1},m_{2}}\left(\mathbb{R}^{4};\Sigma_{0}\right)$,
$\Psi_{\textrm{cl},t}^{m_{1},m_{2}}\left(\mathbb{R}^{4};\Sigma_{0}\right)$
the set of quantizations of the classical symbols \prettyref{eq:classical symbolic expansion},
\prettyref{eq:time dep symbolic exp.}. 

Standard application of Borel's lemma gives asymptotic summation:
for any set of operators $A_{j}\in\Psi^{m_{1}-j,m_{2}+j}$, $B_{j}\in\Psi^{m_{1},m_{2}-j}$,
$j\in\mathbb{N}_{0}$, there exists $A,B\in\Psi^{m_{1},m_{2}}$ such
that
\begin{align}
A-\sum_{j=0}^{N}A_{j}\in\Psi^{m_{1}-\frac{N}{2},m_{2}},\nonumber \\
B-\sum_{j=0}^{N}B_{j}\in\Psi^{m_{1},m_{2}-N}, & \forall N\in\mathbb{N}_{0}\label{eq:asymptotic summation}
\end{align}
and respectively for the classes $\Psi_{\textrm{cl}}^{m_{1},m_{2}}\left(\mathbb{R}^{4};\Sigma_{0}\right)$,
$\Psi_{\textrm{cl},t}^{m_{1},m_{2}}\left(\mathbb{R}^{4};\Sigma_{0}\right)$.

Below we show that these classes are well behaved under composition
and adjoint. 
\begin{prop}
\label{prop: composition in P m1 m2 m3} For $a^{H}\in\Psi^{m_{1},m_{2}}\left(\mathbb{R}^{4};\Sigma_{0}\right),\;b^{H}\in\Psi^{m_{1}',m_{2}'}\left(\mathbb{R}^{4};\Sigma_{0}\right)$
we have 
\begin{align}
a^{H}b^{H} & \in\Psi^{m_{1}+m_{1}',m_{2}+m_{2}'}\left(\mathbb{R}^{4};\Sigma_{0}\right)\nonumber \\
a^{H}b^{H} & =\left(ab\right)^{H}+\Psi^{m_{1}+m_{1}'-1,m_{2}+m_{2}'+1}\left(\mathbb{R}^{4};\Sigma_{0}\right)\nonumber \\
\left(a^{H}\right)^{*} & =\bar{a}^{H}\label{eq:composition properties}
\end{align}
and respectively for the classes $\Psi_{\textrm{cl}}^{m_{1},m_{2}}\left(\mathbb{R}^{4};\Sigma_{0}\right)$,
$\Psi_{\textrm{cl},t}^{m_{1},m_{2}}\left(\mathbb{R}^{4};\Sigma_{0}\right)$.
\end{prop}
\begin{proof}
We first prove that the corresponding symbols $a_{k}$ \& $b_{k}$
\prettyref{eq: Hermite transform of symbol} compose for each $k\in\mathbb{N}_{0}$.
From \prettyref{eq:def. Hermite Quantization 2} and composition of
Weyl symbols, the composed symbol $a_{k}\circ b_{k}$ has an asymptotic
expansion
\begin{align}
a_{k}^{W}\circ b_{k}^{W} & \sim\left[\sum_{\left|\alpha\right|=0}^{\infty}\frac{i^{\alpha}}{\alpha!}\left(\underbrace{D_{x}D_{\eta}-D_{y}D_{\xi}}_{=A\left(D\right)}\right)^{\alpha}\left[a_{k}\left(\underline{x},\underline{\xi}\right)b_{k}\left(\underline{y},\underline{\eta}\right)\right]_{\underline{x}=\underline{y};\underline{\xi}=\underline{\eta}}\right]^{W}.\label{eq:composition exp.}
\end{align}
Each successive term above then corresponds to a symbol in $S^{m_{1}+m_{1}'-\left|\alpha\right|,m_{2}+m_{2}'+\left|\alpha\right|}\left(\mathbb{R}^{4};\Sigma_{0}\right)\subset S^{m_{1}+m_{1}'-\frac{\left|\alpha\right|}{2},m_{2}+m_{2}'}\left(\mathbb{R}^{4};\Sigma_{0}\right)$
\prettyref{eq: inclusion of pseudo 1} and can be asymptotically summed
\prettyref{eq:asymptotic summation}. The residual term above is then
in $\Psi^{-\infty,m_{2}}\left(\mathbb{R}^{4}\right)\subset\Psi^{-\infty}\left(\mathbb{R}^{4}\right)$
\prettyref{eq:inclusion of psedos}. The support condition for the
composed symbol follows from a standard integral representation formula
for the symbol of the composition $a_{k}\circ b_{k}$ (\cite{HormanderIII}
Sec. 18.1). The adjoint property is an immediate consequence of the
usual adjoint property $\left(a_{k}^{W}\right)^{*}=\bar{a}_{k}^{W}$of
Weyl quantization for each $k$.
\end{proof}
The principal symbol of $A=a^{H}\in\Psi_{\textrm{cl}}^{m_{1},m_{2}}$
is now defined via 
\begin{equation}
\sigma_{m_{1},m_{2}}^{H}\left(A\right)=a_{0}\in\left(\beta^{*}\hat{d}_{\rho}\right)^{-m_{2}}C_{c,\textrm{inv}}^{\infty}\left(\left[S^{*}K_{1,1};S^{*}\Sigma_{0}\right]\right)\label{eq:symbol def.}
\end{equation}
to be the leading term in the expansion \prettyref{eq:classical symbolic expansion}
above. One has the symbol short exact sequence 
\begin{equation}
0\rightarrow\Psi_{\textrm{cl}}^{m_{1}-1,m_{2}+1}\left(\mathbb{R}^{4};\Sigma_{0}\right)\rightarrow\Psi_{\textrm{cl}}^{m_{1},m_{2}}\left(\mathbb{R}^{4};\Sigma_{0}\right)\xrightarrow{\sigma_{m_{1},m_{2}}^{H}}\left(\beta^{*}\hat{d}_{\rho}\right)^{-m_{2}}C_{c,\textrm{inv}}^{\infty}\left(\left[S^{*}K_{1,1};S^{*}\Sigma_{0}\right]\right)\rightarrow0.\label{eq:symbol exact seq.}
\end{equation}
From \prettyref{eq:composition properties}, it follows that the symbol
\prettyref{eq:symbol def.} is multiplicative and closed under adjoints
\begin{align}
\sigma_{m_{1}+m_{1}',m_{2}+m_{2}'}^{H}\left(AB\right) & =\sigma_{m_{1},m_{2}}^{H}\left(A\right)\sigma_{m_{1}',m_{2}'}^{H}\left(B\right),\label{eq:symbol is multiplicative}\\
\sigma_{m_{1},m_{2}}^{H}\left(A^{*}\right) & =\overline{\sigma_{m_{1},m_{2}}^{H}\left(A\right)},
\end{align}
$\forall A\in\Psi_{\textrm{cl}}^{m_{1},m_{2}}\left(\mathbb{R}^{4};\Sigma_{0}\right)$,
$B\in\Psi_{\textrm{cl}}^{m_{1}',m_{2}'}\left(\mathbb{R}^{4};\Sigma_{0}\right)$.
The symbol exact sequence \prettyref{eq:symbol exact seq.} gives
\begin{align*}
\left[A,B\right] & \in\Psi_{\textrm{cl}}^{m_{1}+m_{1}'-1,m_{2}+m_{2}'+1}\left(\mathbb{R}^{4};\Sigma_{0}\right)\quad\textrm{with}\\
\sigma_{m_{1}+m_{1}'-1,m_{2}+m_{2}'+1}^{H}\left(\left[A,B\right]\right) & =i\left\{ \sigma_{m_{1},m_{2}}^{H}\left(A\right),\sigma_{m_{1}',m_{2}'}^{H}\left(B\right)\right\} 
\end{align*}
following from \prettyref{eq:composition exp.}.

Next we define the generalized Sobolev spaces as the subspace of tempered
distributions $u\in H^{s_{1},s_{2}}\left(\mathbb{R}_{x}^{4};\Sigma_{0}\right)\subset\mathcal{S}_{c}'\left(\mathbb{R}_{x}^{4}\right)$,
$s_{1},s_{2}\in\mathbb{R}$, micro-supported in $K_{1,1}$ satisfying
\begin{equation}
\left\Vert u\right\Vert _{H^{s_{1},s_{2}}}\coloneqq\int\sum_{k\in\mathbb{N}_{0}}\left|\left(2k+1\right)^{-\frac{1}{2}s_{2}}\left\langle \xi_{3}\right\rangle ^{s_{1}+\frac{1}{2}s_{2}}H_{k}^{*}u\right|^{2}dx<\infty.\label{eq: anisotropicSobolev norm}
\end{equation}
Following \prettyref{eq:raising lowering Op}, \prettyref{eq: Hermite transf symbols},
\prettyref{eq:symbolic estimates} and the Calderon-Vaillancourt inequality,
these can be equivalently characterized as $u\subset\mathcal{S}_{c}'\left(\mathbb{R}_{x}^{4}\right)$
micro-supported in $K_{1,1}$ satisfying 
\begin{equation}
A\in\Psi^{s_{1},s_{2}}\left(\mathbb{R}^{4};\Sigma_{0}\right)\implies Au\in L^{2}.\label{eq:pseudodiff charac of Hs1,s2}
\end{equation}
In light of the inclusions \prettyref{eq:inclusion of psedos}, \prettyref{eq:inclusion of inv. classical symb.}
this gives 
\begin{align}
H^{s_{1}+\frac{1}{2},s_{2}-1}\left(\mathbb{R}_{x}^{4};\Sigma_{0}\right) & \subset H^{s_{1},s_{2}}\left(\mathbb{R}_{x}^{4};\Sigma_{0}\right)\nonumber \\
H^{s_{1},0}\left(\mathbb{R}_{x}^{4};\Sigma_{0}\right) & =\left\{ u\in H^{s_{1}}\left(\mathbb{R}_{x}^{4}\right)|WF\left(u\right)\subset K_{1,1}\right\} .\label{eq:Sobolev inclusions}
\end{align}
One further has Sobolev boundedness
\begin{equation}
u\in H^{s_{1},s_{2}}\left(\mathbb{R}_{x}^{4};\Sigma_{0}\right),\,A\in\Psi^{m_{1},m_{2}}\left(\mathbb{R}^{4};\Sigma_{0}\right)\implies Au\in H^{s_{1}-m_{1},s_{2}-m_{2}}\left(\mathbb{R}_{x}^{4};\Sigma_{0}\right).\label{eq:Sobolev boundedness}
\end{equation}

Next, we define the characteristic wavefront $WF_{\Sigma_{0}}\left(A\right)\subset\partial\left[S^{*}K_{1,1};S^{*}\Sigma_{0}\right]$
of an operator $A\in\Psi_{\textrm{cl}}^{m_{1},m_{2}}\left(\mathbb{R}^{4};\Sigma_{0}\right)$
in the exotic class as a subset of the boundary of the blowup. This
is the intersection $\left[\bigcap_{j=0}^{\infty}\textrm{spt}\left(a_{j}\right)\right]\cap\partial\left[S^{*}K_{1,1};S^{*}\Sigma_{0}\right]$,
of the supports of the symbols in its symbolic expansion \prettyref{eq:classical symbolic expansion}.
The characteristic wavefront $WF_{\Sigma_{0}}\left(u\right)\subset\partial\left[S^{*}K_{1,1};S^{*}\Sigma_{0}\right]$
of a distribution $u\subset\mathcal{S}'\left(\mathbb{R}_{x}^{4}\right)$
micro-supported in $K_{1,1}$ is also defined via 
\begin{equation}
\left(x,\xi\right)\notin WF_{\Sigma_{0}}\left(u\right)\iff\exists A\in\Psi_{\textrm{cl}}^{0,0}\left(\mathbb{R}^{4};\Sigma_{0}\right),\textrm{s.t.}\,\left(x,\xi\right)\in WF_{\Sigma_{0}}\left(A\right),\,Au\in C^{\infty}\label{eq:PDO charac. of WF}
\end{equation}
or equivalently via
\begin{equation}
\left(x,\xi\right)\notin WF_{\Sigma_{0}}\left(u\right)\iff\exists A\in\Psi_{\textrm{cl}}^{0,0}\left(\mathbb{R}^{4};\Sigma_{0}\right),\textrm{s.t.}\,\sigma_{0,0}^{H}\left(A\right)\left(x,\xi\right)\neq0,\,Au\in C^{\infty}.\label{eq:equiv PDO charac. of WF}
\end{equation}

The wavefronts can also be considered as conic subsets of $\partial\left[K_{1,1};\Sigma_{0}\right]$
and are again rotationally invariant under the action of $x_{1}\partial_{\hat{\xi}_{1}}-\hat{\xi}_{1}\partial_{x_{1}}$
by definition. The following are easily established 
\begin{align}
WF_{\Sigma_{0}}\left(A+B\right) & \subset WF_{\Sigma_{0}}\left(A\right)\cup WF_{\Sigma_{0}}\left(B\right)\nonumber \\
WF_{\Sigma_{0}}\left(AB\right) & \subset WF_{\Sigma_{0}}\left(A\right)\cap WF_{\Sigma_{0}}\left(B\right)\nonumber \\
WF_{\Sigma_{0}}\left(Au\right) & \subset WF_{\Sigma_{0}}\left(A\right)\cap WF_{\Sigma_{0}}\left(u\right)\label{eq: properties of WF}
\end{align}
$\forall A\in\Psi_{\textrm{cl}}^{m_{1},m_{2}}\left(\mathbb{R}^{4};\Sigma_{0}\right),\,B\in\Psi_{\textrm{cl}}^{m_{1}',m_{2}'}\left(\mathbb{R}^{4};\Sigma_{0}\right)$
and $u\subset\mathcal{S}'\left(\mathbb{R}_{x}^{4}\right)$ micro-supported
in $K_{1,1}$. Finally using \prettyref{eq:twist pullback}, \prettyref{eq:inclusion of inv. classical symb.}
and a partition of unity argument one shows 
\begin{equation}
\beta\left(WF_{\Sigma_{0}}\left(u\right)\right)=WF\left(u\right)\cap\Sigma_{0}\label{eq: ch WF projects to WF}
\end{equation}
under the blowdown map \prettyref{eq:blowdown}, for each $u\subset\mathcal{S}'\left(\mathbb{R}_{x}^{4}\right)$
micro-supported in $K_{1,1}$.

\section{\label{sec:Birkhoff-normal-forms}Birkhoff normal forms }

In this section we obtain two Birkhoff normal forms for $\Delta_{g^{E},\mu}$.
The first near points on the characteristic cone $\Sigma$ and the
second near any closed characteristic. 

\subsection{\label{subsec:Normal-form-near-Sigma}Normal form near $\Sigma$}

Choose the canonical quasi-contact form $a_{g^{E}}$ \prettyref{eq:Popp one form}
defining $E$ and let $x\in X$. As before one then has a system of
local Darboux coordinates $\left(x_{0},x_{1},x_{2},x_{3}\right)$
centered at $x$ such that 
\begin{align}
\hat{\rho}a_{g^{E}} & =\frac{1}{2}\left[dx_{3}+x_{1}dx_{2}-x_{2}dx_{1}\right]\label{eq: quasi-contact normal form}\\
U_{0}=Z & =\partial_{x_{0}}.\label{eq:characterictic normal form}
\end{align}
The distribution $E$ is locally generated by the vector fields $U_{0}$,
$U_{1}=\partial_{x_{1}}+x_{2}\partial_{x_{3}}$ and $U_{2}=\partial_{x_{2}}-x_{1}\partial_{x_{3}}$
and we let $g_{ij}$ denote the components of the metric in this basis.
Further let $X,Y\in\left(L^{E}\right)^{\perp}$ be orthonormal. The
relations $da_{g^{E}}\left(X,Y\right)=\hat{\rho}da_{g^{E}}\left(U_{1},U_{2}\right)=1$
imply the existence of locally defined functions $\delta_{1}$, $\delta_{2}$
and $\Lambda=\begin{bmatrix} & -1\\
1
\end{bmatrix}\begin{bmatrix}\alpha & \gamma\\
\gamma & \beta
\end{bmatrix}$ with the latter taking values in $\mathfrak{sp}\left(2\right)$ such
that 
\begin{align*}
Z & =U_{0}\\
\begin{bmatrix}X\\
Y
\end{bmatrix} & =\begin{bmatrix}\delta_{1}U_{0}\\
\delta_{2}U_{0}
\end{bmatrix}+\hat{\rho}^{1/2}e^{\Lambda}\begin{bmatrix}U_{1}\\
U_{2}
\end{bmatrix}.
\end{align*}
The symbol of the Laplacian is then calculated 
\begin{align*}
\sigma\left(\Delta_{g^{E},\mu}\right) & =\sigma\left(X\right)^{2}+\sigma\left(Y\right)^{2}+\sigma\left(Z\right)^{2}\\
 & =\left(1+\delta_{1}^{2}+\delta_{2}^{2}\right)\xi_{0}^{2}+2\xi_{0}\hat{\rho}^{1/2}\begin{bmatrix}\delta_{1} & \delta_{2}\end{bmatrix}e^{\Lambda}\begin{bmatrix}\eta_{1}\\
\eta_{2}
\end{bmatrix}+\hat{\rho}\begin{bmatrix}\eta_{1} & \eta_{2}\end{bmatrix}e^{\Lambda^{t}}e^{\Lambda}\begin{bmatrix}\eta_{1}\\
\eta_{2}
\end{bmatrix}
\end{align*}
with $\eta_{0},\eta_{1},\eta_{2}$ denoting the symbols 
\begin{align*}
\eta_{0} & =\sigma\left(U_{0}\right)=\xi_{0}\\
\eta_{1} & =\sigma\left(U_{1}\right)=\xi_{1}+x_{2}\xi_{3}\\
\eta_{2} & =\sigma\left(U_{2}\right)=\xi_{2}-x_{1}\xi_{3},
\end{align*}
in terms of the induced coordinates $\left(x,\xi\right)$ on the cotangent
bundle. The characteristic variety or vanishing locus of the symbol
is given by 
\begin{align*}
\Sigma & =\left\{ \left(x,sa_{g^{E}}\left(x\right)\right)|s\in\mathbb{R}\right\} .
\end{align*}
Now if $\left(x,\xi\right)\in\left(\Sigma\setminus0\right)\cap\pi^{-1}\left(x\right)$,
we clearly have from \prettyref{eq: quasi-contact normal form} that
$\left(x_{0},x_{1},x_{2},x_{3};\xi_{0},\xi_{1},\xi_{2}\right)=0$
while $\xi_{3}\neq0$. We may assume 
\begin{equation}
\left(x,\xi\right)=c\underbrace{\left(0,0,0,0;0,0,0,1\right)}_{=p_{0}},\label{eq:model point}
\end{equation}
$c>0$, is a positive homogeneous multiple of the given point. The
homogeneous coordinates $x_{j}$, $\hat{\xi_{j}}=\frac{\xi_{j}}{\xi_{3}}$,
$0\leq j\leq2$, are then well defined on $C\setminus0$ for a conic
neighborhood $C$ of $\left(x,\xi\right)$. We set 
\begin{equation}
f_{0}=\xi_{3}\left(x_{1}x_{2}+\hat{\xi}_{1}\hat{\xi}_{2}\right)\label{eq:first Hamiltonian}
\end{equation}
and compute 
\begin{eqnarray}
e^{\frac{\pi}{4}H_{f_{0}}}\left(x_{1},\hat{\xi}_{1};x_{2},\hat{\xi}_{2}\right) & = & \left(\frac{x_{1}+\hat{\xi}_{2}}{\sqrt{2}},\frac{-x_{2}+\hat{\xi}_{1}}{\sqrt{2}};\frac{x_{2}+\hat{\xi}_{1}}{\sqrt{2}},\frac{-x_{1}+\hat{\xi}_{2}}{\sqrt{2}}\right),\nonumber \\
e^{\frac{\pi}{4}H_{f_{0}}}\left(x_{0},x_{3};\xi_{0},\xi_{3}\right) & = & \left(x_{0},x_{3}+\frac{1}{2}\left(x_{1}x_{2}-\hat{\xi}_{1}\hat{\xi}_{2}\right);\xi_{0},\xi_{3}\right).\label{eq:first Hamiltonian flow}
\end{eqnarray}
We further compute
\begin{align*}
\left(e^{\frac{\pi}{4}H_{f_{0}}}\right)^{*}\sigma\left(\Delta_{g^{E},\mu}\right) & =\xi_{3}^{2}\left[a_{0}\hat{\xi}_{0}^{2}+\hat{\xi}_{0}B_{0}\begin{bmatrix}\hat{\xi}_{1}\\
x_{1}
\end{bmatrix}+2\hat{\rho}\begin{bmatrix}\hat{\xi}_{1} & x_{1}\end{bmatrix}e^{\Lambda_{0}^{t}}e^{\Lambda_{0}}\begin{bmatrix}\hat{\xi}_{1}\\
x_{1}
\end{bmatrix}\right]
\end{align*}
where $a_{0}=\left(e^{\frac{\pi}{4}H_{f_{0}}}\right)^{*}\left(1+\delta_{1}^{2}+\delta_{2}^{2}\right)$,
$B_{0}=2\sqrt{2}\left(e^{\frac{\pi}{4}H_{f_{0}}}\right)^{*}\hat{\rho}^{1/2}\begin{bmatrix}\delta_{1} & \delta_{2}\end{bmatrix}e^{\Lambda}$
and $\Lambda_{0}=\left(e^{\frac{\pi}{4}H_{f_{0}}}\right)^{*}\Lambda$.

Next denote by $O_{\Sigma}\left(k\right)$ homogeneous (of degree
2) functions on $T^{*}X$ which vanish to order $k$ along $\Sigma=\left\{ x_{1}=\hat{\xi}_{1}=\hat{\xi}_{0}=0\right\} $.
We also denote by $O_{\Sigma}\left(k\right)$ the Weyl quantizations
on $\mathbb{R}^{4}$ of such symbols. A Taylor expansion gives
\begin{align*}
\left(e^{\frac{\pi}{4}H_{f_{0}}}\right)^{*}\sigma\left(\Delta_{g^{E},\mu}\right) & =\xi_{3}^{2}\left[\bar{a}\hat{\xi}_{0}^{2}+\hat{\xi}_{0}\bar{B}\begin{bmatrix}\hat{\xi}_{1}\\
x_{1}
\end{bmatrix}+2\hat{\rho}\begin{bmatrix}\hat{\xi}_{1} & x_{1}\end{bmatrix}e^{\bar{\Lambda}^{t}}e^{\bar{\Lambda}}\begin{bmatrix}\hat{\xi}_{1}\\
x_{1}
\end{bmatrix}\right]\\
 & \qquad\qquad+O_{\Sigma}\left(3\right)
\end{align*}
where $\bar{a}>0,\bar{B}$ and $\bar{\Lambda}\in\mathfrak{sp}\left(2\right)$
may now be considered as functions of $\left(x_{0},x_{2},x_{3};\xi_{2}\right)$.
Next we consider another function $f_{1}$ of the form 
\[
f_{1}=\frac{\xi_{3}^{2}}{2}\begin{bmatrix}\hat{\xi}_{1} & x_{1}\end{bmatrix}\begin{bmatrix} & -1\\
1
\end{bmatrix}\bar{\Lambda}\begin{bmatrix}\hat{\xi}_{1}\\
x_{1}
\end{bmatrix}.
\]
and compute
\begin{eqnarray*}
\left(e^{H_{f_{1}}}\right)^{*}\xi_{3}\begin{bmatrix}\hat{\xi}_{0}\\
\hat{\xi}_{1}\\
x_{1}
\end{bmatrix} & = & \xi_{3}\begin{bmatrix}\hat{\xi}_{0}\\
e^{-\bar{\Lambda}}\begin{bmatrix}\hat{\xi}_{1}\\
x_{1}
\end{bmatrix}
\end{bmatrix}+O_{\Sigma}\left(2\right).
\end{eqnarray*}
Following this we may further compute 
\[
\left(e^{H_{f_{1}}}\right)^{*}\left(e^{\frac{\pi}{4}H_{f_{0}}}\right)^{*}\sigma\left(\Delta_{g^{E},\mu}\right)=\xi_{3}^{2}\left[a_{1}\hat{\xi}_{0}^{2}+b_{1}\hat{\xi}_{0}x_{1}+c_{1}\hat{\xi}_{0}\hat{\xi}_{1}+2\hat{\rho}\left(x_{1}^{2}+\hat{\xi}_{1}^{2}\right)\right]+O_{\Sigma}\left(3\right)
\]
for some functions $a_{1}>0,b_{1}$ and $c_{1}$ of $\left(x_{0},x_{2},x_{3};\xi_{2}\right)$.
By a symplectic change of coordinates in the $\left(\underline{x},\underline{\xi}\right)$
variables we may set $a_{1}=1$ following which 
\begin{equation}
\kappa_{0}^{*}\sigma\left(\Delta_{g^{E},\mu}\right)=\xi_{0}^{2}+2\xi_{3}^{2}\hat{\rho}\left(x_{1}^{2}+\hat{\xi}_{1}^{2}\right)+O_{\Sigma}\left(3\right)\label{eq: prelim norm. form}
\end{equation}
for $\kappa_{0}=\left(e^{H_{f_{2}}}\right)\left(e^{H_{f_{1}}}\right)\left(e^{\frac{\pi}{4}H_{f_{0}}}\right)$
with $f_{2}=\frac{1}{2}\left[c_{1}\xi_{0}x_{1}-b_{1}\xi_{0}\xi_{1}\right]$.
Here $\rho=\xi_{3}\hat{\rho}$, $\hat{\rho}=\hat{\rho}\left(x_{0},x_{2},x_{3};\xi_{2}\right)$
is homogeneous of degree one, and is identifiable with the only positive
eigenvalue of the fundamental matrix \prettyref{eq:trace fund. matrix}. 

Next we claim that for some Hamiltonian diffeomorphism $\kappa_{1}=e^{H_{\xi_{3}f_{3}}}$,
$f_{3}\left(x;\hat{\xi}\right)\in O_{\Sigma}\left(1\right)$, and
$\hat{\xi}_{0}$-independent function $R\left(x_{1}^{2}+\hat{\xi}_{1}^{2};x_{0},x_{2},\hat{x}_{3};\hat{\xi}_{2}\right)\in O_{\Sigma}\left(4\right)$,
$\left[R,\Omega\right]=0$, we have 
\begin{equation}
\kappa_{1}^{*}\kappa_{0}^{*}\sigma\left(\Delta_{g^{E},\mu}\right)=\xi_{3}^{2}\left[\hat{\xi}_{0}^{2}+2\hat{\rho}\left(x_{1}^{2}+\hat{\xi}_{1}^{2}\right)\right]+R.\label{eq: symbolic normal form}
\end{equation}
To this end, we first define the complex variables $z_{1}=\frac{1}{\sqrt{2}}\left(x_{1}+i\hat{\xi}_{1}\right)$,
$\bar{z}_{1}=\frac{1}{\sqrt{2}}\left(x_{1}-i\hat{\xi}_{1}\right)$
and a grading on monomials in the variables $\left(\hat{\xi}_{0},z_{1},\bar{z}_{1}\right)$
via $\textrm{gr}\left(\hat{\xi}_{0}^{a}z_{1}^{b}\bar{z}_{1}^{c}\right)=2a+b+c$.
Further define by $\mathcal{O}_{\Sigma}\left(k\right)$ the set of
homogeneous (of degree 2) functions defined near $\Sigma$ whose Taylor
series involves monomials of grading at least $k$. We first prove
that for each $N\geq3$ there exists $g_{N}\left(x;\hat{\xi}\right)\in\mathcal{O}_{\Sigma}\left(1\right)$,
$R_{N}\left(x_{1}^{2}+\hat{\xi}_{1}^{2};x_{0},x_{2},\hat{x}_{3};\hat{\xi}_{2}\right)\in\mathcal{O}_{\Sigma}\left(4\right)$
such that 
\begin{align}
\left(e^{H_{\xi_{3}^{-1}g_{N}}}\right)^{*}\kappa_{0}^{*}\sigma\left(\Delta_{g^{E},\mu}\right) & =\xi_{3}^{2}\left[\hat{\xi}_{0}^{2}+2\hat{\rho}\left(x_{1}^{2}+\hat{\xi}_{1}^{2}\right)\right]+R_{N}+\mathcal{O}_{\Sigma}\left(N\right)\nonumber \\
g_{N}-g_{N-1} & =\mathcal{O}_{\Sigma}\left(N-3\right).\nonumber \\
R_{N}-R_{N-1} & =\mathcal{O}_{\Sigma}\left(N\right).\label{eq:ind. hyp.}
\end{align}
The case $N=3$ is \prettyref{eq: prelim norm. form}. To complete
the induction step write 
\begin{equation}
\left(e^{H_{\xi_{3}^{-1}g_{N}}}\right)^{*}\kappa_{0}^{*}\sigma\left(\Delta_{g^{E},\mu}\right)=\xi_{3}^{2}\left[\hat{\xi}_{0}^{2}+2\hat{\rho}\left(x_{1}^{2}+\hat{\xi}_{1}^{2}\right)\right]+R_{N}+\xi_{3}^{2}\sum_{\begin{subarray}{l}
2a+b+c=N\end{subarray}}r_{abc}\hat{\xi}_{0}^{a}z_{1}^{b}\bar{z}_{1}^{c}+\mathcal{O}_{\Sigma}\left(N+1\right)\label{eq: ind step}
\end{equation}
for complex functions $r_{abc}\left(x_{0},x_{2},\hat{x}_{3};\hat{\xi}_{2}\right)$
satisfying $\bar{r}_{abc}=r_{acb}$. Define
\begin{align*}
g_{N+1} & =g_{N}+\xi_{3}^{2}\left[\sum_{\begin{subarray}{l}
2a+b+c=N\\
b\neq c
\end{subarray}}s_{abc}\hat{\xi}_{0}^{a}z_{1}^{b}\bar{z}_{1}^{c}+\sum_{\begin{subarray}{l}
2a+2b=N-2\end{subarray}}s_{abb}\hat{\xi}_{0}^{a}\left(z_{1}\bar{z}_{1}\right)^{b}\right],\\
s_{abc} & =\frac{1}{4i\left(b-c\right)\hat{\rho}}r_{abc};\quad b\neq c,\\
s_{abb} & =-\frac{1}{2}\int_{0}^{x_{0}}r_{\left(a-1\right)bb};\quad a\geq1.
\end{align*}
A simple computation from \prettyref{eq: ind step} then gives
\begin{align*}
\left(e^{H_{\xi_{3}^{-1}g_{N+1}}}\right)^{*}\kappa_{0}^{*}\sigma\left(\Delta_{g^{E},\mu}\right) & =\xi_{3}^{2}\left[\hat{\xi}_{0}^{2}+2\hat{\rho}\underbrace{\left(x_{1}^{2}+\hat{\xi}_{1}^{2}\right)}_{=\xi_{3}^{-1}\Omega}+R_{N+1}\right]+O_{\Sigma}\left(N+1\right)\\
R_{N+1} & =R_{N}+\xi_{3}^{2}r_{0\frac{N}{2}\frac{N}{2}}\left(z_{1}\bar{z}_{1}\right)^{\frac{N}{2}}
\end{align*}
where the term involving $\frac{N}{2}$ above is understood to be
zero for $N$ odd. This completes the induction step. An application
of Borel's lemma then gives $g\left(x;\hat{\xi}\right)\in\mathcal{O}_{\Sigma}\left(1\right)$,
$R\left(x_{1}^{2}+\hat{\xi}_{1}^{2};x_{0},x_{2},\hat{x}_{3};\hat{\xi}_{2}\right)\in\mathcal{O}_{\Sigma}\left(4\right)$,
$R_{\infty}\in\mathcal{O}_{\Sigma}\left(\infty\right)$ such that
\begin{equation}
\left(e^{H_{\xi_{3}^{-1}g}}\right)^{*}\kappa_{0}^{*}\sigma\left(\Delta_{g^{E},\mu}\right)=\xi_{3}^{2}\underbrace{\left[\hat{\xi}_{0}^{2}+2\hat{\rho}\left(x_{1}^{2}+\hat{\xi}_{1}^{2}\right)+\xi_{3}^{-2}R\right]}_{=\sigma_{0}}+\xi_{3}^{2}R_{\infty}.\label{eq:symbolic normal form}
\end{equation}
We shall now eliminate the last infinite order error term $\xi_{3}^{2}R_{\infty}$
above by the following lemma.
\begin{lem}
\label{lem:infinite error removal}There exists a smooth, homogeneous
of degree one function $f_{\infty}$ defined in a conic neighborhood
of $p_{0}$ \prettyref{eq:model point} satisfying 
\begin{equation}
\left(e^{H_{f_{\infty}}}\right)^{*}\left(\xi_{3}^{2}\sigma_{0}+\xi_{3}^{2}R_{\infty}\right)=\xi_{3}^{2}\sigma_{0}.\label{eq:removing inf. error}
\end{equation}
\end{lem}
\begin{proof}
Without loss of generality assume that the conic neighborhood in which
\prettyref{eq: symbolic normal form} holds to be of the form $C_{\varepsilon}=\left\{ \left|\left(x_{0},x_{1},x_{2},x_{3},\hat{\xi}_{0},\hat{\xi}_{1},\hat{\xi}_{2}\right)\right|\leq\varepsilon\right\} $
for some $\varepsilon>0$. Next with $\chi\in C_{c}^{\infty}\left(-1,1\right)$
with $\chi=1$ on $\left[-\frac{1}{2},\frac{1}{2}\right]$, define
the microlocal cutoff $\chi_{\varepsilon}\coloneqq\chi\left(\frac{\left|\left(x_{0},x_{1},x_{2},x_{3},\hat{\xi}_{0},\hat{\xi}_{1},\hat{\xi}_{2}\right)\right|}{\varepsilon}\right)$.
Further define the function 
\begin{align}
\tilde{\sigma}_{0}\coloneqq\sigma_{0}+\chi_{\varepsilon}R_{\infty} & =\begin{cases}
\sigma_{0}+R_{\infty}; & \textrm{on }C_{\varepsilon/2}\\
\sigma_{0}; & \textrm{on }C_{\varepsilon}^{c}
\end{cases}\nonumber \\
\textrm{satisfying}\quad\xi_{3}^{2}\left(\tilde{\sigma}_{0}-\sigma_{0}\right) & =\mathcal{O}_{\Sigma}\left(\infty\right).\label{eq:def modified hamiltonian}
\end{align}
We may then compute the Hamilton vector field 
\[
H_{\xi_{3}\tilde{\sigma}_{0}^{1/2}}=\tilde{\sigma}_{0}^{-1/2}\left[2\hat{\xi}_{0}\partial_{x_{0}}+2\xi_{3}^{-1}\Omega H_{\rho}+2\hat{\rho}H_{\Omega}+\xi_{3}^{-1}H_{\xi_{3}^{2}\left(R+R_{\infty}\right)}\right]
\]
which is well-defined on $\left\{ \tilde{\sigma}_{0}\neq0\right\} =T^{*}\mathbb{R}^{4}\setminus\Sigma$.
From $R+R_{\infty}=O\left(\tilde{\sigma}_{0}^{2}\right)$ one may
calculate 
\begin{align*}
e^{tH_{\xi_{3}\tilde{\sigma}_{0}^{1/2}}}\left(x_{0},x_{1},x_{2},x_{3},\hat{\xi}_{0},\hat{\xi}_{1},\hat{\xi}_{2}\right) & \coloneqq\left(x_{0}\left(t\right),x_{1}\left(t\right),x_{2}\left(t\right),x_{3}\left(t\right),\hat{\xi}_{0}\left(t\right),\hat{\xi}_{1}\left(t\right),\hat{\xi}_{2}\left(t\right)\right),\quad\textrm{ with }\\
x_{0}\left(t\right) & =x_{0}+2t\tilde{\sigma}_{0}^{-1/2}\hat{\xi}_{0}+O\left(\tilde{\sigma}_{0}^{3/2}\right)\\
x_{3}\left(t\right) & =x_{3}+2t\tilde{\sigma}_{0}^{-1/2}\xi_{3}^{-1}\Omega+O\left(\tilde{\sigma}_{0}^{3/2}\right).
\end{align*}
The above shows that there exists a uniform $c_{1}$ such that any
point $p\in C_{\varepsilon}$ flows out of the cone 
\begin{equation}
e^{tH_{\xi_{3}\tilde{\sigma}_{0}^{1/2}}}p\notin C_{\varepsilon}\textrm{ for time }t>c_{1}\tilde{\sigma}_{0}\left(p\right)^{-1/2}.\label{eq: finite time exit}
\end{equation}
Outside of the cone $C_{\varepsilon}$ the flows of $e^{tH_{\xi_{3}\tilde{\sigma}_{0}^{1/2}}}$
and $e^{tH_{\xi_{3}\sigma_{0}^{1/2}}}$agree by \prettyref{eq:def modified hamiltonian}. 

Now we define the symplectomorphism 
\begin{align*}
\kappa_{\infty}: & T^{*}\mathbb{R}^{4}\setminus\Sigma\rightarrow T^{*}\mathbb{R}^{4}\setminus\Sigma\\
\kappa_{\infty}\coloneqq & \lim_{t\rightarrow\infty}e^{-tH_{\xi_{3}\sigma_{0}^{1/2}}}\circ e^{tH_{\xi_{3}\tilde{\sigma}_{0}^{1/2}}}.
\end{align*}
The limit exists, and is in fact attained in finite time, since 
\[
e^{-t'H_{\xi_{3}\sigma_{0}^{1/2}}}\circ e^{t'H_{\xi_{3}\tilde{\sigma}_{0}^{1/2}}}p=e^{-t'H_{\xi_{3}\sigma_{0}^{1/2}}}\circ e^{\left(t'-t\right)H_{\xi_{3}\tilde{\sigma}_{0}^{1/2}}}\circ e^{tH_{\xi_{3}\tilde{\sigma}_{0}^{1/2}}}p=e^{-tH_{\xi_{3}\sigma_{0}^{1/2}}}\circ e^{tH_{\xi_{3}\tilde{\sigma}_{0}^{1/2}}}p
\]
 for $t'>t>c_{1}\tilde{\sigma}_{0}\left(p\right)^{-1/2}$ using \prettyref{eq: finite time exit}
and the fact that $e^{tH_{\xi_{3}\tilde{\sigma}_{0}^{1/2}}}=e^{tH_{\xi_{3}\sigma_{0}^{1/2}}}$
outside $C_{\varepsilon}$. It is thus a Hamiltonian symplectomorphism
$\kappa_{\infty}=e^{H_{f_{\infty}}}$ and clearly satisfies 
\begin{equation}
\left(e^{H_{f_{\infty}}}\right)_{*}H_{\xi_{3}\tilde{\sigma}_{0}^{1/2}}=H_{\xi_{3}\sigma_{0}^{1/2}}\label{eq:transf. Hamilton v.fields}
\end{equation}
 by definition. Finally to prove that it extends to the characteristic
variety, first define $\tilde{x}_{0}\left(t\right)\coloneqq\left(e^{-tH_{\xi_{3}\sigma_{0}^{1/2}}}\circ e^{tH_{\xi_{3}\tilde{\sigma}_{0}^{1/2}}}\right)^{*}x_{0}$
which equals $\kappa_{\infty}^{*}x_{0}$ on $T^{*}\mathbb{R}^{4}\setminus\Sigma$
for $t>c_{1}\tilde{\sigma}_{0}\left(p\right)^{-1/2}$. We then compute
the time derivative 
\begin{align*}
\frac{d}{dt}\tilde{x}_{0}\left(t\right) & =\frac{d}{dt}\left(e^{tH_{\xi_{3}\tilde{\sigma}_{0}^{1/2}}}\right)^{*}\left(e^{-tH_{\xi_{3}\sigma_{0}^{1/2}}}\right)^{*}x_{0}\\
 & =\left\{ \xi_{3}\tilde{\sigma}_{0}^{1/2}-\left(e^{tH_{\xi_{3}\tilde{\sigma}_{0}^{1/2}}}\right)^{*}\xi_{3}\sigma_{0}^{1/2},\tilde{x}_{0}\left(t\right)\right\} \\
 & =\left\{ \left(e^{tH_{\xi_{3}\tilde{\sigma}_{0}^{1/2}}}\right)^{*}\left(\xi_{3}\tilde{\sigma}_{0}^{1/2}-\xi_{3}\sigma_{0}^{1/2}\right),\tilde{x}_{0}\left(t\right)\right\} =\xi_{3}^{-2}\mathcal{O}_{\Sigma}\left(\infty\right)
\end{align*}
uniformly on compact intervals of time following \prettyref{eq:def modified hamiltonian}.
It now follows that the function $\kappa_{\infty}^{*}x_{0}=x_{0}+\int_{0}^{c_{1}\tilde{\sigma}_{0}\left(p\right)^{-1/2}}\frac{d}{dt}\tilde{x}_{0}\left(t\right)dt$
extends smoothly by the identity to the characteristic variety. A
similar argument for the other coordinate functions along with \prettyref{eq:transf. Hamilton v.fields}
completes the proof.
\end{proof}
The proof of the lemma above follows the 'scattering trick' of Nelson
\cite{Nelson-book69,Colin-de-Verdiere-Hillairet-TrelatIg}; as already
pointed out in \cite[Sec. 5]{Colin-de-Verdiere-Hillairet-TrelatI}
its requisite analog is missing from \cite{Melrose-hypoelliptic}.

We now prove a Birkhoff normal form for the total symbol of $\Delta_{g^{E},\mu}$.
Below let $C_{\kappa}\subset T^{*}X\times T^{*}\mathbb{R}^{4}$ denote
the canonical relation associated to the symplectomorphism $\kappa\coloneqq\kappa_{0}\circ\kappa_{1}$
in \prettyref{eq: symbolic normal form} and the pullback $\left(\kappa^{*}\rho\right)\left(x_{0},x_{2},\hat{x}_{3};\xi_{2},\xi_{3}\right)$
by the same notation $\rho\left(x_{0},x_{2},\hat{x}_{3};\xi_{2},\xi_{3}\right)$.
\begin{thm}
\label{prop: normal form near pt}There exists a Fourier integral
operator $U\in I_{\textrm{cl}}^{0}\left(X,\mathbb{R}^{4};C_{\kappa}\right)$
and $\xi_{0}$-independent symbols $R\left(x_{1}^{2}+\hat{\xi}_{1}^{2};x_{0},x_{2},\hat{x}_{3};\xi_{2},\xi_{3}\right)\in O_{\Sigma}\left(4\right)$,
$r\left(x_{1}^{2}+\hat{\xi}_{1}^{2};x_{0},x_{2},\hat{x}_{3};\xi_{2},\xi_{3}\right)\in S_{\textrm{cl}}^{0}$
satisfying 
\begin{align}
U\Delta_{g^{E},\mu}U^{*} & =\underbrace{\xi_{0}^{2}+2\rho\underbrace{\left(\xi_{3}x_{1}^{2}+\xi_{3}^{-1}\xi_{1}^{2}\right)}_{=\Omega}+R+r}_{\eqqcolon\Delta_{\rho,R,r}}+\Psi^{-\infty}\left(\mathbb{R}^{4}\right),\label{eq:normal form}
\end{align}
and $UU^{*}=1$ microlocally on some open conic neighborhood $C\supset\left(\Sigma\setminus0\right)\cap\pi^{-1}\left(x\right)$.
\end{thm}
\begin{proof}
If $U_{1}:L^{2}\left(X\right)\rightarrow L^{2}\left(\mathbb{R}^{4}\right)$
denotes a unitary Fourier integral operator quantizing the symplectomorphism
$\kappa_{0}\circ e^{H_{\xi_{3}g}}\circ e^{H_{f_{\infty}}}$ in \prettyref{eq: symbolic normal form},
\prettyref{eq:removing inf. error} one has 
\[
\sigma\left(U_{1}\Delta_{g^{E},\mu}U_{1}^{*}\right)=\xi_{0}^{2}+2\xi_{3}\Omega+R
\]
by Egorov's theorem. By an argument of Weinstein (see Prop. 6 of \cite{Colin-de-Verdiere-Hillairet-TrelatI})
the quantization $U_{1}$ may be further chosen so that the sub-principal
symbol of the composition is zero and we may rewrite 
\begin{equation}
U_{1}\Delta_{g^{E},\mu}U_{1}^{*}=\xi_{0}^{2}+\xi_{3}\Omega+R+\Psi^{0}\left(\mathbb{R}^{4}\right)\label{eq:preliminary normal form}
\end{equation}
at the operator level; we drop the Weyl quantization symbol above
for simplicity. 

Next we prove by induction that $\forall N\geq0$, there exists $r_{N}\left(x_{1}^{2}+\hat{\xi}_{1}^{2};x_{0},x_{2},\hat{x}_{3};\xi_{2},\xi_{3}\right)\in S_{\textrm{cl}}^{0}$
and Fourier integral operator $e^{if_{N}}$, $f_{N}\in S^{-1}\left(\mathbb{R}^{4}\right)$
such that
\begin{align}
e^{if_{N}}U_{1}\Delta_{g^{E},\mu}U_{1}^{*}e^{-if_{N}} & =\xi_{0}^{2}+\xi_{3}\Omega+R+r_{N}+\Psi^{-N}\left(\mathbb{R}^{4}\right)\nonumber \\
f_{N+1}-f_{N} & \in S^{-N-1}\left(\mathbb{R}^{4}\right).\label{eq:inductive step}
\end{align}
The base case of the induction is \prettyref{eq:preliminary normal form}
with $f_{0}=r_{0}=0$. For the inductive step, we first write
\[
e^{if_{N}}U_{1}\Delta_{g^{E},\mu}U_{1}^{*}e^{-if_{N}}=\xi_{0}^{2}+2\xi_{3}\Omega+R+r_{N}+\xi_{3}^{-N}s_{N+1}\left(x_{0},x_{1},x_{2},\hat{x}_{3};\hat{\xi}_{0},\hat{\xi}_{1},\hat{\xi}_{2}\right)+\Psi^{-N-1}\left(\mathbb{R}^{4}\right).
\]
Then with $f_{N+1}=f_{N}+\xi_{3}^{-N-1}g_{N}\left(x_{0},x_{1},x_{2},\hat{x}_{3};\hat{\xi}_{0},\hat{\xi}_{1},\hat{\xi}_{2}\right)$
we compute 
\[
e^{if_{N+1}}U_{1}\Delta_{g^{E},\mu}U_{1}^{*}e^{-if_{N+1}}=\xi_{0}^{2}+2\xi_{3}\Omega+R+r_{N}+\xi_{3}^{-N}\left\{ s_{N+1}+2\hat{\xi}_{0}\partial_{x_{0}}g_{N}+\left(4+2\partial_{\varrho}R\right)\partial_{\theta}g_{N}\right\} +\Psi^{-N-1}\left(\mathbb{R}^{4}\right)
\]
in polar coordinates $\left(x_{1},\hat{\xi}_{1}\right)=\left(\varrho^{1/2}\cos\theta,\varrho^{1/2}\sin\theta\right)$.
We may then choose 
\begin{align*}
g_{N}\coloneqq & \frac{1}{2}\int_{0}^{x_{0}}\bar{s}_{N+1,1}\left(x_{0}',\hat{\xi}_{0},\varrho,\theta'\right)dx_{0}'\\
 & \;+\frac{1}{\left(4+2\partial_{\varrho}R\right)}\int_{0}^{\theta}d\theta'\left[s_{N+1}\left(x_{0}',\hat{\xi}_{0},\varrho,\theta'\right)-\bar{s}_{N+1}\left(x_{0}',\hat{\xi}_{0},\varrho,\theta'\right)\right]\\
\textrm{with}\quad\bar{s}_{N+1}\left(x_{0}',\hat{\xi}_{0},\varrho\right)\coloneqq & \frac{1}{2\pi}\int_{0}^{2\pi}d\theta'\,s_{N+1}\left(x_{0}',\hat{\xi}_{0},\varrho,\theta'\right)\\
\bar{s}_{N+1}\left(x_{0},\hat{\xi}_{0},\varrho\right)\coloneqq & \bar{s}_{N+1}\left(x_{0},0,\varrho\right)+\hat{\xi}_{0}\bar{s}_{N+1,1}\left(x_{0},\hat{\xi}_{0},\varrho\right)
\end{align*}
to compute 
\begin{align*}
e^{if_{N+1}}U_{1}\Delta_{g^{E},\mu}U_{1}^{*}e^{-if_{N+1}} & =\xi_{0}^{2}+2\xi_{3}\Omega+R+r_{N+1}+\Psi^{-N-1}\left(\mathbb{R}^{4}\right)\quad\textrm{with}\\
r_{N+1} & =r_{N}+\xi_{3}^{-N}\bar{s}_{N+1}\left(x_{0},0,\varrho\right).
\end{align*}
Finally an application of Borel's lemma following \prettyref{eq:inductive step}
completes the proof.
\end{proof}
In the normal form above since $R\in O_{\Sigma}\left(4\right)$ in
\prettyref{eq:normal form} we may write $R=\left(x_{1}^{2}+\hat{\xi}_{1}^{2}\right)R_{0}$,
with $R_{0}=R_{0}\left(x_{1}^{2}+\hat{\xi}_{1}^{2};x_{0},x_{2},\hat{x}_{3};\xi_{2},\xi_{3}\right)\in O_{\Sigma}\left(2\right)$.
Given $\varepsilon>0$, one may thus arrange 
\begin{align}
\left\Vert H_{l}RH_{l}^{*}\right\Vert _{L^{2}\left(\mathbb{R}^{3}\right)\rightarrow H^{-1}\left(\mathbb{R}^{3}\right)} & \leq\varepsilon\left(2l+1\right),\quad\forall l\in\mathbb{N}_{0},\label{eq: small cone arrangement 0}\\
\left\Vert B\left[\left(\xi_{0}^{2}+r\right)^{W},U\right]\right\Vert _{L^{2}\left(X\right)\rightarrow H^{-1}\left(\mathbb{R}^{4}\right)} & \leq\varepsilon\quad;\label{eq:small cone arrangement}
\end{align}
$\forall B\in\Psi^{0}\left(\mathbb{R}^{4}\right),\:WF\left(B\right)\subset C',\:WF\left(1-B\right)\subset C,$
on choosing $C\subset C'$ to be sufficiently small neighborhoods
of $\left(\Sigma\setminus0\right)\cap\pi^{-1}\left(x\right)$.

As an immediate corollary of the normal form \prettyref{prop: normal form near pt}
above we now prove the existence part of \prettyref{prop: Hamilton v field extension},
the invariance of $\Omega$ \prettyref{eq:symplectic commutation}
will be proved later in \prettyref{subsec:Invariance}.
\begin{proof}[Proof of \prettyref{prop: Hamilton v field extension}]
 At the symbolic level \prettyref{eq:normal form} reads $\sigma=\xi_{0}^{2}+2\rho\left(\xi_{3}x_{1}^{2}+\xi_{3}^{-1}\xi_{1}^{2}\right)+O_{\Sigma}\left(4\right)$.
This gives 
\begin{align}
H_{\sigma^{1/2}} & =\sigma^{-1/2}\left[\xi_{0}\partial_{x_{0}}-\left(\rho_{x_{0}}\Omega\right)\partial_{\xi_{0}}+\rho H_{\Omega}+O_{\Sigma}\left(2\right)\right]\nonumber \\
 & =\Xi_{0}\partial_{x_{0}}-\frac{1}{2}\rho^{-1}\rho_{x_{0}}\left(1-\Xi_{0}^{2}\right)\partial_{\Xi_{0}}+\sigma^{-1/2}\rho H_{\Omega}+\sigma^{-1/2}O_{\Sigma}\left(2\right)\label{eq:leading part Ham. vector field}
\end{align}
 in terms of the new coordinates $\left(x_{0},\Xi_{0}\coloneqq\frac{\xi_{0}}{\sqrt{\xi_{0}^{2}+2\rho\Omega}},x_{1},\ldots\right)$
and where the $O_{\Sigma}\left(2\right)$ term above denotes a vector
field vanishing to second order along $\Sigma$. The blowup and its
boundary are locally modeled by $\left[M;\Sigma\right]=\left\{ \xi_{0}^{2}+2\xi_{3}\rho\left(\xi_{3}x_{1}^{2}+\xi_{3}^{-1}\xi_{1}^{2}\right)\geq1\right\} $,
$SN\Sigma=\left\{ \xi_{0}^{2}+2\xi_{3}\rho\left(\xi_{3}x_{1}^{2}+\xi_{3}^{-1}\xi_{1}^{2}\right)=1\right\} $
while the function $\Xi_{0}\coloneqq\frac{\xi_{0}}{\sqrt{\xi_{0}^{2}+2\rho\Omega}}$
is identified with \prettyref{eq:invariant function blowup}. The
Hamiltonian vector field on the interior of the blowup is identified
with \prettyref{eq:leading part Ham. vector field} where $O_{\Sigma}\left(2\right)$
now denotes a vector field vanishing to second order near the boundary
of the blowup. One may then rewrite 
\[
H_{\sigma^{1/2}}=\sigma^{-1/2}\sigma\left(R\right)H_{\Omega}+\hat{Z}+O_{\Sigma}\left(1\right)
\]
with 
\begin{equation}
\hat{Z}\coloneqq\Xi_{0}\partial_{x_{0}}-\frac{1}{2}\rho^{-1}\rho_{x_{0}}\left(1-\Xi_{0}^{2}\right)\partial_{\Xi_{0}}+\sigma^{-1/2}\left[\rho-\sigma\left(R\right)\right]H_{\Omega}\label{eq:Hvfield restriction to boundary}
\end{equation}
The invariance property \prettyref{eq:invariance properties of flow}
follows from the definition, \prettyref{eq: QC Popp volume}, \prettyref{eq: quasi-contact normal form-1}
and \prettyref{eq:lift Popp measure} via the further identifications
\begin{align}
\left.H_{\Omega}\right|_{SN\Sigma} & =R_{0}\textrm{ and }\nonumber \\
\mu_{\textrm{Popp}}^{SNS^{*}\Sigma} & =\hat{\rho}^{-2}(1-\Xi_{0}^{2})d\Xi_{0}dx_{0}dx_{2}d\hat{x}_{3}d\hat{\xi}_{2}.\label{eq:identifications}
\end{align}
 
\end{proof}

\subsection{\label{subsec:Normal-form-near-closed-characteristic}Normal form
near a closed characteristic}

We next obtain a normal form for $\Delta_{g^{E},\mu}$ near a primitive
closed characteristic $\gamma$ assuming that the characteristic line
$L^{E}$ is volume preserving. Before proceeding we however note that
there exists a large class of quasi-contact structures where $L^{E}$
is not volume preserving as below.
\begin{example}
\label{exa:non-measure preserving example} Let $\left(Y,F\subset TY\right)$
be a contact manifold with contact vector field $H\in C^{\infty}\left(TY\right)$
, $\left(e^{tH}\right)_{*}F=F$. The mapping torus $X\coloneqq Y\times\left[0,1\right]_{x_{0}}/\left\{ \left(0,y\right)\sim\left(1,e^{H}\left(y\right)\right)\right\} $
carries the quasi-contact structure $E=F\oplus\mathbb{R}\left[\partial_{x_{0}}+H\right]$
whose characteristic line field is $L^{E}=\mathbb{R}\left[Z\right]=\mathbb{R}\left[\partial_{x_{0}}+H\right]$
(cf. \cite[Lemma 2.5]{Colin-Presas-Vogel-2018}). The Poincare section
$Y\times\left\{ 0\right\} $ however cannot carry an $H$ invariant
volume such as $\hat{a}_{g^{E}}\wedge d\hat{a}_{g^{E}}$ in the case
when the time one flow of $H$ is a strictly expanding/contracting
map on some region; say near one of its zeros. An explicit example
of such an $H$ is quite easily constructed; choose Darboux coordinates
on an $\varepsilon$-ball $B_{p}\left(\varepsilon\right)$ centered
at a point $p\in Y$ in which $F=\textrm{ker}\left(\underbrace{x_{1}dx_{2}-x_{2}dx_{1}+dx_{3}}_{=a_{0}}\right)$.
Letting $\chi\in C_{c}^{\infty}\left(\left[-1,1\right];\left[0,1\right]\right)$,
$\chi=1$ on $\left[-\frac{1}{2},\frac{1}{2}\right]$, define the
contact Hamiltonian vector field
\[
H=H_{\varphi}=\begin{cases}
\left(\varphi_{x_{1}}+x_{2}\varphi_{x_{3}}\right)\partial_{x_{2}}-\left(\varphi_{x_{2}}-x_{1}\varphi_{x_{3}}\right)\partial_{x_{1}}+\left(2\varphi-x_{1}\varphi_{x_{1}}-x_{2}\varphi_{x_{2}}\right)\partial_{x_{3}}; & x\in B_{p}\left(\varepsilon\right)\\
0; & x\notin B_{p}\left(\varepsilon\right)
\end{cases},
\]
 $\varphi\coloneqq x_{3}\chi\left(\frac{\left|\left(x_{1},x_{2},x_{3}\right)\right|}{\varepsilon}\right)$,
satisfying $L_{H_{\varphi}}a_{0}=2\varphi_{x_{3}}a_{0}$ on $B_{p}\left(\varepsilon\right)$,
and which has a strictly expanding time one flow near the origin where
$\left.H_{\varphi}\right|_{B_{p}\left(\varepsilon/2\right)}=x_{1}\partial_{x_{1}}+x_{2}\partial_{x_{2}}+2x_{3}\partial_{x_{3}}$.
\end{example}
We now show that in the volume preserving case, the normal structure
of $E$ is described as such a mapping torus of an $x_{3}$-independent
contact Hamiltonian $\varphi$ near a non-degenerate closed characteristic.
To this end, as mentioned before, in the volume preserving case one
has a $Z$-invariant defining one form $L_{Z}\left(\underbrace{\hat{\rho}_{Z}a_{g^{E}}}_{\eqqcolon\hat{a}_{g^{E}}}\right)=0$.
The linearized Poincare return map $P_{\gamma}$ of $Z$ is seen to
be symplectic on $\left(E/L^{E},d\hat{a}_{g^{E}}\right)$. We call
the characteristic elliptic if the eigenvalues of $P_{\gamma}$ are
of the form $e^{\pm i\alpha}$($2\pi>\alpha\geq0$) and (positive)
hyperbolic if of the form $e^{\pm\beta}$($\beta\geq0$). The characteristic
is said to be non-degenerate iff $\frac{\alpha}{2\pi}\notin\mathbb{Q}$
or $\beta\neq0$. For each $\gamma$ we then define the model quadratic
on $\mathbb{R}^{2}$ via
\begin{equation}
Q=\begin{cases}
\frac{\alpha}{2}\left(x_{1}^{2}+x_{2}^{2}\right); & \gamma\textrm{ is elliptic},\\
\beta x_{1}x_{2}; & \gamma\textrm{ is hyperbolic.}
\end{cases}\label{eq: model quadratic}
\end{equation}
We first begin describing the normal structure of the Popp form $\hat{a}_{g^{E}}$
near a non-degenerate $\gamma$. In the theorem below we let $\gamma^{0}\coloneqq S^{1}\times\left\{ 0\right\} \subset S_{x_{0}}^{1}\times\mathbb{R}^{3}$
and $T_{\gamma}$ the length of $\gamma$.
\begin{prop}
There exists a diffeomorphism $\kappa:\varOmega_{\gamma}^{0}\rightarrow\varOmega_{\gamma}$
between some neighborhood of $\gamma^{0}\subset\varOmega_{\gamma}^{0}$
and some neighborhood of the closed characteristic $\gamma\subset\varOmega_{\gamma}$
such that 
\begin{align}
\kappa^{*}\hat{a}_{g^{E}} & =\underbrace{\varphi dx_{0}+\frac{1}{2}\left[x_{1}dx_{2}-x_{2}dx_{1}+dx_{3}\right]}_{\eqqcolon a_{\varphi}}\label{eq: normal form contact a-1}\\
\left|\tilde{Z}\right| & \coloneqq\left|-\partial_{x_{0}}+\varphi_{x_{1}}\partial_{x_{2}}-\varphi_{x_{2}}\partial_{x_{1}}+\left[2\varphi-\left(x_{1}\varphi_{x_{1}}+x_{2}\varphi_{x_{2}}\right)\right]\partial_{x_{3}}\right|\nonumber \\
 & =T_{\gamma}+O\left(\left|\left(x_{1},x_{2}\right)\right|^{2}\right)+O\left(x_{3}\right)\label{eq:normal form generator}
\end{align}
modulo $O\left(Q^{\infty}\right)$. Here 
\begin{equation}
\varphi=\varphi\left(Q\right)=Q+O\left(Q^{2}\right)\label{eq:form of effective Hamiltonian}
\end{equation}
 above \prettyref{eq: normal form contact a-1} is a function on $\mathbb{R}^{2}$
of the quadratic \prettyref{eq: model quadratic} with linear term
$Q$.
\end{prop}
\begin{proof}
Choose a Poincare section $Y$ transversal to $Z$ through a point
$p\in\gamma$ with Poincare return map and return time functions $P_{Y}:Y\rightarrow Y$,
$T_{Y}:Y\rightarrow\mathbb{R}$. Having $P_{Y}=e^{\tilde{Z}}$; $\tilde{Z}=T_{Y}Z$,
we may compute 
\begin{align}
P_{Y}^{*}\hat{a}_{g^{E}}-\hat{a}_{g^{E}} & =\int_{0}^{1}\left(\mathcal{L}_{\tilde{Z}}\hat{a}_{g^{E}}\right)dt\nonumber \\
 & =\int_{0}^{1}\left(di_{\tilde{Z}}\hat{a}_{g^{E}}+i_{\tilde{Z}}d\hat{a}_{g^{E}}\right)dt\nonumber \\
 & =0.\label{eq: Popp form contactomorphism}
\end{align}
The one form $a$ is contact on $Y$ with contact hyperplane $F=TY\cap E$
and we choose a set of Darboux coordinates $\left(x_{1},x_{2},x_{3}\right)$
with $\left.\hat{a}_{g^{E}}\right|_{Y}=\frac{1}{2}\left[dx_{3}+x_{1}dx_{2}-x_{2}dx_{1}\right]$
as \prettyref{eq: quasi-contact normal form-1}. By \prettyref{eq: Popp form contactomorphism},
the return map $P_{Y}$ is a contactomorphism with its linearization
at $0$ being identified with $P_{\gamma}$. We now claim that such
a contactomorphism is given by 
\begin{align}
P_{Y} & =e^{H_{\varphi}},\quad\textrm{ with}\nonumber \\
H_{\varphi} & \coloneqq\varphi_{x_{1}}\partial_{x_{2}}-\varphi_{x_{2}}\partial_{x_{1}}+\left[-2\varphi+\left(x_{1}\varphi_{x_{1}}+x_{2}\varphi_{x_{2}}\right)\right]\partial_{x_{3}}\label{eq:contact Hamilton v. field}
\end{align}
under the non-degeneracy assumption. To see the above let $P_{Y}=\left(P_{1},P_{2},P_{3}\right)$.
Since the Reeb vector field is mapped to itself, $\partial_{x_{3}}=\frac{\partial P_{1}}{\partial x_{3}}\partial_{P_{1}}+\frac{\partial P_{2}}{\partial x_{3}}\partial_{P_{2}}+\frac{\partial P_{3}}{\partial x_{3}}\partial_{P_{3}}=\partial_{P_{3}}$
giving $\frac{\partial P_{1}}{\partial x_{3}}=\frac{\partial P_{2}}{\partial x_{3}}=0$
and thus $P_{1},P_{2}$ are independent of $x_{3}$. The map $P_{Y}^{0}\coloneqq\left(P_{1},P_{2}\right):\mathbb{R}^{2}\rightarrow\mathbb{R}^{2}$
is then symplectic with respect to $dx_{1}dx_{2}$; the eigenvalues
of its linearization are those of $P_{\gamma}^{+}$ not equal to $1$.
Thus $P_{Y}^{0}=e^{H_{\varphi}}$ for some $\varphi\left(Q\right)$
of the form \prettyref{eq:form of effective Hamiltonian} under the
degeneracy assumption. Next define $\left(P_{1}\left(t\right),P_{2}\left(t\right)\right)\coloneqq e^{tH_{\varphi}}\left(x_{1},x_{2}\right)$
and calculate 
\begin{align*}
\frac{d}{dt}\left(e^{tH_{\varphi}}\right)^{*}\frac{1}{2}\left[x_{1}dx_{2}-x_{2}dx_{1}\right] & =\mathcal{L}_{H_{\varphi}}\frac{1}{2}\left[x_{1}dx_{2}-x_{2}dx_{1}\right]\\
 & =\left(i_{H_{\varphi}}d+di_{H_{\varphi}}\right)\frac{1}{2}\left[x_{1}dx_{2}-x_{2}dx_{1}\right]\\
 & =d\left[-\varphi+\frac{1}{2}\left(x_{1}\varphi_{x_{1}}+x_{2}\varphi_{x_{2}}\right)\right]
\end{align*}
 to obtain 
\[
\left(e^{H_{\varphi}}\right)^{*}\frac{1}{2}\left[x_{1}dx_{2}-x_{2}dx_{1}\right]-\frac{1}{2}\left[x_{1}dx_{2}-x_{2}dx_{1}\right]=d\left[-\varphi+\frac{1}{2}\left(x_{1}\varphi_{x_{1}}+x_{2}\varphi_{x_{2}}\right)\right].
\]
 This gives 
\begin{align*}
0=P_{Y}^{*}\hat{a}_{g^{E}}-\hat{a}_{g^{E}} & =\frac{1}{2}d\left(P_{3}-x_{3}\right)+\left(e^{H_{\varphi}}\right)^{*}\frac{1}{2}\left[x_{1}dx_{2}-x_{2}dx_{1}\right]-\frac{1}{2}\left[x_{1}dx_{2}-x_{2}dx_{1}\right]\\
 & =\frac{1}{2}d\left(P_{3}-x_{3}\right)+d\left[-\varphi+\frac{1}{2}\left(x_{1}\varphi_{x_{1}}+x_{2}\varphi_{x_{2}}\right)\right]
\end{align*}
 and thus $P_{3}=x_{3}-2\varphi+\left(x_{1}\varphi_{x_{1}}+x_{2}\varphi_{x_{2}}\right)$
on knowing $P_{Y}\left(0\right)=0$, proving the claim \prettyref{eq:contact Hamilton v. field}.
Now, noting that Poincare map is also given via $P_{\Sigma}=e^{\tilde{Z}}$;
with 
\begin{equation}
\tilde{Z}=-\partial_{x_{0}}+H_{\varphi}\label{eq:generator charateristic line}
\end{equation}
satisfying $i_{\tilde{Z}}a_{\varphi}=i_{\tilde{Z}}da_{\varphi}=0$
for the model form $a_{\varphi}$ \prettyref{eq: normal form contact a-1}
proves \prettyref{eq: normal form contact a-1}. 

To prove \prettyref{eq:normal form generator}, first note 
\begin{equation}
\left|\tilde{Z}\right|=T_{Y}\label{eq:period equals norm}
\end{equation}
 and compute 
\begin{align*}
P_{Y}^{*}dT_{Y}-dT_{Y} & =\int_{0}^{1}\left(\mathcal{L}_{\tilde{Z}}dT_{Y}\right)dt\\
 & =\int_{0}^{1}d\tilde{Z}\left(T_{Y}\right)dt=0
\end{align*}
by definition; $T_{Y}$ is defined on a neighborhood of $\gamma$
using the flow of $\tilde{Z}$. This gives 
\[
P_{Y}^{*}T_{Y}=\left(e^{H_{\varphi}}\right)^{*}T_{Y}=T_{Y}
\]
 on knowing $P_{Y}\left(0\right)=0$. Comparing the coefficients in
the last equation using \prettyref{eq:form of effective Hamiltonian},
\prettyref{eq:contact Hamilton v. field} shows that the linear $\left(x_{1},x_{2}\right)$
terms in $T_{Y}$ must vanish under the non-degeneracy assumption.
\end{proof}
The distribution $E$ is now locally generated by the vector fields
$U_{0}=-\partial_{x_{0}}+2\varphi\partial_{x_{3}}$, $U_{1}=\partial_{x_{1}}+x_{2}\partial_{x_{3}}$
and $U_{2}=\partial_{x_{2}}-x_{1}\partial_{x_{3}}$. The generator
of the characteristic line \prettyref{eq:generator charateristic line}
maybe written 
\[
\tilde{Z}=U_{0}-\varphi_{x_{2}}U_{1}+\varphi_{x_{1}}U_{2}.
\]
We may again choose $X,Y\in\left(L^{E}\right)^{\perp}$ satisfying
$\left|X\right|=1$, $da_{g^{E}}\left(Y,X\right)=\hat{\rho}_{Z}da_{g^{E}}\left(U_{1},U_{2}\right)=1$
and 
\begin{align*}
\begin{bmatrix}X\\
Y
\end{bmatrix} & =\begin{bmatrix}\delta_{1}\tilde{Z}\\
\delta_{2}\tilde{Z}
\end{bmatrix}+\hat{\rho}_{Z}^{1/2}e^{\Lambda}\begin{bmatrix}U_{1}\\
U_{2}
\end{bmatrix}
\end{align*}
for some set of functions $\delta_{1},\delta_{2}$ and $\Lambda=\begin{bmatrix} & -1\\
1
\end{bmatrix}\begin{bmatrix}\alpha & \gamma\\
\gamma & \beta
\end{bmatrix}\in\mathfrak{sp}\left(2\right)$.

The symbol of the Laplacian is then calculated 
\begin{align*}
\sigma\left(\Delta_{g^{E},\mu}\right) & =\frac{1}{\left|\tilde{Z}\right|^{2}}\sigma\left(\tilde{Z}\right)^{2}+\sigma\left(X\right)^{2}+\sigma\left(Y\right)^{2}\\
 & =a_{0}\tilde{\eta}_{0}^{2}+2\hat{\rho}_{Z}^{1/2}\tilde{\eta}_{0}\begin{bmatrix}\delta_{1} & \delta_{2}\end{bmatrix}e^{\Lambda}\begin{bmatrix}\eta_{1}\\
\eta_{2}
\end{bmatrix}+\hat{\rho}_{Z}\begin{bmatrix}\eta_{1} & \eta_{2}\end{bmatrix}e^{\Lambda^{t}}e^{\Lambda}\begin{bmatrix}\eta_{1}\\
\eta_{2}
\end{bmatrix}+\xi_{3}^{2}O_{\Sigma}\left(2\right)O_{\gamma}\left(1\right).
\end{align*}
Here $\tilde{\eta}_{0},\eta_{1},\eta_{2}$ denote the symbols 
\begin{align*}
\tilde{\eta}_{0} & \coloneqq\sigma\left(\tilde{Z}\right)=-\xi_{0}+2\varphi\xi_{3}\\
\eta_{1} & \coloneqq\sigma\left(U_{1}\right)=\xi_{1}+x_{2}\xi_{3}\\
\eta_{2} & \coloneqq\sigma\left(U_{2}\right)=\xi_{2}-x_{1}\xi_{3},
\end{align*}
of the given vector fields while $O_{\Sigma}\left(k\right)$, $O_{\gamma}\left(k\right)$,
denote homogeneous degree zero symbols which vanish to order $k$
in the variables $\left(\xi_{3}^{-1}\tilde{\eta}_{0},\xi_{3}^{-1}\eta_{1},\xi_{3}^{-1}\eta_{2}\right)$
and $\left(x_{1},x_{2},x_{3}\right)$ respectively. 

Setting $f_{0}=\xi_{3}\left(x_{1}x_{2}+\hat{\xi}_{1}\hat{\xi}_{2}\right)$
as before, we again calculate
\begin{align*}
\left(e^{\frac{\pi}{4}H_{f_{0}}}\right)^{*}\sigma\left(\Delta_{g^{E},\mu}\right) & =\xi_{3}^{2}\left[a_{0}\left(-\hat{\xi}_{0}+2\bar{\varphi}\right)^{2}+2\hat{\rho}_{Z}^{1/2}\left(-\hat{\xi}_{0}+2\bar{\varphi}\right)\begin{bmatrix}\delta_{1} & \delta_{2}\end{bmatrix}e^{\Lambda}\begin{bmatrix}\hat{\xi}_{1}\\
x_{1}
\end{bmatrix}+2\hat{\rho}_{Z}\begin{bmatrix}\hat{\xi}_{1} & x_{1}\end{bmatrix}e^{\Lambda^{t}}e^{\Lambda}\begin{bmatrix}\hat{\xi}_{1}\\
x_{1}
\end{bmatrix}\right]\\
 & \qquad+\xi_{3}^{2}O_{\Sigma}\left(2\right)O_{\gamma}\left(1\right)+\xi_{3}^{2}O_{\Sigma}\left(3\right)
\end{align*}
where $\bar{\varphi}=$$a_{0},\delta_{1},\delta_{2},\Lambda$ are
functions of $\left(x_{0},x_{2},x_{3};\hat{\xi}_{2}\right)$ while
$O_{\Sigma}\left(k\right)$, $O_{\gamma}\left(k\right)$, denote homogeneous
degree zero symbols which vanish to order $k$ in the variables $\left(\hat{\xi}_{0}+2\varphi,x_{1},\hat{\xi}_{1}\right)$
and $\left(x_{2},x_{3};\hat{\xi}_{2}\right)$ respectively. 

Further, with $f_{1}$ of the form 
\[
f_{1}=\frac{\xi_{3}^{2}}{2}\begin{bmatrix}\hat{\xi}_{1} & x_{1}\end{bmatrix}\begin{bmatrix} & -1\\
1
\end{bmatrix}\bar{\Lambda}\begin{bmatrix}\hat{\xi}_{1}\\
x_{1}
\end{bmatrix}
\]
we compute
\[
\left(e^{H_{f_{1}}}\right)^{*}\xi_{3}\begin{bmatrix}\hat{\xi}_{1}\\
x_{1}
\end{bmatrix}=\xi_{3}e^{-\Lambda_{0}}\begin{bmatrix}\hat{\xi}_{1}\\
x_{1}
\end{bmatrix}+\xi_{3}O_{\Sigma}\left(2\right),
\]
giving 
\begin{align*}
\left(e^{H_{f_{1}}}\right)^{*}\left(e^{\frac{\pi}{4}H_{f_{0}}}\right)^{*}\sigma\left(\Delta_{g^{E},\mu}\right) & =\xi_{3}^{2}\left[a_{0}\left(-\hat{\xi}_{0}+2\bar{\varphi}\right)^{2}+\left(-\hat{\xi}_{0}+2\bar{\varphi}\right)\left(b_{0}x_{1}+c_{0}\hat{\xi}_{1}\right)+2\hat{\rho}_{Z}\left(x_{1}^{2}+\hat{\xi}_{1}^{2}\right)\right]\\
 & \qquad+\xi_{3}^{2}O_{\Sigma}\left(2\right)O_{\gamma}\left(2\right)+\xi_{3}^{2}O_{\Sigma}\left(3\right)
\end{align*}
for $\begin{bmatrix}\delta_{1}^{1}\\
\delta_{2}^{1}
\end{bmatrix}\in O_{\gamma}\left(1\right)$. Finally $f_{2}=\frac{1}{2}\left[c_{0}\xi_{0}x_{1}-b_{0}\xi_{0}\xi_{1}\right]$
we also have
\[
\kappa_{0}^{*}\sigma\left(\Delta_{g^{E},\mu}\right)=a_{0}\xi_{3}^{2}\left(-\hat{\xi}_{0}+2\bar{\varphi}\right)^{2}+2\hat{\rho}_{Z}\xi_{3}^{2}\left(x_{1}^{2}+\hat{\xi}_{1}^{2}\right)+\xi_{3}^{2}O_{\Sigma}\left(2\right)O_{\gamma}\left(2\right)+\xi_{3}^{2}O_{\Sigma}\left(3\right)
\]
for $\kappa_{0}=\left(e^{H_{f_{2}}}\right)\left(e^{H_{f_{1}}}\right)\left(e^{\frac{\pi}{4}H_{f_{0}}}\right)$
and for some $a_{0}\left(x_{0},x_{2},\hat{x}_{3};\hat{\xi}_{2}\right)>0$.
Finally, by another Hamiltonian diffeomorphism we may set $\hat{\rho}_{Z}=1$
and $a_{0}=a_{0}\left(x_{2},\hat{x}_{3};\hat{\xi}_{2}\right)$ independent
of $x_{0}$ and satisfying 
\[
a_{0}\left(0,0;0\right)=\frac{1}{T_{\gamma}^{2}}.
\]
Following the preliminary normal form above the rest of the normal
form procedure proceeds as in the previous section. We then first
have a Hamiltonian diffeomorphism $\kappa_{1}=e^{H_{\xi_{3}f_{3}}}$,
$f_{3}\left(x;\hat{\xi}\right)\in O_{\Sigma}\left(1\right)$, and
function $R\left(x_{1}^{2}+\hat{\xi}_{1}^{2};x_{0},x_{2},\hat{x}_{3};\hat{\xi}_{0}+\bar{\varphi},\hat{\xi}_{2}\right)\in\xi_{3}^{2}O_{\Sigma}\left(2\right)O_{\gamma}\left(2\right)+\xi_{3}^{2}O_{\Sigma}\left(3\right)$
such that
\begin{equation}
\kappa_{1}^{*}\kappa_{0}^{*}\sigma\left(\Delta_{g^{E},\mu}\right)=\xi_{3}^{2}\left[a_{0}\left(-\hat{\xi}_{0}+\bar{\varphi}\right)^{2}+2\left(x_{1}^{2}+\hat{\xi}_{1}^{2}\right)+R\right]+O_{\Sigma}\left(\infty\right).\label{eq: symbolic normal form-2-1}
\end{equation}
The normal form for the symbol is now given next.
\begin{thm}
\label{prop: normal form-1}There exists a Hamiltonian symplectomorphism
$\kappa:T^{*}\left(S_{x_{0}}^{1}\times\mathbb{R}^{3}\right)\rightarrow T^{*}X$
and symbol $R\left(x_{1}^{2}+\hat{\xi}_{1}^{2};x_{0},x_{2},\hat{x}_{3};-\hat{\xi}_{0}+\bar{\varphi},\hat{\xi}_{2}\right)\in\xi_{3}^{2}O_{\Sigma}\left(2\right)O_{\gamma}\left(2\right)+\xi_{3}^{2}O_{\Sigma}\left(3\right)$
\begin{align}
\kappa^{*}\sigma\left(\Delta_{g^{E},\mu}\right) & =\xi_{3}^{2}\left[a_{0}\left(-\hat{\xi}_{0}+\bar{\varphi}\right)^{2}+2\left(x_{1}^{2}+\hat{\xi}_{1}^{2}\right)+R\right]+O_{\Sigma}\left(\infty\right)\label{eq:normal form-1}
\end{align}
on some open conic neighborhood $C\supset\left(\Sigma\setminus0\right)\cap\pi^{-1}\left(\gamma\right)$.
\end{thm}
We refer to a (nondegenerate) closed characteristic $\gamma$ as flat
if there exists a normal form as above with $R=0$, $a_{0}=T_{\gamma}$
(constant). 

We next compute $\mathscr{L}_{\hat{Z}}$ the set of closed periods
of the vector field $\hat{Z}$, in both the volume preserving and
non-preserving cases. First note that in the volume preserving case
since $L_{Z}\hat{\rho}_{Z}=da_{g^{E}}\left(R,Z\right)\hat{\rho}_{Z}$
for some positive function $\hat{\rho}_{Z}$ one has 
\begin{align}
L_{Z}\left(\ln\hat{\rho}_{Z}\right) & =da_{g^{E}}\left(R,Z\right)\quad\textrm{ and hence }\label{eq:log derivative}\\
\int_{0}^{T_{\gamma}}dt\left(e^{tZ}\right)^{*}\left.da_{g^{E}}\left(R,Z\right)\right|_{\gamma} & =0\label{eq:invariant along charac.}
\end{align}
along any closed characteristic $\gamma$ with period $T_{\gamma}$.
Motivated by this we say that \textit{$L^{E}$ is volume} \textit{preserving
along $\gamma$ }iff the equation \prettyref{eq:invariant along charac.}
above holds. In this case we may define a unique positive function
$\hat{\rho}_{Z}^{\gamma}\in C^{\infty}\left(\gamma;\left(0,1\right]\right)$
satisfying \prettyref{eq:log derivative} along $\gamma$ and $\sup_{\gamma}\hat{\rho}_{Z}^{\gamma}=1$.
In the case when $L^{E}$ is globally volume preserving this would
equal $\hat{\rho}_{Z}^{\gamma}\coloneqq\frac{\left.\hat{\rho}_{Z}\right|_{\gamma}}{\sup_{\gamma}\hat{\rho}_{Z}}$
for any globally defined function satisfying \prettyref{eq:log derivative}.
Viewing $\hat{\rho}_{Z}^{\gamma}$ as a periodic function on $\mathbb{R}$
with period $T_{\gamma}$, we define $\hat{T}_{\gamma}>T_{\gamma}$
as the smallest positive number for which$\int_{0}^{\hat{T}_{\gamma}}\frac{1-\hat{\rho}_{Z}^{\gamma}}{1+\hat{\rho}_{Z}^{\gamma}}=T_{\gamma}$.
Here we use the convention that $\hat{T}_{\gamma}=\infty$ if $\hat{\rho}_{Z}^{\gamma}\equiv1$,
in which case $\left.L_{Z}\mu_{\textrm{Popp}}\right|_{\gamma}=0$.
Finally, we extend this definition to the case when $L^{E}$ is not
volume-preserving along $\gamma$ by simply setting $\hat{T}_{\gamma}=T_{\gamma}$.
Below we denote by $\mathbb{N}\left[I\right]$ the set of all positive
integer multiples of elements in any given interval $I\subset\mathbb{R}$.
We now have the following.
\begin{prop}
\label{prop:(Density-of-periods)}(Density of periods) The set of
periods 
\begin{equation}
\mathscr{L}_{\hat{Z}}=\bigcup_{\gamma\textrm{ closed characteristic}}\mathbb{N}\left[-\hat{T}_{\gamma},-T_{\gamma}\right]\cup\mathbb{N}\left[T_{\gamma},\hat{T}_{\gamma}\right]\label{eq:period spectrum lift of charac.}
\end{equation}
 In particular if $L_{Z}\mu_{\textrm{Popp}}=0$ along the shortest
closed characteristic, the set of periods 
\begin{equation}
\mathscr{L}_{\hat{Z}}=\left(-\infty,-T_{\textrm{abnormal}}^{E}\right]\cup\left[T_{\textrm{abnormal}}^{E},\infty\right).\label{eq:special period spectrum lift of charac.}
\end{equation}
Finally, if the shortest (nondegenerate) closed characteristic $\gamma$
is flat one has the density of normal periods
\begin{equation}
\mathscr{L}_{\textrm{normal}}\supset\left(-\infty,-T_{\textrm{abnormal}}^{E}\right]\cup\left[T_{\textrm{abnormal}}^{E},\infty\right).\label{eq: density of normal periods}
\end{equation}
\end{prop}
\begin{proof}
Clearly by \prettyref{eq: characteristic pushforward}, a closed integral
curve of $\hat{Z}$ lies over a closed characteristic; say $\gamma\left(t\right)\coloneqq e^{tZ}$
, $\gamma\left(T_{\gamma}\right)=\gamma\left(0\right)$. The restriction
of $\left(SN\Sigma/\mathbb{R}_{+}\right)/S^{1}$ to $\gamma$ is a
$\left[-1,1\right]_{\Xi_{0}}$ bundle  on which $\hat{Z}=\Xi_{0}\partial_{t}-\underbrace{\left.\frac{1}{2}da_{g^{E}}\left(R,Z\right)\right|_{\gamma}}_{\eqqcolon A\left(t\right)}\left(1-\Xi_{0}^{2}\right)\partial_{\Xi_{0}}$,
following the computation \prettyref{eq:Hvfield restriction to boundary},
which we may further view as a vector field on $\mathbb{R}_{t}\times\left[-1,1\right]_{\Xi_{0}}$
that is periodic in $t$. The flow of the above can be explicitly
computed
\begin{equation}
e^{t\hat{Z}}\left(0,\Xi_{0}\left(0\right)\right)=\left(\int_{0}^{t}\Xi_{0}\left(s\right)ds,\underbrace{\frac{1+\Xi_{0}\left(0\right)-\left(1-\Xi_{0}\left(0\right)\right)e^{-2\int_{0}^{t}A\left(s\right)ds}}{1+\Xi_{0}\left(0\right)+\left(1-\Xi_{0}\left(0\right)\right)e^{-2\int_{0}^{t}A\left(s\right)ds}}}_{\eqqcolon\Xi_{0}\left(t\right)}\right).\label{eq: lift flow relation}
\end{equation}
It is clear that the second coordinate above represents a periodic
function only if $\int_{0}^{T_{\gamma}}A\left(s\right)ds=0$ (i.e.
$L^{E}$ is volume-preserving along $\gamma$ ) or $\Xi_{0}\left(0\right)=\pm1$.
Thus in the non-volume preserving case we must have $\Xi_{0}\left(0\right)=\pm1$,
which gives the periods of the $\hat{Z}$ to be the same as those
of $Z$. On the other hand if $\int_{0}^{T_{\gamma}}A\left(s\right)ds=0$,
\prettyref{eq: lift flow relation} is periodic with its periods at
the two initial extreme conditions $\Xi_{0}\left(0\right)=0,1$ computed
to be $\hat{T}_{\gamma},T_{\gamma}$ respectively. The second equality
\prettyref{eq:special period spectrum lift of charac.} is an immediate
specialization of the first while the last \prettyref{eq: density of normal periods}
is an easy computation from of the normal form\prettyref{eq:normal form-1}.
\end{proof}

\section{Global calculus\label{sec:Global-calculus}}

We now define a global calculus of Hermite operators using the local
calculus of \prettyref{sec:Hermite-Calculus} and the normal form
\prettyref{prop: normal form near pt}. To give a definition independent
of choices one needs an invariance lemma in the upcoming section. 

\subsection{\label{subsec:Invariance}Invariance}

Below $p=\left(0,0,0,0;0,0,0,1\right)\in T^{*}\mathbb{R}^{4}$ is
as before \prettyref{eq:model point} while $\kappa:T^{*}\mathbb{R}^{4}\rightarrow T^{*}\mathbb{R}^{4}$
denotes a local conic symplectomorphism fixing $p$ and $\Sigma_{0}$.
Let $C_{\kappa}\subset\left(T^{*}\mathbb{R}^{4}\right)\times\left(T^{*}\mathbb{R}^{4}\right)^{-}$
be the associated canonical relation. We denote by the same notation
$\kappa$ the induced local diffeomorphisms of $S^{*}\mathbb{R}_{x}^{4}$
as well as the blowup $\left[S^{*}\mathbb{R}_{x}^{4};S^{*}\Sigma_{0}\right]$.
Furthermore $\Delta_{R,r}$ is as in \prettyref{eq:normal form} and
$C\subset C'$ are conic neighborhoods of $\left(\Sigma\setminus0\right)\cap\pi^{-1}\left(x\right)$
satisfying \prettyref{eq: small cone arrangement 0}, \prettyref{eq:small cone arrangement}.
\begin{lem}
\label{lem:Invariance Lemma for Psim1m2} Let $U\in I_{\textrm{cl}}^{0}\left(\mathbb{R}^{4},\mathbb{R}^{4};C_{\kappa}\right)$
be a local Fourier integral and $\rho,\rho'=\kappa^{*}\rho\in C^{\infty}\left(\Sigma_{0}\right)$,
$R,R'\in O_{\Sigma}\left(4\right),\,r,r'\in S_{\textrm{cl}}^{0}$
as in Theorem \prettyref{prop: normal form near pt} satisfying 
\begin{align}
U\Delta_{\rho,R,r}U^{*} & =\Delta_{\rho',R',r'}\label{eq:FIO preserves Delta}\\
UU^{*} & =U^{*}U=1\label{eq:microlocal unitarity}
\end{align}
microlocally on a conic neighborhood $C$ of $p$. 

Then one has 
\begin{align}
U\Omega U^{*} & =\Omega\;\textrm{ microlocally on }C,\textrm{ and }\label{eq: Omega is invariant statement}\\
UAU^{*} & \in\Psi_{\textrm{cl}}^{m_{1},m_{2}}\left(\mathbb{R}^{4};\Sigma_{0}\right),\label{eq: class is invariant}\\
\sigma_{m_{1},m_{2}}^{H}\left(UAU^{*}\right) & =\kappa^{*}\sigma_{m_{1},m_{2}}^{H}\left(A\right),\label{eq:symbol well defined}
\end{align}
$\forall A\in\Psi_{\textrm{cl}}^{m_{1},m_{2}}\left(\mathbb{R}^{4};\Sigma_{0}\right)$
with microsupport in $C\times C$ .
\end{lem}
\begin{proof}
First from Thm. \prettyref{prop: normal form near pt} one has $\Delta_{\rho,R,r}H_{k}=H_{k}\left[\xi_{0}^{2}+\left(2k+1\right)\rho+R+r\right]^{W}$.
For $B\in\Psi^{0}\left(\mathbb{R}^{4}\right)$, $WF\left(B\right)\subset C'$,
$B=1$ on $C$ one computes 
\begin{align*}
 & c_{0}\left|\left(l-k\right)\right|\left(1+O\left(\varepsilon\right)\right)\left\Vert H_{l}BUH_{k}^{*}\right\Vert _{L^{2}\left(\mathbb{R}^{3}\right)\rightarrow L^{2}\left(\mathbb{R}^{3}\right)}\\
= & \left\Vert \left(H_{l}B\left[\left(2l+1\right)\rho+R^{W}\right]UH_{k}^{*}\right)-\left(H_{l}BU\left[\left(2k+1\right)\rho'+R'^{W}\right]H_{k}^{*}\right)\right\Vert _{L^{2}\left(\mathbb{R}^{3}\right)\rightarrow H^{-1}\left(\mathbb{R}^{3}\right)}\\
= & \left\Vert H_{l}B\left(\left[U,\left(\xi_{0}^{2}\right)^{W}\right]+Ur^{W}-r'^{W}U\right)H_{k}^{*}\right\Vert _{L^{2}\left(\mathbb{R}^{3}\right)\rightarrow L^{2}\left(\mathbb{R}^{3}\right)}\\
\leq & O\left(\varepsilon\right)\left\Vert H_{l}BUH_{k}^{*}\right\Vert _{L^{2}\left(\mathbb{R}^{3}\right)\rightarrow H^{-1}\left(\mathbb{R}^{3}\right)}
\end{align*}
using the ellipticity of $\rho,\rho'$ near $p$ and \prettyref{eq: small cone arrangement 0},
\prettyref{eq:small cone arrangement}, \prettyref{eq:FIO preserves Delta}.
This gives 
\begin{equation}
H_{l}BUH_{k}^{*}=0,\quad\forall l\neq k;\label{eq:U is diagonal}
\end{equation}
i.e. $U$ microlocally preserves the Landau levels.

Next, for $A=a^{W}\in\Psi^{m}\left(\mathbb{R}^{4}\right)$ with $WF\left(A\right)\subset C$
the above and $\Omega=\xi_{3}x_{1}^{2}+\xi_{3}^{-1}\xi_{1}^{2}=\sum_{k=0}^{\infty}\left(2k+1\right)H_{k}^{*}H_{k}$
gives

\begin{align}
\left[U,\Omega\right] & =0\quad\textrm{ microlocally on }C\label{eq: Omega is invariant}\\
\left[a^{W},\Omega\right] & =0\implies\left[U^{*}a^{W}U,\Omega\right]=0\label{eq: conjug preserves commutation}
\end{align}
proving \prettyref{eq: Omega is invariant statement}. In other words,
for a symbol $a\in C_{\textrm{inv}}^{\infty}\left(T^{*}\mathbb{R}_{x}^{4}\right)$,
the conjugate $a_{U}^{W}\coloneqq U^{*}a^{W}U$ is again of the same
form $a_{U}\in C_{\textrm{inv}}^{\infty}\left(T^{*}\mathbb{R}_{x}^{4}\right)$.
Furthermore by an Egorov argument as in \cite[Ch. 10]{Grigis-Sjostrand94},
the conjugate has the form $a_{U}\sim\kappa^{*}\left(\underbrace{P_{j}^{U}a}_{\in S^{m-j}}\right)$;
where each $P_{j}$ is a differential operator of homogeneous degree
$-j$ mapping $S^{m}$ to $S^{m-j}$ with $P_{0}=1$. The last implies
that each of $\left\{ \underbrace{x_{1}^{2}+\hat{\xi}_{1}^{2}}_{\eqqcolon\varrho^{2}},\hat{\xi}_{0},\hat{\xi}_{2},\xi_{3};x_{0},x_{2},\hat{x}_{3}\right\} $
maps under $\kappa$ to a function of the same set of variables. Thus
each 
\begin{equation}
P_{j}^{U}=\sum_{\alpha\in\mathbb{N}_{0}^{7}}c_{\alpha,j}\left(\varrho,\hat{\xi}_{0},\hat{\xi}_{2},\xi_{3};x_{0},x_{2},\hat{x}_{3}\right)\left(\varrho\partial_{\varrho}\right)^{\alpha_{1}}\partial_{\xi_{0}}^{\alpha_{0}}\partial_{\xi_{2}}^{\alpha_{2}}\partial_{\xi_{3}}^{\alpha_{3}}\partial_{x_{0}}^{\beta_{2}}\partial_{x_{2}}^{\beta_{2}}\partial_{\hat{x}_{3}}^{\beta_{3}}\label{eq:differential operator Pj}
\end{equation}
 is also a differential operator in the given set of variables. 

Finally for $A=a^{H}\in\Psi_{\textrm{cl}}^{m_{1},m_{2}}\left(\mathbb{R}^{4};\Sigma_{0}\right)$;
with $a\in\left(\beta^{*}d\right)^{-m_{2}}C_{c,\textrm{inv}}^{\infty}\left(\left[K_{1,1};\Sigma_{0}\right]\right)$
supported in the lift of $C$, it now follows using \prettyref{eq:def. Hermite Quantization 2},
\prettyref{eq:U is diagonal} that 
\begin{align}
U^{*}a^{H}U & =\sum_{k=0}^{\infty}U^{*}H_{k}^{*}a_{k}^{W}H_{k}U\nonumber \\
 & =\sum_{k=0}^{\infty}H_{k}^{*}U^{*}a_{k}^{W}UH_{k}\nonumber \\
 & =\sum_{k=0}^{\infty}H_{k}^{*}a_{U,k}^{W}H_{k}\nonumber \\
 & =a_{U}^{H}\label{eq: conjugate Hermite}
\end{align}
Here $a_{U}\in\left(\beta^{*}d\right)^{-m_{2}}C_{c,\textrm{inv}}^{\infty}\left(\left[K_{1,1};\Sigma_{0}\right]\right)$
satisfies
\begin{equation}
a_{U}\sim\kappa^{*}\left(\underbrace{\tilde{P}_{j}^{U}a}_{\in S^{m_{1}-\frac{j}{2},m_{2}}}\right).\label{eq:hermite conjugate symbol}
\end{equation}
where $\tilde{P}_{j}^{U}$ denotes the lift to the blowup $\left[K_{1,1};\Sigma_{0}\right]$
of the differential operator obtained by deleting the terms in \prettyref{eq:differential operator Pj}
involving a $\varrho\partial_{\varrho}$ derivative (with $\alpha_{1}\geq1$)
. The necessary symbolic estimates and and expansion for the conjugate
symbol $a_{U}$ in $S_{\textrm{cl}}^{m_{1},m_{2}}$ now follow from
\prettyref{eq:hermite conjugate symbol} and the corresponding estimates
for $a\in S_{\textrm{cl}}^{m_{1},m_{2}}$. In order to obtain the
symbolic expansion we note $U\rho U^{*}=\underbrace{\rho'}_{=\kappa^{*}\rho}+S^{0}$,
\prettyref{eq:FIO preserves Delta} and \prettyref{eq: Omega is invariant statement}
give $U\xi_{0}^{2}U^{*}=\xi_{0}^{2}+O\left(\xi_{3}^{2}\left(\hat{d}_{\rho'}\right)^{4}\right)$.
Then
\[
Ud_{\rho}U^{*}=d_{\rho'}+O\left(\xi_{3}^{2}\left(\hat{d}_{\rho'}\right)^{4}\right)
\]
\prettyref{eq: conjugate Hermite} and symbolic calculus in the $\rho'$
calculus give the necessary symbolic expansion for $a_{U}$. 
\end{proof}
We note that \prettyref{eq: Omega is invariant statement} establishes
the invariance of $\Omega$, completing the proof of \prettyref{prop: Hamilton v field extension}.

\subsection{Calculus}

Following the invariance \prettyref{lem:Invariance Lemma for Psim1m2},
one may now construct a global calculus of Hermite operators. To this
end, we choose a collection of points $\left\{ p_{j}\in\Sigma\right\} _{j=1}^{M}$
along with diagonalizing Fourier integral operators $\left\{ U_{j}:L^{2}\left(X\right)\rightarrow L^{2}\left(\mathbb{R}^{4}\right)\right\} _{j=1}^{M}$
associated to symplectomorphisms $\kappa_{j}:T^{*}X\rightarrow T^{*}\mathbb{R}^{4}$
which put $\Delta_{g^{E},\mu}$ in normal form \prettyref{eq:normal form}
in conic neighborhoods $\left\{ p_{j}\in C_{j}\right\} _{j=1}^{M}$
covering $\Sigma$. 
\begin{defn}
\label{def: Psi m123 on X}An operator $T:C^{\infty}\left(X\right)\rightarrow C^{-\infty}\left(X\right)$
is said to lie in the class $\Psi_{\textrm{cl}}^{m_{1},m_{2}}\left(X,\Sigma\right)$
iff it can be written $T=T_{0}+\sum_{j=1}^{M}T_{j}$ where
\end{defn}
\begin{enumerate}
\item $WF\left(T_{0}\right)\subset\left(T^{*}X\times T^{*}X\right)\setminus\left(\Sigma\times\Sigma\right)$
with $T_{0}\in\Psi_{\textrm{cl}}^{m_{1}}\left(X\right)$ 
\item $WF\left(T_{j}\right)\subset C_{j}\times C_{j}$ with $U_{j}T_{j}U_{j}^{*}\in\Psi_{\textrm{cl}}^{m_{1},m_{2}}\left(\mathbb{R}^{4};\Sigma_{0}\right)$,
$j=1,\ldots,M.$
\end{enumerate}
It is an easy exercise from \prettyref{lem:Invariance Lemma for Psim1m2}
that the definition above is independent of the choice of diagonalizing
Fourier integral operators $\left\{ U_{j}:L^{2}\left(X\right)\rightarrow L^{2}\left(\mathbb{R}^{4}\right)\right\} _{j=1}^{M}$
. 

The symbol of $T\in\Psi_{\textrm{cl}}^{m_{1},m_{2}}\left(X,\Sigma\right)$,
$m_{2}\leq0$, is then defined via 
\begin{equation}
\sigma_{m_{1},m_{2}}^{H}\left(T\right)\coloneqq\sigma\left(T_{0}\right)+\kappa_{j}^{*}\sigma_{m_{1},m_{2}}^{H}\left(U_{j}T_{j}U_{j}^{*}\right)\in C^{\infty}\left(\left[T^{*}X;\Sigma\right]\right)\label{eq:symbol definition}
\end{equation}
and is again invariantly defined by virtue of \prettyref{eq:symbol well defined}.
Much like \prettyref{eq:symbol def.}, the symbol above has an invariance
property. First note that by \prettyref{eq: Omega is invariant},
the pseudo-differential operator $\Omega\in\Psi_{\textrm{cl}}^{1}\left(X\right)$
and its homogeneous symbol are microlocally and invariantly defined
on a conic neighborhood $C_{\Omega}\subset\Sigma$ of the the characteristic
variety. We also denote by $\Omega$ its pullback to the blowup defined
on the neighborhood $\beta^{-1}\left(C_{\Omega}\right)$ of the boundary.
Furthermore its Hamilton vector field $H_{\Omega}$ has a lift to
the blowup, that is tangent to the boundary and homogeneous of degree
zero, which we denote by the same notation $H_{\Omega}\in C^{\infty}\left(T\left[T^{*}X;\Sigma\right]\right)$.
Its restriction to the boundary is the rotational vector field 
\begin{equation}
\left.H_{\Omega}\right|_{SN\Sigma}=R_{0}\label{eq: rotational vector field rest.}
\end{equation}
 is the rotational vector field following the identification \prettyref{eq:identifications}. 

We then define the space of invariant symbols
\begin{align}
C_{\textrm{inv}}^{\infty}\left(\left[\left(S^{*}X\right);S^{*}\Sigma\right]\right)\coloneqq & \left\{ f\in C^{\infty}\left(\left[\left(S^{*}X\right);S^{*}\Sigma\right]\right)|f=f_{0}+f_{1},\,f_{0}\in C_{c}^{\infty}\left(\left[\left(S^{*}X\right);S^{*}\Sigma\right]^{o}\right),\right.\nonumber \\
 & \,\qquad\quad\quad\left.f_{1}\in C_{c}\left(\beta^{-1}\left(C_{\Omega}\right)\right),\,H_{\Omega}f_{1}=0\right\} .\label{eq:symbol space}
\end{align}
The above may also be considered as homogeneous functions of degree
zero on $\left[T^{*}X;\Sigma\right]$. We may then similarly define
$C_{\textrm{inv},m}^{\infty}$, $m\in\mathbb{Z}$, by requiring homogeneity
of degree $m$; this space is however non-canonically identified with
\prettyref{eq:symbol space} on choosing positive function in \prettyref{eq:symbol space}.
It follows from definition that the sR Laplacian $\Delta_{g^{E},\mu}\in\Psi^{2,-2}\left(X,\Sigma\right)$.
Further, it easy to see from the normal form \prettyref{eq:normal form}
that with 
\begin{equation}
d\coloneqq\left.\left[\sigma_{2,-2}^{H}\left(\Delta_{g^{E},\mu}\right)\right]^{1/2}\right|_{S^{*}X}\label{eq:symb. square root}
\end{equation}
being homogeneous of degree one, $\left(\beta^{*}d\right)$ defines
an element of the symbol space \prettyref{eq:symbol space}. The symbol
of a general $T\in\Psi_{\textrm{cl}}^{m_{1},m_{2}}\left(X,\Sigma\right)$
is defined by the same formula \ref{eq:symbol definition} and is
now an element of 
\begin{equation}
\sigma_{m_{1},m_{2}}^{H}\left(T\right)\in\left(\beta^{*}d\right)^{-m_{2}}C_{\textrm{inv}}^{\infty}\left(\left[\left(S^{*}X\right);S^{*}\Sigma\right]\right).\label{eq: symbol invariant def.}
\end{equation}
 We shall say that an element $T\in\Psi_{\textrm{cl}}^{m_{1},m_{2}}\left(X,\Sigma\right)$
is elliptic in the exotic calculus if and only if
\begin{equation}
c\left(\beta^{*}d\right)^{-m_{2}}\leq\sigma_{m_{1},m_{2}}^{H}\left(T\right)\leq C\left(\beta^{*}d\right)^{-m_{2}}\label{eq:ellipticity}
\end{equation}
for some constants $c,C>0$. Similar to \prettyref{eq:inclusion of psedos},
\prettyref{eq:inclusion of inv. classical symb.} one then has the
inclusions
\begin{align}
\Psi_{\textrm{cl}}^{m_{1},m_{2}}\left(X;\Sigma\right) & \subset\Psi_{\textrm{cl}}^{m_{1}+\frac{1}{2},m_{2}-1}\left(X;\Sigma\right)\nonumber \\
\Psi_{\textrm{inv},\textrm{cl}}^{m}\left(X\right) & \subset\Psi_{\textrm{cl}}^{m,0}\left(X;\Sigma\right)\label{eq:inclusion of pseudos}
\end{align}
where 
\[
\Psi_{\textrm{inv},\textrm{cl}}^{m}\left(X\right)\coloneqq\left\{ A=A_{0}+A_{1}\in\Psi_{\textrm{cl}}^{m}\left(X\right)|WF\left(A_{0}\right)\subset C_{\Omega},\,WF\left(\left[A_{0},\Omega\right]\right)\subset T^{*}X\setminus\Sigma,\,WF\left(A_{1}\right)\subset T^{*}X\setminus\Sigma\right\} .
\]

One similarly defines the generalized Sobolev spaces $H^{s_{1},s_{2}}\left(X,\Sigma\right)$
via $u\in H^{s_{1},s_{2}}\left(X,\Sigma\right)$ if and only if $u=u_{0}+\sum_{j=1}^{M}u_{j}$
where 1.$WF\left(u_{0}\right)\subset T^{*}X\setminus\Sigma$ with
$u_{0}\in H^{s_{1}}$ and 2. $WF\left(u_{j}\right)\subset C_{j}$
with $u_{j}\in H^{s_{1},s_{2}}\left(\mathbb{R}_{x}^{4};\Sigma_{0}\right)$.
A pseudo-differential characterization of $H^{s_{1},s_{2}}\left(X,\Sigma\right)$
is given using \prettyref{eq:pseudodiff charac of Hs1,s2} by 
\begin{equation}
u\in H^{s_{1},s_{2}}\left(X,\Sigma\right)\iff Au\in L^{2},\forall A\in\Psi^{s_{1},s_{2}}\left(X,\Sigma\right).\label{eq: pseudo charac Sobolev}
\end{equation}
 Following \prettyref{eq:inclusion of pseudos} this now gives
\begin{align}
H^{s,0}\left(X,\Sigma\right) & =H^{s}\left(X\right)\nonumber \\
H^{s_{1}+\frac{1}{2},s_{2}-1}\left(X,\Sigma\right) & \subset H^{s_{1},s_{2}}\left(X,\Sigma\right).\label{eq:Sobolev inclusions-1}
\end{align}

The characteristic wavefront set $WF_{\Sigma}\left(T\right)\subset\partial\left[\left(S^{*}X\right);S^{*}\Sigma\right]$
of an operator $T\in\Psi_{\textrm{cl}}^{m_{1},m_{2}}\left(X,\Sigma\right)$
is defined via $\left(x,\xi\right)\in WF_{\Sigma}\left(T\right)\iff\kappa\left(x,\xi\right)\in WF\left(UTU^{*}\right)$.
Here $U:L^{2}\left(X\right)\rightarrow L^{2}\left(\mathbb{R}^{4}\right)$
is a diagonalizing FIO near $\beta\left(x,\xi\right)$ associated
to a homogeneous symplectomorphisms $\kappa:T^{*}X\rightarrow T^{*}\mathbb{R}^{4}$
mapping $\Sigma$ to $\Sigma_{0}$ and with lift $\kappa:\left[\left(S^{*}X\right);S^{*}\Sigma\right]\rightarrow\left[\left(S^{*}\mathbb{R}^{4}\right);S^{*}\Sigma_{0}\right]$
being denoted by the same notation. The characteristic wavefront set
$WF_{\Sigma}\left(u\right)\subset\partial\left[\left(S^{*}X\right);S^{*}\Sigma\right]$
of any distribution $u\in C^{-\infty}\left(X\right)$ is then defined
via 
\begin{equation}
\left(x,\xi\right)\notin WF_{\Sigma}\left(u\right)\iff\exists A\in\Psi_{\textrm{cl}}^{0,0}\left(X,\Sigma\right),\textrm{s.t.}\,\left(x,\xi\right)\in WF_{\Sigma}\left(A\right),\,Au\in C^{\infty}.\label{eq:pseudodiff charac WF}
\end{equation}
or equivalently 
\begin{equation}
\left(x,\xi\right)\notin WF_{\Sigma_{0}}\left(u\right)\iff\exists A\in\Psi_{\textrm{cl}}^{0,0}\left(\mathbb{R}^{4};\Sigma_{0}\right),\textrm{s.t.}\,\sigma_{0,0}^{H}\left(A\right)\left(x,\xi\right)\neq0,\,Au\in C^{\infty}.\label{eq:equiv. pseudodiff charc of WF}
\end{equation}
The wavefront projects to restriction of the wavefront $\beta\left(WF_{\Sigma}\left(u\right)\right)=WF\left(u\right)\cap\Sigma$
under the blowdown map \prettyref{eq: ch WF projects to WF}. 

Following their pseudo-differential characterizations \prettyref{eq: pseudo charac Sobolev},
\prettyref{eq:pseudodiff charac WF} it is clear that $H^{s_{1},s_{2}}\left(X,\Sigma\right)$
and $WF_{\Sigma}\left(u\right)$ are also defined independently of
the choice of diagonalizing Fourier integral operators. The properties
of the Hermite calculus from \prettyref{sec:Hermite-Calculus} then
easily carry over globally. We state them below.
\begin{enumerate}
\item (Adjoint \& Composition) The class \prettyref{def: Psi m123 on X}
is closed under composition and adjoint 
\begin{align}
A\in\Psi_{\textrm{cl}}^{m_{1},m_{2}}\left(X,\Sigma\right),\,B\in\Psi_{\textrm{cl}}^{m_{1}',m_{2}'}\left(X,\Sigma\right) & \implies AB\in\Psi_{\textrm{cl}}^{m_{1}+m_{1}',m_{2}+m_{2}'}\left(X,\Sigma\right)\nonumber \\
A\in\Psi_{\textrm{cl}}^{m_{1},m_{2}}\left(X,\Sigma\right) & \implies A^{*}\in\Psi_{\textrm{cl}}^{m_{1},m_{2}}\left(X,\Sigma\right).\label{eq: comp. and adjoint closure}
\end{align}
\item (Characterization of residual terms) One has the inclusions and characterization
of residual terms and in particular the characterization of residual
terms
\begin{align}
\Psi_{\textrm{cl}}^{m_{1},m_{2}}\left(X,\Sigma\right) & \subset\Psi_{\textrm{cl}}^{m_{1}+\frac{1}{2},m_{2}-1}\left(X,\Sigma\right)\label{eq:}\\
\Psi^{-\infty,m_{2}}\left(X,\Sigma\right) & =\Psi_{\textrm{inv},\textrm{cl}}^{-\infty,m_{2}}\left(X,\Sigma\right)\subset\Psi^{-\infty}\left(X\right).\label{eq:residual term charac.}
\end{align}
\item (Principal symbol) There exists a multiplicative principal symbol
map 
\[
\sigma_{m_{1},m_{2}}^{H}:\Psi_{\textrm{cl}}^{m_{1},m_{2}}\left(X,\Sigma\right)\rightarrow\left(\beta^{*}d\right)^{-m_{2}}C_{\textrm{inv}}^{\infty}\left(\left[\left(S^{*}X\right);S^{*}\Sigma\right]\right)
\]
satisfying
\begin{align}
\sigma_{m_{1}+m_{1}',m_{2}+m_{2}'}^{H}\left(AB\right) & =\sigma_{m_{1},m_{2}}^{H}\left(A\right)\sigma_{m_{1}',m_{2}'}^{H}\left(B\right)\nonumber \\
\sigma_{m_{1},m_{2}}^{H}\left(A^{*}\right) & =\overline{\sigma_{m_{1},m_{2}}^{H}\left(A\right)}\label{eq:symb is multiplicative}
\end{align}
for every $A\in\Psi_{\textrm{cl}}^{m_{1},m_{2}}\left(X,\Sigma\right),\,B\in\Psi_{\textrm{cl}}^{m_{1}',m_{2}'}\left(X,\Sigma\right)$.
\item (Symbol exact sequence) The principal symbol fits into the exact sequence
below
\begin{equation}
0\rightarrow\Psi_{\textrm{cl}}^{m_{1}-1,m_{2}+1}\left(X,\Sigma\right)\hookrightarrow\Psi_{\textrm{cl}}^{m_{1},m_{2}}\left(X,\Sigma\right)\xrightarrow{\sigma_{m_{1},m_{2}}^{H}}\left(\beta^{*}d\right)^{-m_{2}}C_{\textrm{inv}}^{\infty}\left(\left[\left(S^{*}X\right);S^{*}\Sigma\right]\right)\rightarrow0.\label{eq:symbol exact seq.-1}
\end{equation}
\item (Quantization) There exists a surjective quantization map 
\[
\textrm{Op}^{H}:\left(\beta^{*}d\right)^{-m_{2}}C_{\textrm{inv}}^{\infty}\left[\left(S^{*}X\right);S^{*}\Sigma\right]\rightarrow\Psi_{\textrm{cl}}^{m_{1},m_{2}}\left(X,\Sigma\right)
\]
 which is a left-inverse to the principal symbol
\begin{align}
\sigma_{m_{1},m_{2}}^{H}\left(\textrm{Op}^{H}a\right) & =a\nonumber \\
\textrm{Op}^{H}\left[\sigma_{m_{1},m_{2}}^{H}\left(A\right)\right] & =A\quad\left(\textrm{mod }\Psi_{\textrm{cl}}^{m_{1}-1,m_{2}+1}\left(X,\Sigma\right)\right).\label{eq: sym =000026 Q are inverse}
\end{align}
\item (Symbol of commutator) For $A\in\Psi_{\textrm{cl}}^{m_{1},m_{2}}\left(X,\Sigma\right)$,
$B\in\Psi_{\textrm{cl}}^{m_{1}',m_{2}'}\left(X,\Sigma\right)$ the
commutator $\left[A,B\right]\in\Psi_{\textrm{cl}}^{m_{1}+m_{1}'-1,m_{2}+m_{2}'+1}\left(X,\Sigma\right)$
with symbol 
\begin{equation}
\sigma_{m_{1}+m_{1}'-1,m_{2}+m_{2}'+1}^{H}\left(\left[A,B\right]\right)=i\left\{ \sigma_{m_{1},m_{2}}^{H}\left(A\right),\sigma_{m_{1}',m_{2}'}^{H}\left(B\right)\right\} .\label{eq:symbol of commutator}
\end{equation}
\item (Asymptotic summation) For any set of operators $A_{j}\in\Psi^{m_{1}-j,m_{2}+j}\left(X,\Sigma\right)$,
$B_{j}\in\Psi^{m_{1},m_{2}-j}\left(X,\Sigma\right)$, (resp. $\Psi_{\textrm{cl}}\left(X,\Sigma\right)$),
$j\in\mathbb{N}_{0}$, there exists $A,B\in\Psi^{m_{1},m_{2}}$ such
that
\begin{align}
A-\sum_{j=0}^{N}A_{j}\in\Psi^{m_{1}-\frac{N}{2},m_{2}}\left(X,\Sigma\right),\nonumber \\
B-\sum_{j=0}^{N}B_{j}\in\Psi^{m_{1},m_{2}-N}\left(X,\Sigma\right), & \forall N\in\mathbb{N}_{0}\label{eq:asymptotic summation-1}
\end{align}
\item (Sobolev boundedness) For any $A\in\Psi_{\textrm{cl}}^{m_{1},m_{2}}\left(X,\Sigma\right)$
and $u\in H^{s_{1},s_{2}}\left(X,\Sigma\right)$ one has $Au\in H^{s_{1}-m_{1},s_{2}-m_{2}}\left(X,\Sigma\right)$. 
\item (Microlocality) For any $A\in\Psi_{\textrm{cl}}^{m_{1},m_{2}}\left(X,\Sigma\right)$,
$B\in\Psi_{\textrm{cl}}^{m_{1}',m_{2}'}\left(X,\Sigma\right)$ and
$u\in C^{-\infty}\left(X\right)$ one has
\begin{align}
WF_{\Sigma}\left(A+B\right) & \subset WF_{\Sigma}\left(A\right)\cup WF_{\Sigma}\left(B\right)\nonumber \\
WF_{\Sigma}\left(AB\right) & \subset WF_{\Sigma}\left(A\right)\cap WF_{\Sigma}\left(B\right)\nonumber \\
WF_{\Sigma}\left(Au\right) & \subset WF_{\Sigma}\left(A\right)\cap WF_{\Sigma}\left(u\right).\label{eq: properties of WF-1}
\end{align}
\end{enumerate}
As a first application of the calculus we construct parametrices for
elliptic elements of $\Psi_{\textrm{cl}}^{m_{1},m_{2}}\left(X,\Sigma\right)$. 
\begin{prop}
\label{prop:parametrix}Let $P\in\Psi_{\textrm{cl}}^{m_{1},m_{2}}\left(X,\Sigma\right)$
be elliptic. Then there exists $Q\in\Psi_{\textrm{cl}}^{-m_{1},-m_{2}}\left(X,\Sigma\right)$
satisfying $PQ-I\in\Psi_{\textrm{cl}}^{-\infty}\left(X\right)$, $QP-I\in\Psi_{\textrm{cl}}^{-\infty}\left(X\right)$.
\end{prop}
\begin{proof}
This is a usual application of the pseudo-differential calculus albeit
in the exotic class \prettyref{def: Psi m123 on X}. Since $\sigma_{m_{1},m_{2}}^{H}\left(P\right)\in\left(\beta^{*}d\right)^{-m_{2}}C_{\textrm{inv}}^{\infty}\left(\left[\left(S^{*}X\right);S^{*}\Sigma\right]\right)$
satisfies \prettyref{eq:ellipticity}, its inverse $\left[\sigma_{m_{1},m_{2}}^{H}\left(P\right)\right]^{-1}\in\left(\beta^{*}d\right)^{m_{2}}C_{\textrm{inv}}^{\infty}\left(\left[\left(S^{*}X\right);S^{*}\Sigma\right]\right)$
can be seen to lie in the given space and maybe quantized $Q_{0}\coloneqq\textrm{Op}^{H}\left(\left[\sigma_{m_{1},m_{2}}^{H}\left(P\right)\right]^{-1}\right)\in\Psi_{\textrm{cl}}^{-m_{1},-m_{2}}\left(X,\Sigma\right)$.
We now compute $\sigma_{0,0}\left(PQ_{0}-I\right)=0$ using \prettyref{eq:symb is multiplicative},
\prettyref{eq: sym =000026 Q are inverse} and thus $PQ_{0}-I\in\Psi_{\textrm{cl}}^{-1,1}\left(X,\Sigma\right)$.
We then set 
\[
Q_{1}\coloneqq-\textrm{Op}^{H}\left(\left[\sigma_{m_{1},m_{2}}^{H}\left(P\right)\right]^{-1}\sigma_{-1,1}\left(PQ_{0}-I\right)\right)\in\Psi_{\textrm{cl}}^{-m_{1}-1,-m_{2}+1}\left(X,\Sigma\right)
\]
 and again compute $P\left(Q_{0}+Q_{1}\right)-I\in\in\Psi_{\textrm{cl}}^{-2,2}\left(X,\Sigma\right)$.
Continuing in this fashion gives a sequence $Q_{j}\in\Psi_{\textrm{cl}}^{-m_{1}-j,-m_{2}+j}\left(X,\Sigma\right)$,
$j=0,1,\ldots$ such that 
\[
P\left(\sum_{j=0}^{N}Q_{j}\right)-I\in\Psi_{\textrm{cl}}^{-N-1,N+1}\left(X,\Sigma\right)\subset\Psi_{\textrm{cl}}^{-\frac{1}{2}\left(N+1\right),0}\left(X,\Sigma\right).
\]
The asymptotic summation $A\sim\sum_{j=0}^{\infty}Q_{j}$ \prettyref{eq:asymptotic summation-1}
then satisfies $PQ-I\in\Psi_{\textrm{cl}}^{-\infty,0}\left(X,\Sigma\right)\subset\Psi_{\textrm{cl}}^{-\infty}\left(X\right)$
as required. The construction of the left parametrix $Q'$ satisfying
$Q'P-I\in\Psi_{\textrm{cl}}^{-\infty}\left(X\right)$ is similar.
Seeing these to agree $Q-Q'=Q'\left(I-PQ\right)+\left(Q'P-I\right)Q\in\Psi_{\textrm{cl}}^{-\infty}\left(X\right)$
modulo residual terms gives the result.
\end{proof}
As an application we improve a subelliptic estimate.
\begin{prop}
Let $P\in\Psi_{\textrm{cl}}^{m_{1},m_{2}}\left(X,\Sigma\right)$ be
elliptic. Then there exists $C>0$ such that 
\begin{equation}
\left\Vert f\right\Vert _{H^{s_{1}+m_{1},s_{2}+m_{2}}}\leq C\left[\left\Vert Pf\right\Vert _{H^{s_{1},s_{2}}}+\left\Vert f\right\Vert _{H^{s_{1},s_{2}}}\right]\label{eq:exotic subelliptic estimate}
\end{equation}
$\forall f\in C^{\infty}\left(X\right),\,s_{1},s_{2}\in\mathbb{R}.$
\end{prop}
\begin{proof}
With $Q$ being the parametrix \prettyref{prop:parametrix} for $P$,
write $f=QPf+\left(I-QP\right)f$ and use the Sobelev boundedness
$\left\Vert Q\right\Vert _{H^{s_{1},s_{2}}\left(X,\Sigma\right)\rightarrow H^{s_{1}+m_{1},s_{2}+m_{2}}\left(X,\Sigma\right)}<\infty$,
$\left\Vert I-QP\right\Vert _{H^{s_{1},s_{2}}\left(X,\Sigma\right)\rightarrow H^{s_{1}+m_{1},s_{2}+m_{2}}\left(X,\Sigma\right)}<\infty$.
\end{proof}
Since $\Delta_{g^{E},\mu}$ is clearly elliptic in $\Psi_{\textrm{cl}}^{2,-2}\left(X,\Sigma\right)$
by definition, the above proposition gives 
\begin{equation}
\left\Vert f\right\Vert _{H^{s_{1}+2,s_{2}-2}}\leq C\left[\left\Vert \Delta_{g^{E},\mu}f\right\Vert _{H^{s_{1},s_{2}}}+\left\Vert f\right\Vert _{H^{s_{1},s_{2}}}\right]\label{eq:refined subelliptic estimate}
\end{equation}
$\forall f\in C^{\infty}\left(X\right),\,s_{1},s_{2}\in\mathbb{R}.$
In light of the inclusions \prettyref{eq:Sobolev inclusions-1} the
above refines the subelliptic estimate for the sR Laplacian \prettyref{eq:hypoelliptic with loss of one derivative}
in our particular 4D quasi-contact case.
\begin{rem}
Although the notation suppresses it, the class of pseudo-differential
operators $\Psi_{\textrm{cl}}^{m_{1},m_{2}}\left(X,\Sigma\right)$
is depends on the Laplacian $\Delta_{g^{E},\mu}$ and not just the
characteristic variety. This class differs from the more well-known
class of operators defined in \cite{Boutet-de-Monvel-hypoelliptic-paramatrices,Boutet-Treves-74}
wherein the corresponding classes depend only on the characteristic
variety and their symbols do not necessarily satisfy any invariance
condition.
\end{rem}

\subsection{\label{subsec:Egorov-and-propagation}Egorov and propagation}

In this section we explore some immediate consequences of the global
calculus of the previous subsection. We first begin by showing that
the square root of the Laplacian lies in the given class.
\begin{prop}
\label{prop: square root in class}The square root lies in the given
class $\sqrt{\Delta_{g^{E},\mu}}\in\Psi_{\textrm{cl}}^{1,-1}\left(X,\Sigma\right)$
with symbol $\sigma_{1,-1}^{H}\left(\sqrt{\Delta_{g^{E},\mu}}\right)=d$
\prettyref{eq:symb. square root}. 
\end{prop}
\begin{proof}
This is another application of the pseudo-differential calculus \prettyref{def: Psi m123 on X}.
As noted before $\beta^{*}d$ lies in the symbol space and can be
quantized $A_{0}\coloneqq\textrm{Op}^{H}\left(\beta^{*}d\right)\in\Psi_{\textrm{cl}}^{1,-1}\left(X,\Sigma\right)$.
It squares principally $\sigma_{2,-2}^{H}\left(\Delta_{g^{E},\mu}-\textrm{Op}^{H}\left(\beta^{*}d\right)^{2}\right)=0$
by \prettyref{eq:symb. square root},\prettyref{eq:symb is multiplicative},
\prettyref{eq: sym =000026 Q are inverse} and thus $\Delta_{g^{E},\mu}-A_{0}^{2}\in\Psi_{\textrm{cl}}^{1,-1}\left(X,\Sigma\right)$
by \prettyref{eq:symbol exact seq.-1}. Now define 
\[
A_{1}\coloneqq\frac{1}{2}\textrm{Op}^{H}\left[\left(\beta^{*}d\right)^{-1}\sigma_{1,-1}\left(\Delta_{g^{E},\mu}-\textrm{Op}^{H}\left(\beta^{*}d\right)^{2}\right)\right]\in\Psi_{\textrm{cl}}^{0,0}\left(X,\Sigma\right)
\]
 and again calculate $\Delta_{g^{E},\mu}-\left(A_{0}+A_{1}\right)^{2}\in\Psi_{\textrm{cl}}^{0,0}\left(X,\Sigma\right)$.
Continuing in this fashion we inductively construct a sequence $A_{j}\in\Psi_{\textrm{cl}}^{1-j,-1+j}\left(X,\Sigma\right)$,
$j=0,1,2,\ldots$ such that
\[
\Delta_{g^{E},\mu}-\left(\sum_{j=0}^{N}A_{j}\right)^{2}\in\Psi_{\textrm{cl}}^{1-N,-1+N}\left(X,\Sigma\right).
\]
The asymptotic summation $A\sim\sum_{j=0}^{\infty}A_{j}\in\Psi_{\textrm{cl}}^{1,-1}\left(X,\Sigma\right)$
\prettyref{eq:asymptotic summation-1} then satisfies $\Delta_{g^{E},\mu}-A^{2}\in\Psi_{\textrm{cl}}^{-\infty,0}\left(X,\Sigma\right)\subset\Psi_{\textrm{cl}}^{-\infty}\left(X\right)$.
The symbol $\sigma_{1,-1}^{H}\left(A\right)=\beta^{*}d$ shows that
$A$ is elliptic, satisfying the subelliptic estimate\prettyref{eq:exotic subelliptic estimate},
and hence has a compact resolvent by \prettyref{eq:Sobolev inclusions-1}.
It thus has only finitely many non-positive eigenvalues and can be
altered, by projecting off the negative eigenvalues, to a positive
operator. We now write the difference 
\[
\sqrt{\Delta_{g^{E},\mu}}-A=\frac{1}{2\pi i}\int_{\Gamma}dz\,z^{-1/2}\left(\Delta_{g^{E},\mu}-z\right)^{-1}\underbrace{\left(\Delta_{g^{E},\mu}-A^{2}\right)}_{\in\Psi_{\textrm{cl}}^{-\infty,0}\left(X\right)}\left(A^{2}-z\right)^{-1}
\]
 with $\Gamma$ representing a contour around the positive real axis,
to see that the difference above is also in $\Psi_{\textrm{cl}}^{-\infty}\left(X\right)$
and complete the proof.
\end{proof}
We next prove an Egorov theorem for conjugation by the half wave operator
$e^{it\sqrt{\Delta_{g^{E},\mu}}}$. Below and it what follows we note
that the evolution $\left(e^{tH_{d}}\right)^{*}\sigma_{m_{1},m_{2}}^{H}\left(P\right)\in\in\left(\beta^{*}d\right)^{-m_{2}}C_{\textrm{inv}}^{\infty}\left(\left[\left(S^{*}X\right);S^{*}\Sigma\right]\right)$
\prettyref{eq: symbol invariant def.} lies in the same class on account
of \prettyref{eq:singular expansion Hamiltonian} and circular invariance
of symbol \prettyref{eq:symbol space}.
\begin{thm}
\label{thm:Egororovb thm}For any $P\in\Psi_{\textrm{cl}}^{m_{1},m_{2}}\left(X,\Sigma\right)$
the conjugate $P\left(t\right)\coloneqq e^{-it\sqrt{\Delta_{g^{E},\mu}}}Pe^{it\sqrt{\Delta_{g^{E},\mu}}}\in\Psi_{\textrm{cl}}^{m_{1},m_{2}}\left(X,\Sigma\right)$
lies in the same pseudo-differential class with $\sigma_{m_{1},m_{2}}^{H}\left(P\left(t\right)\right)=\left(e^{-tH_{d}}\right)^{*}\sigma_{m_{1},m_{2}}^{H}\left(P\right)$.
\end{thm}
\begin{proof}
We again use symbolic calculus in the class \prettyref{def: Psi m123 on X}.
Since the conjugate satisfies the differential equation $\partial_{t}P+\left[\sqrt{\Delta_{g^{E},\mu}},P\right]=0$,
we first solve this equation symbolically modulo residual terms. First
define $A_{0}\left(t\right)=\textrm{Op}^{H}\left[\left(e^{-tH_{d}}\right)^{*}\sigma_{m_{1},m_{2}}^{H}\left(P\right)\right]\in\Psi_{\textrm{cl}}^{m_{1},m_{2}}\left(X,\Sigma\right)$;
it is easy to check that $\left(e^{-tH_{d}}\right)^{*}\sigma_{m_{1},m_{2}}^{H}\left(P\right)\in\left(\beta^{*}d\right)^{-m_{2}}C_{\textrm{inv}}^{\infty}\left[\left(S^{*}X\right);S^{*}\Sigma\right]$.
We then compute $\sigma_{m_{1},m_{2}}^{H}\left(\partial_{t}A_{0}+\left[\sqrt{\Delta_{g^{E},\mu}},A_{0}\right]\right)=-H_{d}P+H_{d}P=0$
using \prettyref{eq:symb. square root}, \prettyref{eq:symbol of commutator}
and thus $\partial_{t}A_{0}+\left[\sqrt{\Delta_{g^{E},\mu}},A_{0}\right]\in\Psi_{\textrm{cl}}^{m_{1}-1,m_{2}+1}\left(X,\Sigma\right)$
by \prettyref{eq:symbol exact seq.-1}. Now define 
\begin{align*}
A_{1}\coloneqq & \textrm{Op}^{H}\left[\sigma_{m_{1}-1,m_{2}+1}^{H}\left(P-\sigma_{m_{1},m_{2}}^{H}\left(P\right)\right)\right.\\
 & \left.\quad+\int_{0}^{t}ds\,\left(e^{-\left(t-s\right)H_{d}}\right)^{*}\sigma_{m_{1}-1,m_{2}+1}^{H}\left(\partial_{t}A_{0}+\left[\sqrt{\Delta_{g^{E},\mu}},A_{0}\right]\right)\left(s\right)\right]\in\Psi_{\textrm{cl}}^{m_{1}-1,m_{2}+1}\left(X,\Sigma\right)
\end{align*}
and again compute $\sigma_{m_{1}-1,m_{2}+1}^{H}\left(\partial_{t}\left(A_{0}+A_{1}\right)+\left[\sqrt{\Delta_{g^{E},\mu}},\left(A_{0}+A_{1}\right)\right]\right)=0$
using the Duhamel's principle \prettyref{eq:symb. square root}, \prettyref{eq:symbol of commutator}
and thus $\partial_{t}\left(A_{0}+A_{1}\right)+\left[\sqrt{\Delta_{g^{E},\mu}},\left(A_{0}+A_{1}\right)\right]\in\Psi_{\textrm{cl}}^{m_{1}-2,m_{2}+2}\left(X,\Sigma\right)$.
Continuing in this fashion we inductively construct a sequence $A_{j}\in\Psi_{\textrm{cl}}^{m_{1}-j,m_{2}+j}\left(X,\Sigma\right)$,
$j=0,1,2,\ldots$ such that
\begin{align*}
\partial_{t}\left(\sum_{j=0}^{N}A_{j}\right)+\left[\sqrt{\Delta_{g^{E},\mu}},\left(\sum_{j=0}^{N}A_{j}\right)\right] & \in\Psi_{\textrm{cl}}^{m_{1}-N-1,m_{2}+N+1}\left(X,\Sigma\right)\\
P-\left.\left(\sum_{j=0}^{N}A_{j}\right)\right|_{t=0} & \in\Psi_{\textrm{cl}}^{m_{1}-N-1,m_{2}+N+1}\left(X,\Sigma\right).
\end{align*}
Thus again the asymptotic summation $A\sim\sum_{j=0}^{\infty}A_{j}\in\Psi_{\textrm{cl}}^{m_{1},m_{2}}\left(X,\Sigma\right)$
\prettyref{eq:asymptotic summation-1} then satisfies $R\left(t\right)\coloneqq\partial_{t}A+\left[\sqrt{\Delta_{g^{E},\mu}},A\right]\in\Psi_{\textrm{cl}}^{-\infty,m_{2}}\left(X\right)$,
$R_{0}\coloneqq P-\left.A\right|_{t=0}\in\Psi_{\textrm{cl}}^{-\infty,m_{2}}\left(X\right)$.
Finally Duhamel's principle gives
\[
P\left(t\right)-A=\underbrace{R_{0}}_{\Psi_{\textrm{cl}}^{-\infty,m_{2}}\left(X\right)}+\int_{0}^{t}ds\,e^{-is\sqrt{\Delta_{g^{E},\mu}}}\underbrace{R\left(t-s\right)}_{\in\Psi_{\textrm{cl}}^{-\infty,m_{2}}\left(X\right)}e^{is\sqrt{\Delta_{g^{E},\mu}}},
\]
showing that the difference is in $\Psi_{\textrm{cl}}^{-\infty}\left(X\right)$
and completing the proof.
\end{proof}
As an immediate application of Egorov theorem we have propagation
of singularities \prettyref{eq:pseudodiff charac WF}.
\begin{proof}[Proof of \prettyref{thm:propagation on sing.}]
 By \prettyref{eq:equiv. pseudodiff charc of WF}, for $\left(x,\xi\right)\notin WF_{\Sigma}\left(u\right)$
there exists $A\in\Psi_{\textrm{cl}}^{0,0}\left(X,\Sigma\right)$
with $\sigma_{0,0}^{H}\left(A\right)\left(x,\xi\right)\neq0$ and
$Au\in C^{\infty}$. Thus with $A\left(t\right)\coloneqq e^{it\sqrt{\Delta_{g^{E},\mu}}}Ae^{-it\sqrt{\Delta_{g^{E},\mu}}}\in\Psi_{\textrm{cl}}^{0,0}\left(X,\Sigma\right)$
we have $\sigma_{0,0}^{H}\left(A\left(t\right)\right)\left(e^{t\hat{Z}}\left(x,\xi\right)\right)=\sigma_{0,0}^{H}\left(A\right)\left(x,\xi\right)\neq0$
by \prettyref{eq:singular part Ham. vfield}, \prettyref{eq:singular expansion Hamiltonian},
\prettyref{thm:Egororovb thm} and rotational invariance of the symbol.
Then $A\left(t\right)e^{it\sqrt{\Delta_{g^{E},\mu}}}u=e^{it\sqrt{\Delta_{g^{E},\mu}}}Au\in C^{\infty}$
shows $e^{t\hat{Z}}\left(x,\xi\right)\notin WF_{\Sigma}\left(e^{it\sqrt{\Delta_{g^{E},\mu}}}u\right)$
as required.
\end{proof}

\subsection{\label{subsec:Parametrices}Parametrix}

In this section we construct a small time parametrix for the half
wave operator; we work more generally to construct a parametrix for
$Ae^{it\sqrt{\Delta_{g^{E},\mu}}}$, $A\in\Psi_{\textrm{cl}}^{0,0}\left(X\right)$.
The operators $\sqrt{\Delta_{g^{E},\mu}}$ and $A$ being pseudo-differential,
and $\sqrt{\Delta_{g^{E},\mu}}$ elliptic outside the characteristic
variety, the parametrix construction is achieved by standard Hamilton-Jacobi
theory in the complement of $\Sigma$. It shall then suffice to construct
for each $p\in\Sigma$ a microlocal solution to
\begin{align}
i\partial_{t}P+\sqrt{\Delta_{g^{E},\mu}}P & \in\Psi_{\textrm{cl}}^{-\infty,0}\nonumber \\
\left.P\right|_{t=0}=A & \in\Psi_{\textrm{cl}}^{0,0},\label{eq:half wave IVP}
\end{align}
where we may further suppose $A\in\Psi_{\textrm{cl}}^{0,0}\left(X\right)$
to be micro-supported in a microlocal chart near $p$ where \prettyref{eq:normal form-1}
holds. We shall look for a solution of the form 
\begin{equation}
P\coloneqq\left[\sum_{k\in\mathbb{N}_{0}}H_{k}^{*}P_{k}H_{k}\right]A.\label{eq: Hermite FIO-1}
\end{equation}
Here each
\begin{align}
P_{k} & =\left[I_{\varphi}\left(a\right)\right]_{k}\coloneqq\int e^{i\left(\varphi_{k}-\underline{y}.\underline{\xi}\right)}a_{k}\left(t;\underline{x},\underline{\xi}\right)d\underline{\xi},\quad k\in\mathbb{N}_{0},\label{eq:form of parametrix}
\end{align}
with $a\in S_{\textrm{cl},t}^{0,0},\,\varphi\in\underline{x}.\underline{\xi}+tS_{\textrm{cl},t}^{1,-1},$
and each $\varphi_{k}$ solving the Hamilton-Jacobi equation 
\begin{align}
\partial_{t}\varphi_{k} & =d_{k}\left(x,\partial_{x}\varphi_{k}\right)\nonumber \\
\left.\varphi_{k}\right|_{t=0} & =\underline{x}.\underline{\xi}.\label{eq:Hamilton Jacobi eqn}
\end{align}
We first show that the above has a solution.
\begin{prop}
There exists a sufficiently small conic neighborhood of $p\in C\subset\left[\left(T^{*}X\right);\Sigma\right]$,
$T>0$ and 
\begin{equation}
\varphi\in\underline{x}.\underline{\xi}+tdC_{\textrm{inv}}^{\infty}\left(C\times\left(-T,T\right)\right),\label{eq:soln in the space}
\end{equation}
of homogeneous of degree one such that each corresponding $\varphi_{k}$
, $k\in\mathbb{N}_{0}$, solves the Hamilton Jacobi equation \prettyref{eq:Hamilton Jacobi eqn}.
\end{prop}
\begin{proof}
This is a modification of usual Hamilton-Jacobi theory. From the computation
\prettyref{eq:leading part Ham. vector field}, we may then choose
to work in a microlocal chart $C'$ at $p$ such that $e^{tH_{d}}\left(C\right)$,
$t\in\left(-T,T\right)$, stays in the chart for some sufficiently
small conic neighborhood $C\subset C'$ and $T>0$. With the notation
of \prettyref{eq:normal form}, $d\left(\underline{x},\underline{\xi},\Omega\right)$
being a function of the given variables with $\left\{ \Omega,d\right\} =0$,
the function $\Omega$ is preserved under the flow of $H_{d}$. One
thus has 
\begin{align}
e^{tH_{d}}\left(\underline{x},\underline{\xi};x_{1},\xi_{1}\right) & =\left(e^{tH_{d,\Omega}}\left(\underline{x},\underline{\xi}\right);\ast,\ast\right)\quad\textrm{ for }\nonumber \\
H_{d,\Omega} & \coloneqq\left(\partial_{\underline{\xi}}d\right)\left(\underline{x},\underline{\xi},\Omega\right)\partial_{\underline{x}}-\left(\partial_{\underline{x}}d\right)\left(\underline{x},\underline{\xi},\Omega\right)\partial_{\underline{\xi}}.\label{eq:Frozen Ham. v. field}
\end{align}
The vector field $H_{d,\Omega}$ above extends smoothly to the boundary
of the blowup $\left[\left(T^{*}X\right);\Sigma\right]$.

Given $\alpha\geq0$,$\underline{\xi}\in\mathbb{R}^{3}$, we now define
the flow-out 
\[
\Lambda_{\alpha,\underline{\xi}}\coloneqq\left\{ \left(e^{tH_{d,\alpha}}\left(\underline{x},\underline{\xi}\right),t,d\left(\underline{x},\underline{\xi},\alpha\right)\right)|\left(\underline{x},\xi_{3}^{-1/2}\alpha^{1/2},\underline{\xi},0\right)\in C,\,t\in\left(-T,T\right)\right\} \subset T^{*}\mathbb{R}_{\underline{x}}^{3}\times T_{\left(t,\tau\right)}^{*}\mathbb{R}.
\]
 By \prettyref{eq:identifications} and an application of Gronwall's
lemma, for $C$ and $T$ sufficiently small, the flow-out $\Lambda_{\alpha,\underline{\xi}}$
is horizontal above $\mathbb{R}_{\underline{x}}^{3}\times\mathbb{R}_{t}$
for $t\in\left(-T,T\right)$. Hence one may find a solution $\varphi_{\alpha,\underline{\xi}}\left(\underline{x},t\right)$
to
\begin{align}
\textrm{graph}\left(d\varphi_{\alpha,\underline{\xi}}\right) & \coloneqq\left\{ x,t,d_{\left(\underline{x},t\right)}\varphi_{\alpha,\underline{\xi}}\right\} =\Lambda_{\alpha,\underline{\xi}}\label{eq:graph}
\end{align}
The function $\varphi\left(x,\xi,t\right)\coloneqq\varphi_{\Omega,\underline{\xi}}\left(\underline{x},t\right)$
is then the required solution to \prettyref{eq:Hamilton Jacobi eqn};
its smoothness follows from the smooth extension of \prettyref{eq:Frozen Ham. v. field}
to the boundary. To see that the solution lies in the given space
\prettyref{eq:soln in the space} one needs to check $\varphi-\underline{x}.\underline{\xi}$,
or its pullback to the graph \prettyref{eq:graph}, vanishes on the
boundary $\partial\left[\left(T^{*}\mathbb{R}^{4}\right);\Sigma_{0}\right]$
at all time $t\in\left(-T,T\right)$. This follows from computing
\[
\pi^{*}d\left.\varphi\right|_{\partial\left[\left(T^{*}\mathbb{R}^{4}\right);\Sigma_{0}\right]}=\underbrace{\left.d\right|_{\partial\left[\left(T^{*}\mathbb{R}^{4}\right);\Sigma_{0}\right]}}_{=0}\left(dt\right)+\underbrace{\left.\left(\beta^{-1}\right)^{*}\alpha\left(H_{d,\Omega}\right)\right|_{\partial\left[\left(T^{*}\mathbb{R}^{4}\right);\Sigma_{0}\right]}}_{=0}
\]
from \prettyref{eq:identifications}, with $\alpha$ denoting the
tautological one form on $T^{*}\mathbb{R}^{4}$.
\end{proof}
Next we solve for the amplitude in \prettyref{eq:form of parametrix}.
Differentiation of \prettyref{eq:form of parametrix} using the symbolic
expansion of $\sqrt{\Delta_{g^{E},\mu}}=d+\Psi_{\textrm{cl}}^{0,-1/2}$
gives 
\begin{align*}
i\partial_{t}P+\sqrt{\Delta_{g^{E},\mu}}P & =I_{\varphi}\left(b\right)\quad\textrm{ with}\\
b & =\underbrace{\left(\partial_{t}+H_{d}\right)a+Ra}_{\eqqcolon La}
\end{align*}
and where $R$ maps $S_{\textrm{cl}}^{m_{1},m_{2}}$ to $S_{\textrm{cl}}^{m_{1}-1,m_{2}+1}$,
$\forall m_{1},m_{2}$. One may then write down a solution to the
transport equation $La=0$ (mod $S_{\textrm{cl}}^{-\infty}$) as 
\begin{equation}
a\left(t\right)\sim\sum_{j=0}^{\infty}a_{j}\left(t\right)\in S_{\textrm{cl},t}^{0,0},\label{eq:transport solution}
\end{equation}
$a_{j}\left(t\right)\in S_{\textrm{cl}}^{-j,j}$, starting from the
symbolic expansion $a\sim\sum_{j=0}^{\infty}a_{j}$, $a_{j}\in S_{\textrm{cl}}^{-j,j}$
for $A$ \prettyref{eq:half wave IVP} by inductively solving
\begin{align}
a_{0}\left(t\right) & =\left(e^{-tH_{d}}\right)^{*}a_{0};\nonumber \\
\left(\partial_{t}+H_{d}\right)a_{j}\left(t\right) & =-R\left[a_{0}+\ldots+a_{j-1}\right],\;a_{j}\left(0\right)=a_{j},\quad j\geq1.\label{eq: amplitude solution}
\end{align}

\section{\label{sec:Poisson-relations}Poisson relations}

In this section we prove the Poisson relation \prettyref{thm:Poisson relation}.
We more generally analyze the behavior of the microlocal wave trace
$\textrm{tr }Ae^{it\sqrt{\Delta_{g^{E},\mu}}}$, $A\in\Psi_{\textrm{cl}}^{0,0}\left(X\right)$,
for small time using the parametrix \prettyref{eq:form of parametrix}.
It again suffices to consider the wave trace near characteristic variety
and we may assume $A\in\Psi_{\textrm{cl}}^{0,0}\left(X\right)$ to
be micro-supported in a microlocal chart near $p$ where \prettyref{eq:normal form-1}
holds. By \prettyref{eq:Hamilton Jacobi eqn}, \prettyref{eq:soln in the space}
we have $\varphi-\underline{x}.\underline{\xi}=t\varphi^{0}$ with
$\varphi^{0}=\xi_{3}\hat{d}_{\rho}\left[1+t\underbrace{R\left(t;\underline{x},\underline{\xi},\hat{d}_{\rho}\right)}_{\in S_{\textrm{cl}}^{0,0}}\right]$.
On changing the $\xi_{3}$, $\xi_{0}$ variables to the new variables
$r=\xi_{3}d_{k},\,\Xi_{0}=\frac{\xi_{0}}{d_{k}}$ the wave trace 
\begin{align*}
\textrm{tr }Ae^{it\sqrt{\Delta_{g^{E},\mu}}}= & \sum_{k\in\mathbb{N}_{0}}\int e^{i\left(tr+t^{2}rR_{k}\right)}a_{k}\left(t\right)\frac{\left(1-\Xi_{0}^{2}\right)r^{4}}{\hat{\rho}^{2}\left(2k+1\right)^{2}}dr\,d\Xi_{0}d\hat{\xi}_{2}d\underline{x}\;\;\left(\textrm{mod }t^{\infty}\right),
\end{align*}
in the distributional sense. Since the amplitude was shown to be in
the class $a\in S_{\textrm{cl},t}^{0,0}$, the wave trace mod $t^{N-5}$
is then a finite sum of terms of the form 
\begin{align}
 & \int\sum_{k\in\mathbb{N}_{0}}e^{irt}t^{2\alpha}r^{\alpha-\beta}a_{\alpha,\beta}\left(\underline{x},\hat{\xi}_{2},\Xi_{0};d_{k}\right)\frac{\left(1-\Xi_{0}^{2}\right)r^{4}}{\hat{\rho}^{2}\left(2k+1\right)^{2}}dr\,d\Xi_{0}d\hat{\xi}_{2}d\underline{x};\nonumber \\
 & \qquad\qquad\qquad\qquad\alpha,\beta\in\mathbb{N}_{0},\,0\leq\alpha+\beta\leq N,\label{eq: parametrix term}
\end{align}
with
\begin{align}
a_{\alpha,\beta}\left(\underline{x},\hat{\xi}_{2},\Xi_{0};d_{k}\right) & =b_{\alpha,\beta}\left(\underline{x},\hat{\xi}_{2},\Xi_{0}\left(d_{k}+1\right),d_{k}+1\right)\label{eq: form of amplitude}
\end{align}
for $b_{\alpha,\beta}\left(\underline{x},\hat{\xi}_{2},p_{0},d\right)\in C_{c}^{\infty}\left(\mathbb{R}_{\underline{x},\hat{\xi}_{2},p_{0}}^{5}\times\left[1,\infty\right)_{d}\right)$.
Furthermore the leading part 
\begin{equation}
a_{0,0}=\sigma^{H}\left(A\right)\left(\underline{x},\hat{\xi}_{2},\Xi_{0};d_{k}\right).\label{eq:leading term parametrix}
\end{equation}
We now show how to sum the above in $k$ with the help of the proposition
below.
\begin{prop}
Given $b\in C_{c}^{\infty}\left(\mathbb{R}_{\underline{x},\hat{\xi}_{2},p_{0}}^{5}\times\left[1,\infty\right)_{d}\right)$
and $a$ defined as in \prettyref{eq: form of amplitude}, the expression
\begin{align}
I\left(r\right) & \coloneqq\int d\Xi_{0}d\hat{\xi}_{2}d\underline{x}\left[\sum_{k\in\mathbb{N}_{0}}\frac{\left(1-\Xi_{0}^{2}\right)}{\hat{\rho}^{2}\left(2k+1\right)^{2}}a\left(\underline{x},\hat{\xi}_{2},\Xi_{0};d_{k}\right)\right]\nonumber \\
 & \sim cr^{-1}\ln r+c_{0}+c_{1}r^{-1}+c_{2}r^{-2}+\ldots;\label{eq:sum symbol}\\
c_{0} & =\frac{\pi^{2}}{8}\int d\Xi_{0}d\hat{\xi}_{2}d\underline{x}\frac{\left(1-\Xi_{0}^{2}\right)}{\hat{\rho}^{2}}b\left(\underline{x},\hat{\xi}_{2},\Xi_{0},0\right)\nonumber \\
c & =\frac{1}{2}\int d\Xi_{0}d\hat{\xi}_{2}d\underline{x}\frac{\left(1-\Xi_{0}^{2}\right)}{\hat{\rho}}\left(\Xi_{0}\partial_{p_{0}}+\partial_{d}\right)b\left(\underline{x},\hat{\xi}_{2},\Xi_{0},0\right)\label{eq:computation of first terms}
\end{align}
is the sum of a classical symbol in $r$ of order $0$ and a log term
$r^{-1}\ln r$ .
\end{prop}
\begin{proof}
The symbolic estimates are easily seen on differentiation and noting
$d_{k}=\frac{\hat{\rho}\left(2k+1\right)}{r\left(1-\Xi_{0}^{2}\right)}$
to be a symbol of order $-1$ in $r$ in the region $d_{k}\apprle1$,
$\Xi_{0}\apprle1$. To show a classical expansion, we perform the
change of variables $\alpha=\left(1-\Xi_{0}^{2}\right)^{-1}$ in the
$\Xi_{0}$ integration to obtain 
\[
I\left(r\right)=\sum_{k\in\mathbb{N}_{0}}\int\frac{d\hat{\xi}_{2}d\underline{x}}{\hat{\rho}^{2}\left(2k+1\right)^{2}}\left[\underbrace{\int_{1}^{\infty}\frac{d\alpha}{\alpha^{3}\sqrt{1-\alpha^{-1}}}a\left(\underline{x},\hat{\xi}_{2},\sqrt{1-\alpha^{-1}};\frac{\hat{\rho}\left(2k+1\right)}{r}\alpha\right)}_{\eqqcolon I_{0}\left(\underline{x},\hat{\xi}_{2};\frac{\hat{\rho}\left(2k+1\right)}{r}\right)}\right]
\]
with the integral in parentheses above seen to be $I_{0}\left(\underline{x},\hat{\xi}_{2};\frac{\hat{\rho}\left(2k+1\right)}{r}\right)\in C_{c}^{\infty}\left(\mathbb{R}_{\underline{x},\hat{\xi}_{2},\frac{\hat{\rho}\left(2k+1\right)}{r}}^{5}\right)$. 

From here the proposition follows from \cite[Prop. 7.20]{Melrose-hypoelliptic}
but we give a shorter argument. Differentiating $I_{1}\left(\varepsilon\right)\coloneqq\sum_{k\in\mathbb{N}_{0}}\frac{1}{\hat{\rho}^{2}\left(2k+1\right)^{2}}I_{0}\left(\underline{x},\hat{\xi}_{2};\varepsilon\hat{\rho}\left(2k+1\right)\right)$,
$\varepsilon\coloneqq\frac{1}{r}$, gives 
\begin{align*}
\partial_{\varepsilon}^{2}I_{1} & =\sum_{k\in\mathbb{N}_{0}}\left(\partial_{d}^{2}I_{0}\right)\left(\underline{x},\hat{\xi}_{2};\varepsilon\hat{\rho}\left(2k+1\right)\right)\\
 & =\int_{0}^{\infty}dk\,\left(\partial_{d}^{2}I_{0}\right)\left(\underline{x},\hat{\xi}_{2};\varepsilon\hat{\rho}\left(2k+1\right)\right)\\
 & \;+\frac{1}{2}\sum_{l\in\mathbb{N}\setminus\left\{ 0\right\} }\int dk\,e^{i2\pi kl}\left(\partial_{d}^{2}I_{0}\right)\left(\underline{x},\hat{\xi}_{2};\varepsilon\hat{\rho}\left(2k+1\right)\right)
\end{align*}
by the Poisson summation formula. By repeated integration by parts
the second term in the last line above is seen to be $O\left(\varepsilon^{\infty}\right)$,
while the first term is evaluated to be 

\begin{align}
\int_{0}^{\infty}dk\,\left(\partial_{d}^{2}I_{0}\right)\left(\underline{x},\hat{\xi}_{2};\varepsilon\hat{\rho}\left(2k+1\right)\right) & \sim c\varepsilon^{-1}+c_{0}+c_{1}\varepsilon+\ldots\nonumber \\
c & =\frac{1}{2}\hat{\rho}^{-1}\left[\partial_{d}I_{0}\left(\underline{x},\hat{\xi}_{2};0\right)\right]\label{eq:computation first terms}
\end{align}
to complete the proof. 
\end{proof}
Following this proposition, we may further simplify \prettyref{eq: parametrix term}
as being mod $t^{N-5}$ a sum of terms of the form 
\begin{align*}
 & t^{2\alpha}\int_{0}^{\infty}dr\,e^{irt}r^{4-j+\alpha-\beta};\\
\textrm{or }\quad & t^{2\alpha}\int_{0}^{\infty}dr\,e^{irt}r^{3-j+\alpha-\beta}\ln r\,;
\end{align*}
$\alpha,\beta,j\in\mathbb{N}_{0},\,\alpha+\beta\leq N$. Using the
identifications \prettyref{eq:identifications}, the knowledge of
these elementary Fourier transforms and identifying the constants
we now have the following. 
\begin{thm}
\label{thm: Microlocal wave trace} For any $A\in\Psi_{\textrm{cl}}^{0,0}$,
the microlocal wave trace in the 4D quasi-contact case has the asymptotics
\begin{align}
\textrm{tr }Ae^{it\sqrt{\Delta_{g^{E},\mu}}} & =\sum_{j=0}^{N}c_{j,0}^{A}\left(t+i0\right)^{j-5}+\sum_{j=0}^{N}c_{j,1}^{A}\left(t+i0\right)^{j-4}\ln\left(t+i0\right)+\sum_{j=0}^{N}c_{j,2}^{A}t^{j}\ln^{2}\left(t+i0\right)+O\left(t^{N-4}\right)\label{eq:microlocal wave trace}
\end{align}
 $\forall N\in\mathbb{N}$, as $t\rightarrow0$, in the distributional
sense with leading term $c_{0,0}^{A}=\frac{1}{32\pi}\int\left.\sigma\left(A\right)\right|_{SNS^{*}\Sigma}\mu_{\textrm{Popp}}^{SNS^{*}\Sigma}$. 
\end{thm}
In the case when $A=1$ one has $b_{0,0}=1$ in \prettyref{eq: form of amplitude}
which following \prettyref{eq:computation of first terms} gives that
the first logarithmic term above vanishes $c_{0,1}^{1}=0$ proving
\prettyref{thm:Poisson relation}. Pairing \prettyref{eq:microlocal wave trace}
with $\theta\in C_{c}^{\infty}\left(-C_{0},C_{0}\right)$, for $C_{0}$
sufficiently small, gives 
\begin{align}
\textrm{tr }A\check{\theta}\left(\sqrt{\Delta_{g^{E},\mu}}-\lambda\right) & =\sum_{j=0}^{N}\tilde{c}_{j,0}^{A}\lambda^{4-j}+\sum_{j=0}^{N}\tilde{c}_{j,1}^{A}\lambda^{3-j}\ln\lambda+O\left(\lambda^{4-N}\right)\label{eq:small time microlocal trace}\\
\textrm{tr }\check{\theta}\left(\sqrt{\Delta_{g^{E},\mu}}-\lambda\right) & =\sum_{j=0}^{N}\tilde{c}_{j,0}\lambda^{4-j}+\sum_{j=0}^{N}\tilde{c}_{j,1}\lambda^{2-j}\ln\lambda+O\left(\lambda^{4-N}\right)\label{eq:small time trace}
\end{align}
$\forall N\in\mathbb{N}$ as $\lambda\rightarrow\infty$ with leading
terms 
\begin{align*}
\tilde{c}_{j,0}^{A} & =\frac{\theta\left(0\right)}{32\pi}\int\left.\sigma\left(A\right)\right|_{SNS^{*}\Sigma}\mu_{\textrm{Popp}}^{SNS^{*}\Sigma},\\
\tilde{c}_{j,0}^{1} & =\frac{\theta\left(0\right)}{32\pi}\int\left[\int_{-1}^{1}d\Xi_{0}\left(1-\Xi_{0}^{2}\right)\right]\mu_{\textrm{Popp}}=\frac{\theta\left(0\right)}{24\pi}\mu_{\textrm{Popp}}.
\end{align*}
We now prove the Weyl laws \prettyref{thm:Sharp Weyl law}.
\begin{proof}[Proof of \prettyref{thm:Sharp Weyl law}]
 Following a standard Tauberian theorem for Fourier transforms (cf.
\cite[Sec. 2]{Duistermaat-Guillemin}) \prettyref{eq:small time trace}
gives \prettyref{eq:sharp Weyl law}. To prove \prettyref{eq:improved Weyl law}
one needs to prove \prettyref{eq:small time trace} at leading order
for $\theta\in C_{c}^{\infty}\left(\mathbb{R}\right)$ of arbitrary
support under the dynamical assumption. 

We first consider the trace norm of $A\check{\theta}\left(\sqrt{\Delta_{g^{E},\mu}}-\lambda\right)$,
$A=a^{H}\in\Psi_{\textrm{cl}}^{0,0}$, $a\in C_{c,\textrm{inv}}^{\infty}\left(\left[\left(S^{*}X\right);S^{*}\Sigma\right];\left[0,1\right]\right)$,
for $\theta\in C_{c}^{\infty}\left(-C_{0},C_{0}\right)$ assuming
$\check{\theta}\geq0$. To this end, let $\tilde{a}\in C_{c,\textrm{inv}}^{\infty}\left(\left[\left(S^{*}X\right);S^{*}\Sigma\right];\left[0,1\right]\right)$
such that $\tilde{a}=1$ on $\cup_{t\in\left(-C_{0},C_{0}\right)}e^{tH_{d}}\left(\textrm{spt}a\right)$.
Then an Egorov type argument \prettyref{thm:Egororovb thm} gives
\begin{align}
a^{H}e^{it\left(\sqrt{\Delta_{g^{E},\mu}}-\lambda\right)}\left(1-\tilde{a}^{H}\right) & \in\Psi_{\textrm{cl}}^{-\infty},\quad\forall t\in\left(-C_{0},C_{0}\right),\nonumber \\
\left\Vert a^{H}\check{\theta}\left(\sqrt{\Delta_{g^{E},\mu}}-\lambda\right)\left(1-\tilde{a}^{H}\right)\right\Vert _{\textrm{tr}} & =O\left(\lambda^{-\infty}\right)\label{eq:Egorov estimate}
\end{align}
and thus
\begin{align*}
\left\Vert a^{H}\check{\theta}\left(\sqrt{\Delta_{g^{E},\mu}}-\lambda\right)\right\Vert _{\textrm{tr}} & =\left\Vert a^{H}\check{\theta}\left(\sqrt{\Delta_{g^{E},\mu}}-\lambda\right)\tilde{a}^{H}\right\Vert _{\textrm{tr}}+O\left(\lambda^{-\infty}\right)\\
 & =\left\Vert a^{H}\tilde{a}^{H}\check{\theta}\left(\sqrt{\Delta_{g^{E},\mu}}-\lambda\right)\tilde{a}^{H}\right\Vert _{\textrm{tr}}+O\left(\lambda^{-\infty}\right)
\end{align*}
Now since $\left|a\right|<1+\varepsilon$, $\forall\varepsilon>0$,
we may use symbolic calculus to write $a^{H}=1+\varepsilon-\left(b^{H}\right)^{2}+\Psi_{\textrm{cl}}^{-\infty}$.
This gives 
\begin{align*}
\left\Vert a^{H}\check{\theta}\left(\sqrt{\Delta_{g^{E},\mu}}-\lambda\right)\right\Vert _{\textrm{tr}} & =\left\Vert \left(1+\varepsilon-\left(b^{H}\right)^{2}\right)\check{\theta}\left(\sqrt{\Delta_{g^{E},\mu}}-\lambda\right)\tilde{a}^{H}\right\Vert _{\textrm{tr}}+O\left(\lambda^{-\infty}\right)\\
 & \leq\left(1+\varepsilon\right)\left\Vert \tilde{a}^{H}\check{\theta}\left(\sqrt{\Delta_{g^{E},\mu}}-\lambda\right)\tilde{a}^{H}\right\Vert _{\textrm{tr}}+O\left(\lambda^{-\infty}\right)
\end{align*}
Next for $\check{\theta}\geq0$, the operator $\tilde{a}^{H}\check{\theta}\left(\sqrt{\Delta_{g^{E},\mu}}-\lambda\right)\tilde{a}^{H}$
being positive and self-adjoint, its trace norm coincides with its
trace which is in turn analyzed in a similar fashion to \prettyref{eq:small time microlocal trace}.
Hence 
\begin{equation}
\left\Vert a^{H}\check{\theta}\left(\sqrt{\Delta_{g^{E},\mu}}-\lambda\right)\right\Vert _{\textrm{tr}}\leq\left(1+\varepsilon\right)\lambda^{4}\left[C_{\theta}\int\left.\left|\tilde{a}\right|^{2}\right|_{SNS^{*}\Sigma}\nu_{\textrm{Popp}}^{SNS^{*}\Sigma}\right]+O_{a,\tilde{a},\theta,\varepsilon}\left(\lambda^{3}\right).\label{eq:important trace estimate}
\end{equation}
To remove the condition $\check{\theta}\geq0$ on the Fourier transform
of the cutoff, one may choose $\phi\in C_{c}^{\infty}\left(-C_{0},C_{0}\right)$
satisfying $\phi>0$ on $\textrm{spt}\left(\theta\right)$ and $\check{\phi}\geq0$.
Then writing $\theta=g\phi$, $g\in C_{c}^{\infty}\left(-C_{0},C_{0}\right)$
gives $\check{\theta}=\check{g}\ast\check{\phi}$ and 
\begin{align*}
\left\Vert a^{H}\check{\theta}\left(\sqrt{\Delta_{g^{E},\mu}}-\lambda\right)\right\Vert _{\textrm{tr}} & =\int_{\left|\lambda'\right|<\lambda^{1/2}}d\lambda'\left|\check{g}\left(\lambda'\right)\right|\left\Vert a^{H}\check{\phi}\left(\sqrt{\Delta_{g^{E},\mu}}-\lambda-\lambda'\right)\right\Vert _{\textrm{tr}}\\
 & \quad+\int_{\left|\lambda'\right|>\lambda^{1/2}}d\lambda'\left|\check{g}\left(\lambda'\right)\right|\left\Vert a^{H}\check{\phi}\left(\sqrt{\Delta_{g^{E},\mu}}-\lambda-\lambda'\right)\right\Vert _{\textrm{tr}}
\end{align*}
The second integral above is $O\left(\lambda^{-\infty}\right)$. The
first integral is then estimated following the corresponding estimate
\prettyref{eq:important trace estimate} for $\check{\phi}$. To remove
the condition on $\textrm{spt}\theta$, we may write a function of
arbitrary support as a sum of translates $\theta_{c}\left(s\right)=\theta\left(s-c\right)\in C_{c}\left(\mathbb{R}\right)$,
$c\in\mathbb{R}$, of functions supported near zero. Then 
\[
\left\Vert a^{H}\check{\theta_{c}}\left(\sqrt{\Delta_{g^{E},\mu}}-\lambda\right)\right\Vert _{\textrm{tr}}=\left\Vert a^{H}e^{-ic\left(\sqrt{\Delta_{g^{E},\mu}}-\lambda\right)}\check{\theta}\left(\sqrt{\Delta_{g^{E},\mu}}-\lambda\right)\right\Vert _{\textrm{tr}}=\left\Vert a^{H}\check{\theta}\left(\sqrt{\Delta_{g^{E},\mu}}-\lambda\right)\right\Vert _{\textrm{tr}}
\]
 gives \prettyref{eq:important trace estimate} for any arbitrary
$\theta\in C_{c}^{\infty}\left(\mathbb{R}\right)$. 

We now come to estimating the trace \prettyref{eq:small time trace}
for arbitrary $\theta\in C_{c}^{\infty}\left(\mathbb{R}\right)$.
Splitting $\theta=\underbrace{\vartheta}_{\in C_{c}\left(-C_{0},C_{0}\right)}+\underbrace{\left(\theta-\vartheta\right)}_{\in C_{c}\left(\mathbb{R}\setminus\left(-\frac{C_{0}}{2},\frac{C_{0}}{2}\right)\right)}$,
with the trace $\textrm{tr }\check{\vartheta}\left(\sqrt{\Delta_{g^{E},\mu}}-\lambda\right)$
expanded as \prettyref{eq:small time trace}, we next estimate $\textrm{tr }\left(\check{\theta}-\check{\vartheta}\right)\left(\sqrt{\Delta_{g^{E},\mu}}-\lambda\right)$.
Under the assumption on $L^{E}$, we may find $\forall\varepsilon>0$
a microlocal partition of unity $\left\{ A_{j}=a_{j}^{H}\right\} _{j=-1}^{N}\in\Psi_{\textrm{cl}}^{0,0}\left(X\right)$,
$\left\{ a_{j}\right\} _{j=-1}^{N}\in C_{\textrm{inv}}^{\infty}\left(\left[\left(S^{*}X\right);S^{*}\Sigma\right];\left[0,1\right]\right)$,
$\sum_{j=-1}^{N}a_{j}=1$, satisfying 
\begin{align}
\textrm{spt}a_{-1}\cap SNS^{*}\Sigma & =\emptyset,\nonumber \\
\mu_{\textrm{Popp}}^{SNS^{*}\Sigma}\left(\textrm{spt}a_{0}\cap SNS^{*}\Sigma\right) & \leq\varepsilon,\nonumber \\
\left[\cup_{t\in\textrm{spt}\left(\theta-\vartheta\right)}e^{tH_{d}}\left(\textrm{spt}a_{j}\right)\right]\cap\left(\textrm{spt}a_{j}\right) & =\emptyset,\quad1\leq j\leq N.\label{eq:suitable partition}
\end{align}
The estimate \prettyref{eq:important trace estimate} gives 
\begin{equation}
\left|\textrm{tr }\left(a_{-1}^{H}+a_{0}^{H}\right)\left(\check{\theta}-\check{\vartheta}\right)\left(\sqrt{\Delta_{g^{E},\mu}}-\lambda\right)\right|\leq\varepsilon\left(1+\varepsilon\right)\lambda^{4}+O_{\theta,\varepsilon}\left(\lambda^{3}\right),\:\forall\varepsilon>0.\label{eq: first cutoff contribution}
\end{equation}
Furthermore, choosing $\tilde{a}_{j}\in C_{\textrm{inv}}^{\infty}\left(\left[\left(S^{*}X\right);S^{*}\Sigma\right];\left[0,1\right]\right)$,
$1\leq j\leq N$, with 
\begin{align}
\textrm{spt}\tilde{a}_{j}\cap\textrm{spt}a_{j} & =\emptyset\nonumber \\
\tilde{a}_{j} & =1\textrm{ on }\left[\cup_{t\in\textrm{spt}\left(\theta-\vartheta\right)}e^{tH_{d}}\left(\textrm{spt}a_{j}\right)\right]\label{eq: enlarged cutoff}
\end{align}
gives 
\begin{align}
\textrm{tr }\left[a_{j}^{H}\left(\check{\theta}-\check{\vartheta}\right)\left(\sqrt{\Delta_{g^{E},\mu}}-\lambda\right)\right] & =\textrm{tr }\left[a_{j}^{H}\left(\check{\theta}-\check{\vartheta}\right)\left(\sqrt{\Delta_{g^{E},\mu}}-\lambda\right)\tilde{a}_{j}^{H}\right]\nonumber \\
 & \quad+\textrm{tr }\left[a_{j}^{H}\left(\check{\theta}-\check{\vartheta}\right)\left(\sqrt{\Delta_{g^{E},\mu}}-\lambda\right)\left(1-\tilde{a}_{j}^{H}\right)\right]\nonumber \\
 & =\textrm{tr }\left[\left(\check{\theta}-\check{\vartheta}\right)\left(\sqrt{\Delta_{g^{E},\mu}}-\lambda\right)\tilde{a}_{j}^{H}a_{j}^{H}\right]\nonumber \\
 & \quad+\textrm{tr }\left[a_{j}^{H}\left(\check{\theta}-\check{\vartheta}\right)\left(\sqrt{\Delta_{g^{E},\mu}}-\lambda\right)\left(1-\tilde{a}_{j}^{H}\right)\right]\nonumber \\
 & =O\left(\lambda^{-\infty}\right)\label{eq: other cutoffs small contribution}
\end{align}
following a similar Egorov argument as in \prettyref{eq:Egorov estimate}
and \prettyref{eq: enlarged cutoff}. Thus finally combining \prettyref{eq:small time microlocal trace},
\prettyref{eq: first cutoff contribution} and \prettyref{eq: other cutoffs small contribution}
we have 
\[
\textrm{tr }\check{\theta}\left(\sqrt{\Delta_{g^{E},\mu}}-\lambda\right)=\lambda^{4}\frac{\theta\left(0\right)}{24\pi}\mu_{\textrm{Popp}}+o\left(\lambda^{4}\right)
\]
for any $\theta\in C_{c}^{\infty}\left(\mathbb{R}\right)$, under
the assumption on the closed integral curves of $L^{E}$. Following
the above the usual Tauberian argument of continues to prove \prettyref{eq:improved Weyl law};
cf. \cite[Sec. 2]{Duistermaat-Guillemin} or \cite[Ch. 11]{Dimassi-Sjostrand}.
\end{proof}
Next we prove the large time Poisson relation \prettyref{eq: sing spt wave trace}. 
\begin{proof}[Proof of \prettyref{eq: sing spt wave trace}]
 We shall infact prove the stronger statement 
\begin{align}
\textrm{sing spt}\left(\textrm{tr }e^{it\sqrt{\Delta_{g^{E},\mu}}}\right) & \subset\left\{ 0\right\} \cup\mathscr{L}_{\hat{Z}}\cup\mathscr{L}_{\textrm{normal}}\nonumber \\
 & \subset\left\{ 0\right\} \cup\left(-\infty,-T_{\textrm{abnormal}}^{E}\right]\cup\left[T_{\textrm{abnormal}}^{E},\infty\right)\cup\mathscr{L}_{\textrm{normal}}\label{eq: sing spt wave trace refined statement}
\end{align}
with $\mathscr{L}_{\hat{Z}}$ as in the computation \prettyref{eq:period spectrum lift of charac.}.
Equivalently stated, the above \prettyref{eq: sing spt wave trace refined statement}
amounts to
\[
\textrm{tr }\check{\theta}\left(\sqrt{\Delta_{g^{E},\mu}}-\lambda\right)=O\left(\lambda^{-\infty}\right)
\]
for $\textrm{spt}\left(\theta\right)\Subset\mathbb{R}\setminus\left(\left\{ 0\right\} \cup\mathscr{L}_{\hat{Z}}\cup\mathscr{L}_{\textrm{normal}}\right)$.
We may then again choose a microlocal partition of unity $\left\{ A_{j}=a_{j}^{H}\right\} _{j=1}^{N}\in\Psi_{\textrm{cl}}^{0,0}\left(X\right)$,
$\left\{ a_{j}\right\} _{j=1}^{N}\in C_{\textrm{inv}}^{\infty}\left(\left[\left(S^{*}X\right);S^{*}\Sigma\right];\left[0,1\right]\right)$,
$\sum_{j=1}^{N}a_{j}=1$, satisfying 
\begin{align}
\left[\cup_{t\in\textrm{spt}\left(\theta\right)}e^{tH_{d}}\left(\textrm{spt}a_{j}\right)\right]\cap\left(\textrm{spt}a_{j}\right) & =\emptyset,\quad1\leq j\leq N.\label{eq:suitable partition-1}
\end{align}
Furthermore, again choosing $\tilde{a}_{j}\in C_{\textrm{inv}}^{\infty}\left(\left[\left(S^{*}X\right);S^{*}\Sigma\right];\left[0,1\right]\right)$,
$1\leq j\leq N$, with 
\begin{align}
\textrm{spt}\tilde{a}_{j}\cap\textrm{spt}a_{j} & =\emptyset\nonumber \\
\tilde{a}_{j} & =1\textrm{ on }\left[\cup_{t\in\textrm{spt}\left(\theta\right)}e^{tH_{d}}\left(\textrm{spt}a_{j}\right)\right]\label{eq: enlarged cutoff-1}
\end{align}
gives 
\begin{align}
\textrm{tr }\left[a_{j}^{H}\check{\theta}\left(\sqrt{\Delta_{g^{E},\mu}}-\lambda\right)\right] & =\textrm{tr }\left[a_{j}^{H}\check{\theta}\left(\sqrt{\Delta_{g^{E},\mu}}-\lambda\right)\tilde{a}_{j}^{H}\right]+\textrm{tr }\left[a_{j}^{H}\check{\theta}\left(\sqrt{\Delta_{g^{E},\mu}}-\lambda\right)\left(1-\tilde{a}_{j}^{H}\right)\right]\nonumber \\
 & =\textrm{tr }\left[\check{\theta}\left(\sqrt{\Delta_{g^{E},\mu}}-\lambda\right)\tilde{a}_{j}^{H}a_{j}^{H}\right]+\textrm{tr }\left[a_{j}^{H}\check{\theta}\left(\sqrt{\Delta_{g^{E},\mu}}-\lambda\right)\left(1-\tilde{a}_{j}^{H}\right)\right]\nonumber \\
 & =O\left(\lambda^{-\infty}\right)\label{eq: other cutoffs small contribution-1}
\end{align}
following a similar Egorov argument as in \prettyref{eq:Egorov estimate}
and \prettyref{eq: enlarged cutoff-1}. 
\end{proof}

\section{\label{sec:Quantum-ergodicity}Quantum ergodicity}

In this section we prove the quantum ergodicity theorem for the sR
Laplacian \prettyref{thm:QE theorem}. As usual (see for instance
\cite{Zelditch2009}), it is enough to establish a microlocal Weyl
law 
\begin{align}
E\left(B\right)\coloneqq & \lim_{\lambda\rightarrow\infty}\frac{1}{N\left(\lambda\right)}\sum_{\lambda_{j}\leq\lambda}\left\langle B\varphi_{j},\varphi_{j}\right\rangle \nonumber \\
= & \frac{1}{2}\int d\nu_{\textrm{Popp}}\left[b\left(x,a_{g}\left(x\right)\right)+b\left(x,-a_{g}\left(x\right)\right)\right],\label{eq: micro-weyl}
\end{align}
and variance estimate 
\begin{equation}
V\left(B\right)\coloneqq\lim_{\lambda\rightarrow\infty}\frac{1}{N\left(\lambda\right)}\sum_{\lambda_{j}\leq\lambda}\left|\left\langle \left[B-E\left(B\right)\right]\varphi_{j},\varphi_{j}\right\rangle \right|^{2}=0,\label{eq:variance}
\end{equation}
given $B\in\Psi_{\textrm{cl}}^{0}\left(X\right),$ with $b=\sigma\left(B\right)$. 

\subsection{\label{subsec:Microlocal-Weyl-laws}Microlocal Weyl laws}

We begin with the microlocal Weyl law \prettyref{eq: micro-weyl}.
The upcoming \prettyref{lem:Microlocal weyl law} in fact works more
generally on any equiregular sR manifold; a more detailed discussion
of it including some singular (non-equiregular) analysis will appear
in \cite{Colin-de-Verdiere-Hillairet-TrelatII}.

We first prove a localization result for the heat kernel of the sR
Laplacian on a general sR manifold $X$ of dimension $n$. To state
this, given point a $x\in X$ we choose a privileged coordinate chart
contained inside the open ball $B_{\varrho}^{g^{TX}}\left(x\right)\coloneqq\left\{ x'|d^{g^{TX}}\left(x,x'\right)<\varrho\right\} $;
where $g^{TX}$ denotes a fixed Riemannian metric on $X$ and $\varrho$
depends on $x$. Let $\chi\in C_{c}^{\infty}\left(\left[-1,1\right];\left[0,1\right]\right)$
with $\chi=1$ on $\left[-\frac{1}{2},\frac{1}{2}\right]$. Choose
a local orthonormal frame $U_{1},\ldots,U_{k}$ for $E$ and define
\begin{align*}
\tilde{U}_{j} & =\hat{U}_{j}^{\left(-1\right)}+\chi\left(\frac{d^{g^{TX}}\left(x,x'\right)}{\varrho_{x}}\right)\left(U_{j}-\hat{U}_{j}^{\left(-1\right)}\right),\quad\forall1\leq j\leq k,\\
\tilde{\mu} & =\hat{\mu}+\chi\left(\frac{d^{g^{TX}}\left(x,x'\right)}{\varrho_{x}}\right)\left(\mu-\hat{\mu}\right),
\end{align*}
to be the modified vector fields and volume on $\mathbb{R}^{n}$.
Here $\hat{U}_{j}^{\left(-1\right)}$, $\hat{\mu}$ are the first
terms in the homogeneous privileged coordinate expansions of $U_{j}$
\prettyref{eq:priv cord. exp v. field} and the volume $\mu$ \prettyref{eq: prov cord exp. measure}
respectively. For $\varrho$ sufficiently small, the $\tilde{U}_{j}$'s
are linearly independent and bracket generating with degree of nonholonomy
being $r\left(x\right)$. A formula similar to \prettyref{eq: sR laplacian in orthonormal frame}
now gives an sR Laplacian on $\mathbb{R}^{n}$ via $\tilde{\Delta}_{g,\mu}f=\sum_{j=1}^{k}\left[-\left(\tilde{U}_{j}\right)^{2}\left(f\right)+\tilde{U}_{j}\left(f\right)\left(\textrm{div}_{\tilde{\mu}}\tilde{U}_{j}\right)\right]$.
The operator $\tilde{\Delta}_{g,\mu}$ is again essentially self-adjoint
with a resolvent that maps $\left(\tilde{\Delta}_{g,\mu}-z\right)^{-1}:H_{\textrm{loc}}^{s}\rightarrow H_{\textrm{loc}}^{s+\frac{1}{r}}$
\prettyref{eq:subelliptic estimate} and has a well defined functional
calculus. We now have the following localization lemma.
\begin{lem}
\label{lem: Localization lemma} The heat kernel satisfies 
\begin{equation}
\left[e^{-t\Delta_{g^{E},\mu}}\left(x,x'\right)\right]_{\mu}=ct^{-2nr-1}e^{-\frac{d^{E}\left(x,x'\right)^{2}}{4t}}\label{eq: heat kernel estimate}
\end{equation}
uniformly for $t\leq1$. 

Moreover, there exists $\varrho_{1}\left(x\right)>0$ such that
\begin{equation}
\left\Vert \left[e^{-t\Delta_{g^{E},\mu}}\right]_{\mu}\left(.,x\right)-\left[e^{-t\tilde{\Delta}_{g,\mu}}\right]_{\tilde{\mu}}\left(.,0\right)\right\Vert _{C^{k}\left(X\right)}=C_{x,k}e^{-\frac{\varrho_{1}^{2}}{16t}}\label{eq: heat kernel localization}
\end{equation}
have the same asymptotics for $d^{E}\left(x,x'\right)\leq\varrho_{1}$
as $t\rightarrow0$. 
\end{lem}
\begin{proof}
Both claims follow from the finite propagation result \prettyref{thm:(Finite-propagation-speed)}
and the Fourier transformation formula 
\begin{align}
\left[\Delta_{g^{E},\mu}^{q}e^{-t\Delta_{g^{E},\mu}}\right]_{\mu}\left(x,x'\right) & =\frac{1}{2\pi}\int d\xi\,\left[e^{i\xi\sqrt{\Delta_{g^{E},F,\mu}}}\left(x,x'\right)\right]_{\mu}D_{\xi}^{2q}\frac{e^{-\frac{\xi^{2}}{4t}}}{\sqrt{4\pi t}}\nonumber \\
 & =\frac{1}{2\pi}\int d\xi\,\left[e^{i\xi\sqrt{\Delta_{g^{E},F,\mu}}}\left(x,x'\right)\right]_{\mu}\chi\left(\frac{\xi}{\varrho_{1}}\right)D_{\xi}^{2q}\frac{e^{-\frac{\xi^{2}}{4t}}}{\sqrt{4\pi t}}\nonumber \\
 & \;\;+\frac{1}{2\pi}\int d\xi\,\left[e^{i\xi\sqrt{\Delta_{g^{E},F,\mu}}}\left(x,x'\right)\right]_{\mu}\left[1-\chi\left(\frac{\xi}{\varrho_{1}}\right)\right]D_{\xi}^{2q}\frac{e^{-\frac{\xi^{2}}{4t}}}{\sqrt{4\pi t}},\label{eq:breakup integral}
\end{align}
$\forall q\in\mathbb{N}_{0}$, $\varrho_{1}>0$. By finite propagation,
the integral maybe restricted to $\left|\xi\right|\geq d^{E}\left(x,x'\right)$.
Now the integral estimate 
\[
\left|\int_{\left|\xi\right|\geq d^{E}\left(x,x'\right)}e^{i\xi s}D_{\xi}^{2q}e^{-\frac{\xi^{2}}{4t}}d\xi\right|\leq ct^{-2q-\frac{1}{2}}e^{-\frac{d^{E}\left(x,x'\right)^{2}}{4t}}
\]
gives the bound on $\left\Vert \Delta_{g,\mu}^{q}e^{-t\Delta_{g^{E},F,\mu}}\right\Vert _{L^{2}\rightarrow L^{2}}\leq ct^{-2q-\frac{1}{2}}e^{-\frac{d^{E}\left(x,x'\right)^{2}}{8t}}$.
This combined with the subelliptic estimate \prettyref{eq:subelliptic estimate}
gives \prettyref{eq: heat kernel estimate}. For \prettyref{eq: heat kernel localization},
note that the second summand of \prettyref{eq:breakup integral} is
exponentially decaying $O\left(\exp(-\frac{\varrho_{1}^{2}}{16t})\right)$.
Next for $\varrho_{1}$ sufficiently small, $B_{\varrho_{1}}^{g^{E}}\left(x\right)\subset B_{\varrho}^{g^{TX}}\left(x\right)$.
Thus finite propagation and $\Delta_{g^{E},\mu}=\tilde{\Delta}_{g^{E},\mu}$
on $B_{\varrho_{1}}^{g^{E}}\left(x\right)$ give that the corresponding
first summands for $\Delta_{g^{E},\mu},\,\tilde{\Delta}_{g^{E},\mu}$
agree for $d^{E}\left(x,x'\right)\leq\varrho_{1}$. 
\end{proof}
We now prove the microlocal Weyl law in the equiregular case; below
let $\left[e^{-\hat{\Delta}_{g,\mu}}\right]_{\hat{\mu},x}$ we denote
the heat kernel of the Laplacian on the nilpotentization \prettyref{eq:nilpotent tangent space}
at a point $x$. Denote by
\begin{equation}
\int_{\left(E_{r}/E_{r-1}\right)}dy\,e^{iy.\xi}\left[e^{-\hat{\Delta}_{g,\mu}}\right]_{\hat{\mu}}\left(0;y\right)\label{eq:partial fourier transform}
\end{equation}
its partial Fourier transform in $\left(E_{r}/E_{r-1}\right)$ variables
and evaluation at $0$ in the remaining $\left(E_{1}\right)\oplus\left(E_{2}/E_{1}\right)\oplus\ldots\oplus\left(E_{r-1}/E_{r-2}\right)$
variables. We now have the following.
\begin{thm}
\label{lem:Microlocal weyl law} Let $\left(X,E,g^{E}\right)$ be
an equiregular sR manifold. For $B\in\Psi_{\textrm{cl}}^{0}\left(X,\Sigma\right)$
with $\sigma\left(B\right)=b_{0}$ we have
\begin{align}
E\left(B\right) & =\frac{1}{\left(2\pi\right)^{k_{r}^{E}}\mathcal{P}}\int_{E_{r-1}^{\perp}}d\hat{\mu}d\xi\,\left.b\left(x,\xi\right)\right|_{E_{r-1}^{\perp}}\left\{ \int_{\left(E_{r}/E_{r-1}\right)}dy\,e^{iy.\xi}\left[e^{-\hat{\Delta}_{g,\mu}}\right]_{\hat{\mu},x}\left(0;y\right)\right\} ,\quad\textrm{with }\label{eq: expectation formula}\\
\mathcal{P} & \coloneqq\int_{X}\left[e^{-\hat{\Delta}_{g,\mu}}\right]_{\hat{\mu}}\left(0,0\right)d\hat{\mu}.\label{eq: normalization}
\end{align}
Here the fiber $d\xi$-integral on the annihilator $E_{r-1}^{\perp}=\left(E_{r}/E_{r-1}\right)^{*}$
is with respect to the canonical volume elements \prettyref{eq: flag volume elements}.
\end{thm}
\begin{proof}
By a standard Tauberian argument, it suffices to prove that one has
an on diagonal asymptotic expansion for the heat kernel 
\begin{equation}
\left[Be^{-t\Delta_{g^{E},\mu}}\right]_{\mu}\left(x,x\right)=t^{-Q/2}\left[\sum_{j=0}^{N}b_{j}\left(x\right)t^{j}+O\left(t^{N+1}\right)\right]\label{eq:microlocal heat exp}
\end{equation}
that is uniform in $x\in X$ with leading term 
\begin{equation}
b_{0}=\frac{1}{\left(2\pi\right)^{k_{r}^{E}}}\int_{E_{r-1}^{\perp}}d\mu_{\textrm{Popp}}d\xi\,a\left(x,\xi\right)\left\{ \int_{\left(E_{r}/E_{r-1}\right)}dy\,e^{iy.\xi}\left[e^{-\hat{\Delta}_{g,\mu}}\right]_{\hat{\mu},x}\left(0;y\right)\right\} .\label{eq: leading term formula}
\end{equation}
First consider the case $B=1$. By \prettyref{lem: Localization lemma},
it suffices to demonstrate the expansion for the localized kernel
$\left[e^{-t\tilde{\Delta}_{g,\mu}}\right]_{\tilde{\mu}}\left(0,0\right)$
on $\mathbb{R}^{n}$. To this end, consider the rescaled sR-Laplacian
and measure
\begin{align*}
\tilde{\Delta}_{g^{E},\mu}^{\varepsilon} & \coloneqq\varepsilon^{2}\left(\delta_{\varepsilon}\right)_{*}\tilde{\Delta}_{g^{E},\mu}\\
\mu_{\varepsilon} & \coloneqq\varepsilon^{Q\left(x\right)}\left(\delta_{\varepsilon}\right)_{*}\tilde{\mu}
\end{align*}
using the privileged coordinate dilation from \prettyref{sec:sub-Riemannian-geometry}
.  It is now clear that the Schwartz kernels satisfy the relation
\begin{eqnarray}
\left[e^{-t\tilde{\Delta}_{g,\mu}^{\varepsilon}}\right]_{\mu_{\varepsilon}}\left(x',x\right) & = & \varepsilon^{Q\left(x\right)}\left[e^{-t\varepsilon^{2}\tilde{\Delta}_{g,\mu}}\right]_{\tilde{\mu}}\left(\delta_{\varepsilon}x',\delta_{\varepsilon}x\right).\label{eq: comparison with nilp heat kernels}
\end{eqnarray}
Rearranging and setting $x=x'=0,\,t=1$; gives 
\[
\varepsilon^{-Q\left(x\right)}\left[e^{-\tilde{\Delta}_{g,\mu}^{\varepsilon}}\right]_{\mu_{\varepsilon}}\left(0,0\right)=\left[e^{-\varepsilon^{2}\tilde{\Delta}_{g,\mu}}\right]_{\tilde{\mu}}\left(0,0\right)
\]
 and hence it suffices to compute the expansion of the left hand side
above as the dilation $\varepsilon\rightarrow0$. To this end, first
note that the rescaled Laplacian has an expansion 
\begin{equation}
\tilde{\Delta}_{g,\mu}^{\varepsilon}=\left(\sum_{j=0}^{N}\varepsilon^{j}\hat{\Delta}_{g,\mu}^{j}\right)+\varepsilon^{N+1}R_{N}^{\varepsilon},\quad\forall N\in\mathbb{N}\label{eq: privileged cord expansion Delta}
\end{equation}
Here each $\hat{\Delta}_{g^{E},\mu}^{j}$ is an $\varepsilon$-independent
second order differential operator of homogeneous $E-$order $j-2$.
While each $R_{\varepsilon}^{\left(N\right)}$ is an $\varepsilon$-dependent
second order differential operators on $\mathbb{R}^{n}$ of $E$-order
at least $N-1$. The coefficient functions of $\hat{\Delta}_{g^{E},\mu}^{\left(j\right)}$
are polynomials (of degree at most $j+2r$) while those of $R_{\varepsilon}^{\left(N\right)}$
are uniformly (in $\varepsilon$) $C^{\infty}$-bounded. The first
term is a scalar operator given in terms of the nilpotent approximation
\begin{equation}
\hat{\Delta}_{g^{E},\mu}^{0}=\Delta_{\hat{g}^{E},\hat{\mu};x}=\sum_{j=1}^{m}\left(\hat{U}_{j}^{\left(-1\right)}\right)^{2}.\label{eq: nilpotent Laplacian Euclidean}
\end{equation}
at the point $x$. This expansion \prettyref{eq: privileged cord expansion Delta}
along with the subelliptic estimates now gives 
\[
\left(\tilde{\Delta}_{g^{E},\mu}^{\varepsilon}-z\right)^{-1}-\left(\hat{\Delta}_{g^{E},\mu}^{0}-z\right)^{-1}=O_{H_{\textrm{loc}}^{s}\rightarrow H_{\textrm{loc}}^{s+1/r-2}}\left(\varepsilon\left|\textrm{Im}z\right|^{-2}\right),
\]
$\forall s\in\mathbb{R}$. More generally, we let $I_{j}\coloneqq\left\{ p=\left(p_{0},p_{1},\ldots\right)|p_{\alpha}\in\mathbb{N},\sum p_{\alpha}=j\right\} $
denote the set of partitions of the integer $j$ and define 
\begin{equation}
\mathtt{C}_{j}^{z}\coloneqq\sum_{p\in I_{j}}\left(\hat{\Delta}_{g^{E},\mu}^{0}-z\right)^{-1}\left[\prod_{\alpha}\hat{\Delta}_{g^{E},F,\mu}^{p_{\alpha}}\left(\hat{\Delta}_{g^{E},\mu}^{0}-z\right)^{-1}\right].\label{eq:resolvent exp. terms}
\end{equation}
Then by repeated applications of the subelliptic estimate we have
\[
\left(\tilde{\Delta}_{g^{E},\mu}^{\varepsilon}-z\right)^{-1}-\sum_{j=0}^{N}\varepsilon^{j}\mathtt{C}_{j}^{z}=O_{H_{\textrm{loc}}^{s}\rightarrow H_{\textrm{loc}}^{s+N\left(1/r-2\right)}}\left(\varepsilon^{N+1}\left|\textrm{Im}z\right|^{-2Nw_{n}^{E}-2}\right),
\]
$\forall s\in\mathbb{R}$. A similar expansion as \prettyref{eq: privileged cord expansion Delta}
for the operator $\left(\tilde{\Delta}_{g^{E},\mu}^{\varepsilon}+1\right)^{M}\left(\tilde{\Delta}_{g^{E},\mu}^{\varepsilon}-z\right)$,
$M\in\mathbb{N}$, also gives 
\begin{equation}
\left(\tilde{\Delta}_{g^{E},\mu}^{\varepsilon}+1\right)^{-M}\left(\tilde{\Delta}_{g^{E},\mu}^{\varepsilon}-z\right)^{-1}-\sum_{j=0}^{N}\varepsilon^{j}\mathtt{C}_{j,M}^{z}=O_{H_{\textrm{loc}}^{s}\rightarrow H_{\textrm{loc}}^{s+N\left(1/r-2\right)+\frac{M}{r}}}\left(\varepsilon^{N+1}\left|\textrm{Im}z\right|^{-2Nw_{n}^{E}-2}\right)\label{eq: regularized expansion}
\end{equation}
for operators $\mathtt{C}_{j,M}^{z}=O_{H_{\textrm{loc}}^{s}\rightarrow H_{\textrm{loc}}^{s+N\left(1/r-2\right)+\frac{M}{r}}}\left(\varepsilon^{N+1}\left|\textrm{Im}z\right|^{-2Nw_{n}^{E}-2}\right)$,
$j=0,\ldots,N$, with 
\[
\mathtt{C}_{0,M}^{z}=\left(\hat{\Delta}_{g^{E},\mu}^{0}+1\right)^{-M}\left(\hat{\Delta}_{g^{E},\mu}^{0}-z\right)^{-1}.
\]
For $M\gg0$ sufficiently large, Sobolev's inequality gives an expansion
for the corresponding Schwartz kernels of \prettyref{eq: regularized expansion}
in $C^{0}\left(\mathbb{R}^{n}\times\mathbb{R}^{n}\right)$. By plugging
the resolvent expansion into the Helffer-Sjöstrand formula and noting
$\mu_{\varepsilon}\sim\hat{\mu}+\sum_{j=1}^{\infty}\varepsilon^{j}\mu_{j}$
gives the diagonal heat kernel expansion
\begin{align*}
\left[e^{-t\Delta_{g^{E},\mu}}\right]_{\mu}\left(x,x\right) & =\sum_{j=0}^{\infty}e_{j}\left(x\right)t^{j/2}\quad\textrm{with }\\
e_{0}\left(x\right) & =\left[e^{-\Delta_{\hat{g}^{E},\hat{\mu}}}\right]_{\hat{\mu}}.
\end{align*}
Finally, to see that the expansion only involves only even powers
of $t^{1/2}$, note that the operators $\hat{\Delta}_{g^{E},\mu}^{j}$
in the expansion \prettyref{eq: privileged cord expansion Delta}
change sign by $\left(-1\right)^{j}$ under the rescaling $\delta_{-1}$.
The integral expression \prettyref{eq:resolvent exp. terms} corresponding
to $\mathtt{C}_{j}^{z}\left(0,0\right)$ then changes sign by $\left(-1\right)^{j}$
under this change of variables giving $\mathtt{C}_{j}^{z}\left(0,0\right)=0$
for $j$ odd. 

We now come to the expansion for general $B\in\Psi_{\textrm{cl}}^{0}$.
By a partition of unity and \prettyref{lem: Localization lemma},
we may assume $B$ to be supported in the privileged coordinate chart.
That is it has an integral representation $\left[B\right]_{\mu}\left(0,x\right)=\left[b^{W}\right]_{\mu}\left(0,x\right)=\frac{1}{\left(2\pi\right)^{n}}\int d\xi e^{-ix.\xi}b\left(\frac{1}{2}x,\xi\right)$
in the privileged coordinate chart with symbol $b$ being compactly
supported in $x$ . Next letting 
\begin{align}
\delta_{t^{1/2}}:T^{*}\mathbb{R}^{n} & \rightarrow T^{*}\mathbb{R}^{n},\nonumber \\
\delta_{t^{1/2}}\left(x,\xi\right) & \coloneqq\left(\delta_{t^{1/2}}x,\delta_{t^{-1/2}}\xi\right)\label{eq:privileged dilation phase space}
\end{align}
denote the induced symplectic dilation of phase space, we note
\begin{equation}
\left(\delta_{t^{-1/2}}\right)_{*}b^{W}\coloneqq\delta_{t^{1/2}}^{*}b^{W}\delta_{t^{-1/2}}^{*}=\left(\delta_{t^{1/2}}^{*}b\right)^{W}.\label{eq:dilation of Weyl symbol}
\end{equation}
Furthermore; the classical symbolic expansion for $b\in S_{\textrm{cl}}^{0}$
gives 
\begin{align}
\left(\delta_{t^{1/2}}^{*}b\right)\left(x;\xi\right)=b\left(\delta_{t^{1/2}}x;\delta_{t^{-1/2}}\xi\right) & =b\left(\delta_{t^{1/2}}x;t^{-w_{1}/2}\xi_{1},\ldots,t^{-w_{n}/2}\xi_{n}\right)\nonumber \\
 & =\underbrace{b_{0}\left(0;0,0,\ldots,0,\underbrace{\xi_{n-k_{r}+1},\ldots,\xi_{n}}_{\xi'=}\right)}_{\mathtt{b}_{0}\coloneqq}+O_{S_{\textrm{cl}}^{0}}\left(t\right)\label{eq: limit Weyl symbol}
\end{align}
We now finally compute
\begin{align}
 & \left[Be^{-t\tilde{\Delta}_{g,\mu}}\right]_{\mu}\left(0,0\right)\nonumber \\
= & t^{-Q/2}\left[\left(\delta_{t^{-1/2}}\right)_{*}\left(Be^{-t\tilde{\Delta}_{g,\mu}}\right)\right]_{\mu_{t^{1/2}}}\left(0,0\right)\nonumber \\
= & t^{-Q/2}\left[\left(\delta_{t^{-1/2}}\right)_{*}B\left(\delta_{t^{-1/2}}\right)_{*}e^{-t\tilde{\Delta}_{g,\mu}}\right]_{\mu_{t^{1/2}}}\left(0,0\right)\nonumber \\
= & t^{-Q/2}\left[1+o\left(1\right)\right]\left[\mathtt{b}_{0}^{W}e^{-\hat{\Delta}_{g^{E},\mu}^{0}}\right]_{\hat{\mu}}\left(0,0\right)\nonumber \\
= & t^{-Q/2}\left[1+o\left(1\right)\right]\int e^{-iy'\xi'}b_{0}\left(0;0,\xi'\right)e^{-\Delta_{\hat{g}^{E},\hat{\mu}}}\left(0,y';0\right)\label{eq: expression microlocal heat kernel}
\end{align}
following \prettyref{eq: comparison with nilp heat kernels}, \prettyref{eq:dilation of Weyl symbol}
and \prettyref{eq: limit Weyl symbol}. The theorem now follows on
noting leading term above to agree with \prettyref{eq: leading term formula}
in privileged coordinates.
\end{proof}
The rescaling arguments in the proof above are also analogous to those
in local index theory cf. \cite[Sec. 7]{Savale2017-Koszul} or \cite{Savale-Asmptotics,Savale-thesis2012}
and references therein. A local Weyl law for the semiclassical (magnetic)
analogue of sR Laplacian was also recently explored in \cite[Sec. 3]{Marinescu-Savale18}.

One still needs to identify the right hand side of \prettyref{eq: expectation formula}
with \prettyref{eq: micro-weyl} in the 4D quasi-contact case. First
note that \prettyref{eq: expectation formula}, \prettyref{eq: normalization}
are unchanged on replacing $\hat{\mu}$ by $\mu_{\textrm{Popp}}$.
A model for the nilpotentization is given in terms of the Darboux
coordinates of \prettyref{subsec:Normal-form-near-Sigma} this case
is $\hat{X}=\mathbb{R}^{4}$ with $\hat{E}=\mathbb{R}\left[\partial_{x_{0}},\partial_{x_{1}}+x_{2}\partial_{x_{3}},\partial_{x_{2}}-x_{1}\partial_{x_{3}}\right]$
being the span of the given (orthonormal) vector fields. The partial
Fourier transform in $x_{3}$ of the nilpotent Laplacian is computed
\begin{equation}
\mathcal{F}_{x_{3}}\hat{\Delta}_{g,\mu}^{0}\mathcal{F}_{x_{3}}^{-1}=-\partial_{x_{0}}^{2}\underbrace{-\left(\partial_{x_{1}}+ix_{2}\xi_{3}\right)^{2}-\left(\partial_{x_{2}}-ix_{1}\xi_{3}\right)^{2}}_{=\hat{\Delta}_{\xi_{3}}}\label{eq:partial tranform nilp.}
\end{equation}
while the Popp volume $\mu_{\textrm{Popp}}=\frac{1}{2}dx$ is Euclidean.
Mehler's formula (\cite{Berline-Getzler-Vergne} Sec. 4.2) now gives
the partial Fourier transform \prettyref{eq:partial fourier transform}
of the heat kernel to be
\begin{align}
\int dx_{3}\,e^{-ix_{3}.\xi_{3}}\left[e^{-\hat{\Delta}_{g,\mu_{\textrm{Popp}}}^{0}}\right]_{\mu_{\textrm{Popp}}}\left(0,0\right) & =\left[\exp-\left\{ \partial_{x_{0}}^{2}+\left(\partial_{x_{1}}+ix_{2}\xi_{3}\right)^{2}+\left(\partial_{x_{2}}-ix_{1}\xi_{3}\right)^{2}\right\} \right]_{\frac{1}{2}dx}\left(0,0\right)\nonumber \\
 & =\frac{1}{4\pi^{3/2}}\frac{\left|2\xi_{3}\right|}{\sinh\left|2\xi_{3}\right|},\quad\textrm{ while }\nonumber \\
f\left(\hat{\Delta}_{\xi_{3}}\right) & =\left\langle f\left(s\right),\frac{\left|\xi_{3}\right|}{\pi}\sum_{k=0}^{\infty}\delta\left(s-2\left|\xi_{3}\right|\left(2k+1\right)\right)\right\rangle \label{eq:Mehler formula}
\end{align}
cf. \cite[Sec. 7.]{Savale2017-Koszul}.

We then calculate 
\[
\mathcal{P}=\int d\mu_{\textrm{Popp}}\left[e^{-\hat{\Delta}_{g,\mu_{\textrm{Popp}}}^{0}}\right]_{\mu_{\textrm{Popp}}}\left(0,0\right)=\frac{1}{2\pi}\int d\mu_{\textrm{Popp}}\left\{ \int_{-\infty}^{\infty}d\xi_{3}\,\frac{1}{4\pi^{3/2}}\frac{\left|2\xi_{3}\right|}{\sinh\left|2\xi_{3}\right|}\right\} =\frac{1}{32\sqrt{\pi}}P\left(X\right)
\]
 while the leading term of \prettyref{eq: expectation formula} is
\begin{align*}
E\left(B\right)= & \frac{1}{2\pi\mathcal{P}}\int d\mu_{\textrm{Popp}}\left[b\left(x,a_{g}\left(x\right)\right)+b\left(x,-a_{g}\left(x\right)\right)\right]\left\{ \int_{0}^{\infty}d\xi_{3}\,\frac{1}{8\pi^{3/2}}\frac{\left|2\xi_{3}\right|}{\sinh\left|2\xi_{3}\right|}\right\} \\
= & \frac{1}{64\sqrt{\pi}\mathcal{P}}\int d\mu_{\textrm{Popp}}\left[b\left(x,a_{g}\left(x\right)\right)+b\left(x,-a_{g}\left(x\right)\right)\right]\\
= & \frac{1}{2}\int d\nu_{\textrm{Popp}}\left[b\left(x,a_{g}\left(x\right)\right)+b\left(x,-a_{g}\left(x\right)\right)\right]
\end{align*}
proving \prettyref{eq: micro-weyl}. 

The expression above may be rewritten 
\begin{align*}
E\left(B\right) & =\int\left.b\right|_{S^{*}\Sigma}\nu_{\textrm{Popp}}\\
 & =\int\pi_{S}^{*}\left.b\right|_{S^{*}\Sigma}\nu_{\textrm{Popp}}^{SNS^{*}\Sigma}
\end{align*}
in terms of the lifts of the normalized Popp volume to the unit sphere
of the characteristic variety and its blowup \prettyref{eq:lift Popp measure}.
We now generalize the above expression to prove a microlocal Weyl
law in $\Psi_{\textrm{cl}}^{0,0}\left(X,\Sigma\right)$, via the heat
kernel method here, agreeing with \prettyref{thm: Microlocal wave trace}.
\begin{thm}
For $X$ quasi-contact and $B\in\Psi_{\textrm{cl}}^{0,0}\left(X,\Sigma\right)$,
$\sigma^{H}\left(B\right)=b_{0}$ we have 
\begin{equation}
E\left(B\right)=\int\left.b\right|_{SNS^{*}\Sigma}\nu_{\textrm{Popp}}^{SNS^{*}\Sigma}.\label{eq:exotic microlocal weyl law}
\end{equation}
\end{thm}
\begin{proof}
Since $B$ is microlocally a classical pseudo-differential operator
\prettyref{def: Psi m123 on X} away from the characteristic variety,
where the microlocal Weyl measure \prettyref{eq: micro-weyl} vanishes,
it suffices to prove \prettyref{eq:exotic microlocal weyl law} for
$B$ micro-supported near $\Sigma$. In particular we may work in
the microlocal chart $C$ where \prettyref{eq:normal form} holds.
Note that in the quasi-contact case, the Darboux coordinates \prettyref{eq: quasi-contact normal form},
\prettyref{eq:characterictic normal form} used in the normal form
for $\Delta_{g^{E},\mu}$ and thereafter used in the definition \prettyref{def: Psi m123 on X}
of $\Psi_{\textrm{cl}}^{0,0}\left(X,\Sigma\right)$ are in particular
privileged. Furthermore, the privileged dilation of phase space \prettyref{eq:privileged dilation phase space}
extends to the blow up
\begin{align*}
\delta_{t^{1/2}}: & \left[T^{*}\mathbb{R}^{4},\Sigma_{0}\right]\rightarrow\left[T^{*}\mathbb{R}^{4},\Sigma_{0}\right]\\
\delta_{t^{1/2}}\coloneqq & \beta^{-1}\delta_{t^{1/2}}\beta
\end{align*}
and one has the relation 
\begin{align}
\left(\delta_{t^{-1/2}}\right)_{*}b^{H}=\delta_{t^{1/2}}^{*}b^{H}\delta_{t^{-1/2}}^{*} & =\left(\delta_{t^{1/2}}^{*}b\right)^{H},\label{eq: dilation of Hermite symbol}
\end{align}
 $\forall b\in S_{\textrm{cl}}^{0,0}$ similar to \prettyref{eq:dilation of Weyl symbol}.
Furthermore; the classical symbolic expansion for $b\in S_{\textrm{cl}}^{0,0}$
gives 
\begin{equation}
\delta_{t^{1/2}}^{*}b=\underbrace{b_{0}\left(0,d^{-2}\left(x_{1}^{2}+\hat{\xi}_{1}^{2}\right),0,0;d^{-1}\hat{\xi}_{0},0\right)}_{\mathtt{b}_{0}\coloneqq}+O_{S_{\textrm{cl}}^{0,0}}\left(t\right).\label{eq:limit of Hermite symbol}
\end{equation}
The equation \prettyref{eq:first Hamiltonian} and Duhamel's principle
give 
\begin{align}
U_{t}\coloneqq\left(\delta_{t^{-1/2}}\right)_{*}U=\delta_{t^{1/2}}^{*}U\delta_{t^{-1/2}}^{*} & =e^{i\frac{\pi}{4}f_{0}^{W}}+O_{L_{\textrm{loc}}^{2}\rightarrow H_{\textrm{loc}}^{-1}}\left(t\right)\label{eq:limit conjugating operator}
\end{align}
for the diagonalizing FIO in \prettyref{prop: normal form near pt},
while \prettyref{eq:first Hamiltonian flow}, \prettyref{eq:partial tranform nilp.}
gives
\begin{equation}
e^{i\frac{\pi}{4}f_{0}^{W}}e^{-\hat{\Delta}_{g^{E},\mu}^{0}}e^{-i\frac{\pi}{4}f_{0}^{W}}=\left[e^{-\xi_{0}^{2}+2\rho\xi_{3}\left(x_{1}^{2}+\hat{\xi}_{1}^{2}\right)}\right]^{H}.\label{eq:hermite transf nilp.}
\end{equation}
We may then compute
\begin{align*}
 & \left[U^{*}BUe^{-t\Delta_{g^{E},\mu}}\right]_{\mu}\left(0,0\right)\\
 & =t^{-5/2}\left[\left(\delta_{t^{-1/2}}\right)_{*}U^{*}BUe^{-t\Delta_{g^{E},\mu}}\right]_{\mu_{t^{1/2}}}\left(0,0\right)\\
 & =t^{-5/2}\left[U_{t}^{*}B_{t}U_{t}\left(\delta_{t^{-1/2}}\right)_{*}e^{-t\Delta_{g^{E},\mu}}\right]_{\mu_{t^{1/2}}}\left(0,0\right)\\
 & =t^{-5/2}\left[1+o\left(1\right)\right]\left[e^{-i\frac{\pi}{4}f_{0}^{W}}\mathtt{b}_{0}^{H}e^{i\frac{\pi}{4}f_{0}^{W}}e^{-\hat{\Delta}_{g^{E},\mu}^{0}}\right]_{dx}\left(0,0\right)\\
 & =\frac{t^{-5/2}}{2\pi}\left[1+o\left(1\right)\right]\left[e^{-i\frac{\pi}{4}f_{0}^{W}}\sum_{k=0}^{\infty}H_{k}^{*}.\right.\\
 & \qquad\left.\left..\left\{ \int d\xi_{0}e^{-\xi_{0}^{2}+2\rho\left(2k+1\right)}b_{0}\left(d_{k}^{-1}\hat{\xi}_{0}\right)\right\} H_{k}e^{i\frac{\pi}{4}f_{0}^{W}}\right]_{dx}\right|_{x_{1}=x_{2}=x_{3}=0}\\
 & =\frac{t^{-5/2}}{2\pi}\left[1+o\left(1\right)\right]\left[e^{-i\frac{\pi}{4}f_{0}^{W}}\sum_{k=0}^{\infty}H_{k}^{*}\right.\\
 & \left.\left..\left\{ \int_{-1}^{1}d\Xi_{0}\frac{\rho^{1/2}\left(2k+1\right)^{1/2}}{\left(1-\Xi_{0}^{2}\right)^{3/2}}e^{-\rho\left(2k+1\right)/\left(1-\Xi_{0}^{2}\right)}b_{0}\left(\Xi_{0}\right)\right\} H_{k}e^{i\frac{\pi}{4}f_{0}^{W}}\right]_{dx}\right|_{x_{1}=x_{2}=x_{3}=0}\\
 & =\frac{t^{-5/2}}{8\pi^{2}}\left[1+o\left(1\right)\right]\int d\Xi_{0}\int_{0}^{\infty}d\xi_{3}b_{0}\left(\Xi_{0}\right)\left(1-\Xi_{0}^{2}\right)^{-1}\left[\hat{\Delta}_{\hat{\rho}\xi_{3}/\left(1-\Xi_{0}^{2}\right)}^{1/2}e^{-\hat{\Delta}_{\hat{\rho}\xi_{3}/\left(1-\Xi_{0}^{2}\right)}}\right]_{dx}\left(0,0\right)\\
 & =\frac{t^{-5/2}}{8\pi^{2}}\left[1+o\left(1\right)\right]\left[\int_{-1}^{1}d\Xi_{0}b_{0}\left(\Xi_{0}\right)\int_{0}^{\infty}d\xi_{3}\sum_{k=0}^{\infty}\hat{\rho}\xi_{3}\frac{\rho^{1/2}\left(2k+1\right)^{1/2}}{\left(1-\Xi_{0}^{2}\right)^{3/2}}e^{-\rho\left(2k+1\right)/\left(1-\Xi_{0}^{2}\right)}\right]\\
 & =\frac{t^{-5/2}}{8\pi^{2}}\left[1+o\left(1\right)\right]\left[\left(\int_{0}^{\infty}r^{3/2}e^{-r}dr\right)\left(\sum_{k=0}^{\infty}\frac{1}{\left(2k+1\right)^{2}}\right)\int_{-1}^{1}d\Xi_{0}\left(1-\Xi_{0}^{2}\right)b_{0}\left(\Xi_{0}\right)\right]\\
 & =\frac{3t^{-5/2}}{256\sqrt{\pi}}\left[1+o\left(1\right)\right]\int_{\pi^{-1}\left(x\right)}\left.b\right|_{SNS^{*}\Sigma}\mu_{\textrm{Popp}}^{SNS^{*}\Sigma}
\end{align*}
following \prettyref{eq:Mehler formula}, \prettyref{eq: dilation of Hermite symbol},
\prettyref{eq:limit of Hermite symbol} and \prettyref{eq:limit conjugating operator}
and on identifying the right hand side of \prettyref{eq:exotic microlocal weyl law}
here in terms of privileged coordinates. 
\end{proof}

\subsection{\label{subsec:Variance-estimate}Variance estimate}

We now prove the variance estimate \prettyref{eq:variance}, specializing
again to the 4D quasi-contact case. By replacing $B\in\Psi_{\textrm{cl}}^{0}\left(X\right)$
by the operator $B-E\left(B\right)\in\Psi_{\textrm{cl}}^{0}\left(X\right)$
with the same variance, we may assume that $E\left(B\right)=0$. Furthermore
we have
\begin{align}
V\left(B_{1}\right) & \leq E\left(B_{1}^{*}B_{1}\right)\label{eq:variance vs exp.}\\
V\left(B_{1}+B_{2}\right) & \leq2\left[V\left(B_{1}\right)+V\left(B_{2}\right)\right]\label{eq:youngs ineq.}\\
V\left(B_{1}\right)=0 & \implies V\left(B_{1}+B_{2}\right)=V\left(B_{2}\right),\label{eq:variance pres by var zero}
\end{align}
$\forall B_{1},B_{2}\in\Psi_{\textrm{cl}}^{0}\left(X\right)$, (see
\cite[Lemma 4.1]{Colin-de-Verdiere-Hillairet-TrelatI}) and 
\begin{equation}
\left.\sigma\left(B_{1}\right)\right|_{\Sigma}=0\implies E\left(B_{1}^{*}B_{1}\right)=0\label{eq:vanishing symbol implies vanishing exp.}
\end{equation}
 by \prettyref{lem:Microlocal weyl law} giving
\begin{equation}
\left.\sigma\left(B_{1}\right)\right|_{\Sigma}=\left.\sigma\left(B_{2}\right)\right|_{\Sigma}\implies V\left(B_{1}\right)=V\left(B_{2}\right).\label{eq:variance depends on Sigma}
\end{equation}
From \prettyref{eq:youngs ineq.}, \prettyref{eq:variance depends on Sigma}
and a partition of unity it now suffices to prove \prettyref{eq:variance}
for $B\in\Psi_{\textrm{cl}}^{0}\left(X\right)$ micro-supported in
a conic neighborhood of $p\in\Sigma$ with $E\left(B\right)=0$. By
a Taylor expansion in the coordinates $\left(x_{0},x_{1},x_{2},\hat{x}_{3};\xi_{0},\xi_{1},\xi_{2},\xi_{3}\right)$
of the normal form \prettyref{eq:normal form}, we may write $B=B_{0}+B_{1}$
where $B_{0},B_{1}\in\Psi_{\textrm{cl}}^{0}\left(X\right)$ with $\left.\sigma\left(B_{1}\right)\right|_{\Sigma}=0$
and $B_{0}=\left[b_{0}\left(x_{0},x_{2},\hat{x}_{3},\xi_{2},\xi_{3}\right)\right]^{W}\in\Psi_{\textrm{cl}}^{0}\left(X\right)$.
Furthermore, the homogeneous symbol $\sigma\left(B_{0}\right)=\pi^{*}b_{0}$,
$b_{0}\in C^{\infty}\left(X\right)$, is the pull back of a function
from the base. From \prettyref{eq:variance vs exp.} and \prettyref{eq:vanishing symbol implies vanishing exp.}
we have $V\left(B_{1}\right)=0$ and it suffices to show $V\left(B_{0}\right)=0$
by \prettyref{eq:youngs ineq.}. Clearly $E\left(B_{0}\right)=E\left(B\right)=0$
by \prettyref{eq:variance depends on Sigma} and $\left[B_{0},\Omega\right]=0$
\prettyref{eq:inclusion of inv. classical symb.} showing 
\[
B_{0}\in\Psi_{\textrm{inv},\textrm{cl}}^{0}\left(X\right)\subset\Psi_{\textrm{cl}}^{0,0}\left(X;\Sigma\right).
\]
Next $V\left(B_{0}\right)=V\left(e^{-it\sqrt{\Delta_{g^{E},\mu}}}B_{0}e^{it\sqrt{\Delta_{g^{E},\mu}}}\right)$
by definition, while Egorov's theorem \prettyref{thm:Egororovb thm}
gives 
\[
D\coloneqq e^{-it\sqrt{\Delta_{g^{E},\mu}}}B_{0}e^{it\sqrt{\Delta_{g^{E},\mu}}}-\textrm{Op}^{H}\left[\underbrace{\left(e^{-tH_{d}}\right)^{*}\sigma_{0,0}\left(B_{0}\right)}_{\eqqcolon b_{t}}\right]\in\Psi_{\textrm{cl}}^{-1,1}\left(X,\Sigma\right).
\]
In particular the difference above $D:L^{2}\left(X\right)\rightarrow H^{1,-1}\left(X,\Sigma\right)\hookrightarrow H^{\frac{1}{2},0}\left(X,\Sigma\right)\hookrightarrow L^{2}\left(X\right)$
\prettyref{eq:Sobolev inclusions-1} being a compact operator, its
variance vanishes $V\left(D\right)=0$ \cite[Lemma 4.2]{Colin-de-Verdiere-Hillairet-TrelatI}.
Hence $V\left(B_{0}\right)=V\left(\textrm{Op}^{H}\left[b_{t}\right]\right)$,
$\forall t$, and also
\begin{align}
V\left(B_{0}\right) & =V\left(\underbrace{\textrm{Op}^{H}\left[\underbrace{\frac{1}{T}\int_{0}^{T}dt\,b_{t}}_{\eqqcolon b_{T}}\right]}_{\eqqcolon\bar{B}_{T}}\right)\nonumber \\
 & \leq E\left(\bar{B}_{T}^{*}\bar{B}_{T}\right)=\int\left|\left.b_{T}\right|_{SNS^{*}\Sigma}\right|^{2}\nu_{\textrm{Popp}}^{SNS^{*}\Sigma},\label{eq:last step variance estimate}
\end{align}
$\forall T>0$, by \prettyref{eq:exotic microlocal weyl law} and
\prettyref{eq:variance vs exp.}. Finally $\left.b_{T}\right|_{SNS^{*}\Sigma}\rightarrow0$
in $L^{2}\left(SNS^{*}\Sigma;\nu_{\textrm{Popp}}^{SNS^{*}\Sigma}\right)$
as $T\rightarrow\infty$ under the ergodicity assumption on $\hat{Z}$
by the von Neumann mean ergodic theorem to prove the first part of
\prettyref{thm:QE theorem}. 

Next to prove the second part of \prettyref{thm:QE theorem}, suppose
that $L^{E}$ is ergodic and $L_{Z}\mu_{\textrm{Popp}}=0$. From the
equivalent conditions \prettyref{eq:equivalent conditions} and the
computation \prettyref{eq:Hvfield restriction to boundary} it follows
that the function $\Xi_{0}$ is now preserved under the $\hat{Z}$-flow.
Furthermore, the level sets $\left(SNS^{*}\Sigma/S^{1}\right)_{c}\coloneqq\left\{ \Xi_{0}=c\right\} $,
$c\in\left(-1,1\right)$, are copies of $X$ with the $\hat{Z}$-flow
$\left.\left(\hat{Z}b_{t}\right)\right|_{\left(SNS^{*}\Sigma/S^{1}\right)_{c}}=\left(\pi_{S}\circ\pi\right)^{*}cZb_{t,c}$
being simply lifted from the base
\[
\left.b_{t}\right|_{\left(SNS^{*}\Sigma\right)_{c}}=\left(\pi_{S}\circ\pi\right)^{*}\left(e^{-tcZ}\right)^{*}b_{0}.
\]
Setting, $\left(b_{0}\right)_{T}\coloneqq\frac{1}{T}\int_{0}^{T}dt\,\left(e^{-tZ}\right)^{*}b_{0}$,
we may then compute 
\begin{align}
\int\left|\left.b_{T}\right|_{SNS^{*}\Sigma}\right|^{2}\nu_{\textrm{Popp}}^{SNS^{*}\Sigma} & =\int_{-1}^{1}\left(1-c^{2}\right)dc\int\nu_{\textrm{Popp}}\left|\left(b_{0}\right)_{cT}\right|^{2}\nonumber \\
 & =\frac{1}{T}\int_{-T}^{T}\left(1-\frac{c'^{2}}{T^{2}}\right)dc'\int\nu_{\textrm{Popp}}\left|\left(b_{0}\right)_{c'}\right|^{2}.\label{eq: L2 average symbol}
\end{align}
As noted before the ergodicity assumption on $L^{E}$ is equivalent
to the ergodicity of the vector field $Z\in C^{\infty}\left(L^{E}\right)$.
Since $E\left(B_{0}\right)=\int b_{0}\nu_{\textrm{Popp}}=0$, the
von Neumann mean ergodic theorem applied to $e^{tZ}$ gives $\int\nu_{\textrm{Popp}}\left|\left(b_{0}\right)_{T}\right|^{2}$
and hence \prettyref{eq: L2 average symbol} converges to zero as
$T\rightarrow\infty$. 

We finally remark that the ergodicity of $L^{E}$ alone , which is
a topological condition, does not suffice to prove the variance estimate
and hence quantum ergodicity in the general volume preserving case.
In this case, following the computation \prettyref{eq:Hvfield restriction to boundary}
and \prettyref{eq:log derivative}, functions of the form $f\left(\left(1-\Xi_{0}^{2}\right)/\hat{\rho}_{Z}\right)$
are seen to be invariant under the $\hat{Z}$ flow. The last line
of \prettyref{eq:last step variance estimate} now converging to the
projection of $b_{0}$ onto the $\hat{Z}$ invariant functions, such
a projection $\int\nu_{\textrm{Popp}}^{SNS^{*}\Sigma}f\left(\left(1-\Xi_{0}^{2}\right)/\hat{\rho}_{Z}\right)\left(\beta^{*}\pi^{*}b_{0}\right)$
of the symbol $b_{0}$ , $\int b_{0}\mu_{\textrm{Popp}}=0$, might
be non-zero unless $\hat{\rho}_{Z}=1$.

\bibliographystyle{siam}
\bibliography{biblio}

\end{document}